\documentclass[11pt, twoside]{article}
\usepackage[a4paper,width=150mm,top=25mm,bottom=30mm]{geometry}
\usepackage[all,cmtip]{xy}
\usepackage{comment}
\usepackage[pdftex]{pict2e}
\usepackage{amsmath,amssymb,amscd,amsthm}
\usepackage{a4wide,pstricks}
\usepackage{pdfsync}
\usepackage{fancyhdr}
\usepackage{makebox, bigstrut}
\usepackage[T1]{fontenc}
\usepackage[utf8]{inputenc} 
\usepackage[english]{babel} 
\usepackage{setspace}
\usepackage{bbm}
\usepackage{leftindex}
\usepackage{tasks}
\usepackage{dsfont} 
\usepackage{amsthm}
\usepackage{amsmath}
\usepackage{amssymb}
\usepackage{mathrsfs}
\usepackage{mathtools}
\usepackage{enumitem}
\usepackage{hyperref}
\usepackage{amssymb}
\usepackage{calligra}
\usepackage{ytableau}
\usepackage{marginnote}
\usepackage{cleveref}
\usepackage{babel}
\usepackage{bm}
\usepackage{tikz-cd}
\usepackage{faktor}
\usepackage{stmaryrd}
\renewcommand{\phi}{\varphi}

\let\emptyset\varnothing

\newcommand{\C}{\mathbb{C}}

\newcommand{\N}{\mathbb{N}}

\newcommand{\Q}{{\mathbb Q}}

\newcommand{\T}{\mathbb{T}}
\newcommand{\Z}{\mathbb{Z}}
\newcommand{\F}{\mathbb{F}}

\xyoption{all}
\input{xypic}

\renewcommand{\phi}{\varphi}

\let\emptyset\varnothing

\newcommand{\Irm}{{\rm I}}
\newcommand{\Rrm}{{\rm R}}
\newcommand{\pr}{{\rm pr}}
\newcommand{\car}{\mathfrak{car}}

\usepackage{lmodern}
\xyoption{all}
\input{xypic}

\theoremstyle{plain}
\numberwithin{equation}{subsection}
\crefformat{section}{\S#2#1#3} 
\crefformat{subsection}{\S#2#1#3}
\crefformat{subsubsection}{\S#2#1#3}

\makeindex

\newtheorem{teorema}{Theorem}[subsection]
\newtheorem{prop}[teorema]{Proposition}

\newtheorem{conjecture}[teorema]{Conjecture}
\newtheorem{lemma}[teorema]{Lemma}

\theoremstyle{remark}
\newtheorem{oss}[teorema]{Remark}
\newtheorem{remark}[teorema]{Remark}
\newtheorem{esempio}[teorema]{Example}

\newtheorem{assum}[teorema]{Assumption}

\theoremstyle{definition}
\newtheorem{definizione}[teorema]{Definition}

\newcounter{margin}
{\end{itshape}  \bigskip}

\usepackage{tikz}
\DeclareMathOperator{\lcm}{lcm}
\DeclareMathOperator{\Perv}{Perv}
\DeclareMathOperator{\X}{{\bf X}}
\DeclareMathOperator{\Hom}{Hom}

\DeclareMathOperator{\ord}{ord}
\DeclareMathOperator{\Res}{Res}
\DeclareMathOperator{\Ker}{Ker}

\DeclareMathOperator{\Supp}{supp}
\DeclareMathOperator{\Gl}{GL}
\DeclareMathOperator{\End}{End}
\DeclareMathOperator{\Imm}{Im}

\DeclareMathOperator{\Aut}{Aut}

\DeclareMathOperator{\tr}{Tr}

\DeclareMathOperator{\spec}{Spec}
\DeclareMathOperator{\Sl}{SL}

\DeclareMathOperator{\Rep}{Rep}

\DeclareMathOperator{\Ind}{Ind}

\DeclareMathOperator{\Coeff}{Coeff}

\DeclareMathOperator{\PGl}{PGL}

\DeclareMathOperator{\Gr}{Gr}
\DeclareMathOperator{\IC}{IC}

\begin{document}

\title{${\rm PGL}_n(\C)$-character stacks and Langlands duality over finite fields}

\author{ Emmanuel Letellier
\\ {\it Universit\'e Paris Cit\'e, IMJ-PRG CNRS UMR 7586}
\\{\tt emmanuel.letellier@imj-prg.fr } \and Tommaso Scognamiglio \\ {\it
  University of Bologna, IndAm} \\ {\tt
  scognamiglio@altamatematica.it} }
\maketitle
\pagestyle{myheadings}

\begin{abstract}
In this paper we study the mixed Poincar\'e polynomials of generic $\PGl_n(\C)$-character stacks with coefficients in some local systems arising from the conjugacy classes of $\PGl_n(\C)$ which have non-connected stabilizers. We give a conjectural formula that we prove to be true under the Euler specialization. We then prove that these conjectured formulas  interpolate the  structure coefficients of the two following based rings:

$$
\left(\mathcal{C}(\PGl_n(\F_q)),Loc(\PGl_n),*\right),\hspace{1cm}\left(\mathcal{C}(\Sl_n(\F_q)), CS(\Sl_n),\cdot\right)
$$
where for a group $H$, $\mathcal{C}(H)$ denotes the space of complex valued class functions on $H$, $Loc(\PGl_n)$ denotes the basis of characteristic functions of intermediate extensions of equivariant local systems on conjugacy classes of $\PGl_n$ and $CS(\Sl_n)$ the basis of characteristic functions of Lusztig's character-sheaves on $\Sl_n$. Our result reminds us of a non-abelian Fourier transform.
\end{abstract}

\tableofcontents

\section{Introduction}

For a compact and connected Riemann surface $X$ and a connected reductive group $G$, the study of the moduli spaces of $G$-local systems (i.e. character varieties and character stacks) play a central role in many areas in mathematics. Via non-abelian Hodge theory these spaces are related to the corresponding moduli spaces of $G$-Higgs bundles over $X$. 
Here, we are interested in the case where $X$ is a punctured compact Riemann surface (i.e. of the form $X=Z \setminus D$, where $D$ is a finite subset of a compact Riemann surface $Z$). We study local systems on $X$ that have local monodromies around the points of $D$ in (the Zariski closure) of some fixed conjugacy classes of $G$. 
\bigskip

Most of the results reported in the literature concern the case where $G=\Gl_n$ or $\Sl_n$ with a genericity assumption on local monodromies at punctures. The study of the cohomology of these spaces has been initiated by Hitchin \cite{Hitchin} and pursued by many authors (see for instance \cite{Boden} \cite{Garcia}) in some more complicated cases using geometrical methods. Hausel and Rodriguez-Villegas introduced a new method by counting points of these spaces over finite fields to get information on the cohomology (using the Weil conjectures proved by Deligne). They first treated the case of $\Gl_n$-character varieties with $D=\{pt\}$ and central local monodromy around that point \cite{HRV}. They obtained a conjectural formula for the mixed Poincar\'e polynomial based on a natural deformation of their counting point formulas. Their work has been extended by Hausel-Letellier-Villegas to an arbitrary number of punctures with semisimple local monodromies \cite{HA}. They obtained a conjectural formula for the mixed Hodge polynomials in terms of Macdonald polynomials. These conjectures have been verified by Mellit and Schiffman for the usual Poincar\'e polynomial (ignoring the weight filtration) \cite{Mellit}\cite{schiffmann}. They obtained the Poincar\'e polynomial of character varieties/space of stable Higgs bundles by counting $\Gl_n$-Higgs bundles over finite fields (using the fact that the space of stable Higgs bundles is cohomologically pure). These formulas has been further extended by Letellier \cite{L} and Ballandras \cite{ballandras} to the case of arbitrary generic local monodromies (not necessarily semisimple). Finally, Scognamiglio \cite{scognamiglio2} generalized Hausel-Letellier-Villegas conjectural formulas of the mixed Hodge polynomials to the non-generic case and he proved his conjecture under the Euler specialization (that gives the number of points over finite fields).
\bigskip

In the case $G=\Sl_n$ (in the one-puncture case with central local monodoromy), Mereb \cite{Mereb} computed the number of points of $\Sl_n$-character varieties over finite fields using a similar approach to that of Hausel-Villegas in the $\Gl_n$-case, i.e. by using the character table of $\Sl_n$ over a finite field to count points. However due to the complexity of the character table of $\Sl_n$ (compared to that of $\Gl_n$), it is still not known how to deform conjecturally his formula to get a reasonable conjectural formula for the mixed Hodge polynomial. More recently, Chaudouard \cite{Chaudouard} obtained a counting formula for the number of $\Sl_n$-Higgs bundles over finite fields (in the non-parabolic case).
\bigskip

Giannini, Kamgarpour, Nam and Whitbread \cite{KNWG} study the case of an arbitrary connected reductive group $G$ and they gave a formula for the counting of points over finite fields of $G$-character varieties with local monodromies in generic semisimple regular conjugacy classes. However it is not clear how to guess a natural conjectural formula for the mixed Poincar\'e polynomial from their formula. The problem with groups other than $\Gl_n$ is that there is no combinatorial description of the character table of the finite group $G(\mathbb{F}_q)$ and it is thus not clear what should replace the Macdonald symmetric functions for instance. There are Macdonald symmetric functions in type $B$ and $C$ but their relationship with the character table of $G(\mathbb{F}_q)$ is not established and certainly not straightforward. 
\bigskip

The counting of points of $G$-character varieties over finite fields use the character table of $G(\mathbb{F}_q)$. In fact, the counting of points coincides naturally with the structure coefficients of the center of the group algebra of $G(\mathbb{F}_q)$, or equivalently with the structure coefficient of the space $\mathcal{C}(G(\mathbb{F}_q))$ of complex valued class functions on $G(\mathbb{F}_q)$ equipped with the convolution product and the basis of the characteristic functions of the conjugacy classes. When $G=\Gl_n$ it was observed \cite{HA1}\cite{L} that the "pure part" of the conjectured formula of the mixed Hodge polynomial of the generic character varieties coincides with the structure coefficients of the character ring of $G(\mathbb{F}_q)$, namely, the space $\mathcal{C}(G(\mathbb{F}_q))$ equipped with tensor product and the basis the irreducible characters. This uses a correspondence between irreducible characters and conjugacy classes of $\Gl_n(\mathbb{F}_q)$.
\bigskip

The relationship between the representations of $G(\mathbb{F}_q)$ and the geometry of character varieties is  relevant when the genus of the Riemann surface is zero (otherwise we can still make sense of it, but need to introduce some functions on $G(\mathbb{F}_q)$ attached to the genus which do not seem to be of a great interest from the perspective of the representation theory of finite Lie groups, see for instance \cite{letellier2}).
\bigskip

The problem of finding a criterion for the non-vanishing of the structure coefficients of the character ring of $G(\mathbb{F}_q)$ is a very difficult problem. Already in the case of the symmetric group $S_n$ (where these structure coefficients are called the \emph{Kronecker coefficients}), finding a combinatorial criterion on triples of partitions for the vanishing of the corresponding Kronecker coefficients is still open and is known as one of the most challenging problem in algebraic combinatorics which goes back to Murnaghan (1938). The connection between the structure coefficients of the character ring of $\Gl_n(\mathbb{F}_q)$ and the geometry of generic character varieties (and their additive counterparts which are certain quiver varieties) allowed to connect the non-vanishing problem with the Deligne-Simpson problem (see for instance \cite{letellier2}\cite{letellier3}\cite{letellier4}\cite{scognamiglio1}).

\bigskip

In this paper we are interested in generic $\PGl_n$-character stacks with local monodromies in any conjugacy classes. The conjugacy classes of $\PGl_n$ which have a disconnected stabilizer carry non-trivial irreducible $\PGl_n$-equivariant local systems (in $\Gl_n$, all stabilizers are connected). We propose a general conjecture for the mixed Hodge series of the character stacks with coefficients in local systems arising from the conjugacy classes. We also show how it interpolates convolution products of the intermediate extensions of local systems on $\PGl_n(\overline{\F}_q)$-conjugacy classes with the tensor products of the corresponding character-sheaves on $\Sl_n(\F_q)$. This generalizes the $\Gl_n$-picture as in the $\Gl_n$-case the character-sheaves give the irreducible characters by taking the Frobenius trace (this is not true for other groups like $\Sl_n$). Our main picture looks like a multiplicative version of the Fourier inversion formula (see \S \ref{Fourinv}).  
\bigskip

We prove our conjecture under the Euler specialization and in some explicit examples.

\bigskip

We now describe the main results of this paper in more details. In the following, $K$ is an algebraically closed field which is either $\C$ or $\overline{\F}_q$ and $\kappa$ is a field with $\kappa=\C$ if $K=\C$ and $\kappa=\overline{\Q}_\ell$ if $K=\overline{\F}_q$ where $\ell \nmid q$. 
\bigskip

Put $\PGl_n\coloneqq \PGl_n(K)$.

\bigskip

We fix a $k$-tuple $\bm{\mathcal{C}}=(\mathcal{C}_1,\dots,\mathcal{C}_k)$ of conjugacy classes of $\PGl_n$. If $K=\overline{\F}_q$, we assume that the conjugacy classes $\mathcal{C}_1,\dots,\mathcal{C}_k$ are \emph{split}, i.e. that the eigenvalues are in $\F_q^{\times}$.

Fix $g \in \N$. We consider the character stack \begin{equation}\label{definition-intro}\mathcal{M}_{\overline{\bm{\mathcal{C}
}}}\coloneqq\left[X_{\overline{\bm {\mathcal{C}}}}/\PGl_n\right] .\end{equation}
where$$ X_{\overline{\bm {\mathcal{C}}}}\coloneqq\left\{(A_1,B_1,\dots,A_g,B_g,x_1,\dots,x_k) \in \PGl_n^{2g} \times \overline{\mathcal{C}}_1 \times \cdots \times \overline{\mathcal{C}}_k \ \left| \ \prod_{i=1}^g[A_i,B_i] x_1 \cdots x_k=1\right\}\right.$$
If $K=\C$, for a Riemann surface $X$ of genus $g$, a subset $D \subseteq X$ with $D=\{z_1,\dots,z_k\}$, we can identify $\mathcal{M}_{\overline{\bm {\mathcal{C}}}}$ with the moduli stack of $\PGl_n$-local systems on $X\setminus D$ such that the local monodromy around each $z_i$ belongs to the Zariski closure $\overline{\mathcal{C}}_i$. 

Such moduli spaces are also related to certain moduli spaces of parabolic $\PGl_n(\C)$-Higgs bundles, through the non-abelian Hodge correspondence for $(X,D)$, introduced by Simpson \cite{Simpson}.
\bigskip

We assume that the $k$-tuple  $\bm{\mathcal{C}}$ is \emph{generic} (see \S \ref{section-geometry} for the definition).
\bigskip

We show that the stack $\mathcal{M}_{\overline{\bm{\mathcal{C}}}}$, if non-empty, is an equidimensional  Deligne-Mumford stack of dimension $$\displaystyle (2g-2)n^2+2 +\sum_{i=1}^k \dim(C_i) .$$ Moreover, the  substack $$\mathcal{M}_{\bm{\mathcal{C}}}\coloneqq X_{\bm{\mathcal{C}}}/\PGl_n,$$
where $X_{\bm{\mathcal{C}}}\subset X_{\overline{\bm{\mathcal{C}}}}$ is the subset of $(2g+k)$-tuples $(A_1,B_1,\dots,A_g,B_g,x_1,\dots,x_k)$ with $x_i\in \mathcal{C}_i$, is an open substack of $\mathcal{M}_{\overline{\bm{\mathcal{C}}}}$ smooth and everywhere dense.

For more details, see Proposition \ref{proposition-DM}.

\subsection{Mixed Poincaré series of local systems on $\mathrm{PGL}_n$-character stacks}

One of the main aim of this paper is the computation of the intersection cohomology on $\mathcal{M}_{\overline{\bm{\mathcal{C}}}}$ with coefficients in certain local systems on $\mathcal{M}_{\bm{\mathcal{C}}}$ ($\ell$-adic local systems if $K=\overline{\F}_q$). 
\bigskip

Recall that the irreducible $\PGl_n$-equivariant local systems on a conjugacy class $\mathcal{C}$ of $\PGl_n$ are parametrized by the irreducible characters of the group $A(\mathcal{C})$ of the connected components of the stabilizer of $\mathcal{C}$. For such an irreducible character $\chi$ denote by $\mathcal{L}^\mathcal{C}_\chi$ the corresponding local system on $\mathcal{C}$.
\bigskip

Put $A(\bm{\mathcal{C}})\coloneqq A(\mathcal{C}_1) \times \cdots \times A(\mathcal{C}_k)$. For each $\chi=(\chi_1,\dots,\chi_k) \in \widehat{A(\bm{\mathcal{C}})}$, with $\chi_i \in \widehat{A(\mathcal{C}_i)}$, the local system $$\mathcal{L}^{\mathcal{C}_1}_{\chi_1} \boxtimes \cdots \boxtimes \mathcal{L}^{\mathcal{C}_k}_{\chi_k}$$
on $\mathcal{C}_1\times\cdots\times\mathcal{C}_k$ being $\PGl_n$-equivariant for the diagonal action, defines a unique local system $\mathcal{E}_\chi$ on the open substack $\mathcal{M}_{\bm{\mathcal{C}}}$.

We are interested in the intersection cohomology $IH^{\bullet}_c(\mathcal{M}_{\overline{\bm{\mathcal{C}}}},\mathcal{E}_{\chi})$, i.e. the hypercohomology $\mathbb{H}^{\bullet}_c(\mathcal{M}_{\overline{\bm{\mathcal{C}}}},\IC^{\bullet}_{\mathcal{M}_{\overline{\bm{\mathcal{C}}}},\mathcal{E}_{\chi}})$, where $\IC^{\bullet}_{\mathcal{M}_{\overline{\bm{\mathcal{C}}}},\mathcal{E}_{\chi}}$ is the intersection cohomology complex defined from the local system $\mathcal{E}_{\chi}$. 

\bigskip

Recall that each cohomology group $IH^{i}_c(\mathcal{M}_{\overline{\bm{\mathcal{C}}}},\mathcal{E}_{\chi})$ is equipped with a weight filtration (increasing) $W^i_\bullet$ from which we define the mixed Poincar\'e series

$$
IH_c\big(\mathcal{M}_{\overline{\bm{\mathcal{C}}}},\mathcal{E}_{\chi};q,t\big)=\sum_{i,r}{\rm dim}\left(W^i_r/W^i_{r-1}\right)q^{r/2}t^i.
$$

One of the main result of this paper is a combinatorial (conjectural) formula for the above mixed Poincaré series (see Conjecture \ref{IH-conjecture-thm}).

\begin{conjecture}
\label{IH-conjecture-thm-intro}
For any $\chi \in \widehat{A(\bm{\mathcal{C}})}$, we have

\begin{equation}\label{IHconjecture-intro}
IH_c(\mathcal{M}_{\overline{\bm{\mathcal{C}}}},\mathcal{E}_\chi;q,t)=\dfrac{(qt^2)^{\frac{{\rm dim}\mathcal{M}_{\overline{\bm{\mathcal{C}}}}}{2}}\iota(\bm{\mathcal{C}})}{(qt^2+t)^{2g}|A(\bm{\mathcal{C}})|}\sum_{r \in R_{d_1,\dots,d_k}}\Delta_{r}^{s_{\chi}}\,\mathbb{H}_{g,\bm \omega_r}\left(-t\sqrt{q},\dfrac{1}{\sqrt{q}}\right)
\end{equation}
where $\iota(\bm{\mathcal{C}})$ is the number of irreducible components of $\mathcal{M}_{\overline{\bm{\mathcal{C}}}}$, $d_i=|A(\mathcal{C}_i)|$, and $\bm{\omega}_r$ is a combinatorial object encoding the Jordan form of  conjugacy classes $C_1,\dots, C_k$ of $\Gl_n$ whose images are $\mathcal{C}_1,\dots,\mathcal{C}_k$ under the projection $\Gl_n\rightarrow\PGl_n$.
\end{conjecture}

\bigskip
For detailed definitions and notation of the symbols appearing in Formula (\ref{IHconjecture-intro}), see \cref{localsystemsPGLncohomology}. The rational functions $\mathbb{H}_{g,\omega_r}(z,w)$ appearing in Formula (\ref{IHconjecture-intro}) are the ones introduced by Hausel, Letellier and Rodriguez-Villegas \cite{HA} in the semisimple case and by Letellier in general \cite{L} to compute the cohomology of generic $\Gl_n$-character stacks.

By Theorem \ref{redGL}, the above conjecture is a consequence of a conjectural formula for some twisted mixed Poincar\'e series of some $\Gl_n$-character stacks (the twist come from the action of a certain finite group $H(\bm {\mathcal{C}})$ defined below). This twisted conjectural formula for $\Gl_n$ is new (see Conjecture \ref{conjIHy}) and is a modification of the untwisted version introduced in \cite[Conjecture 1.2.1]{HA} (see also \cref{chapterGln} for a review). Recall that we have many evidences for the untwisted conjectural formula: from the works \cite{HA} and \cite{L} we know that it holds under the Euler specialization $t=-1$. Moreover, from the work of Mellit \cite{MellitD}\cite{Mellit}(who generalized some ideas of Schiffmann \cite{schiffmann} to the parabolic case) and Ballandras \cite{ballandras}, we know that it also holds under the Poincaré specialization (i.e. at $q=1$). 

For the twisted version introduced in this paper, we show that it holds under the Euler specialization (see Theorem \ref{non-degenerate-twisted}) from which we prove the specialization $t\mapsto -1$ of Conjecture \ref{IH-conjecture-thm-intro}, namely

\begin{teorema}For any $\chi\in \widehat{A(\bm{\mathcal{C}})}$, we have

$$
IE(\mathcal{M}_{\overline{\bm{\mathcal{C}}}},\mathcal{E}_\chi;q):=IH_c(\mathcal{M}_{\overline{\bm{\mathcal{C}}}},\mathcal{E}_\chi;q,-1)=\dfrac{q^{{\frac{{\rm dim}\mathcal{M}_{\overline{\bm{\mathcal{C}}}}}{2}}}\iota(\bm{\mathcal{C}})}{(q-1)^{2g}|A(\bm{\mathcal{C}})|}\sum_{r \in R_{d_1,\dots,d_k}}\Delta_{r}^{s_{\chi}}\,\mathbb{H}_{g,\bm \omega_r}\left(\sqrt{q},\dfrac{1}{\sqrt{q}}\right).
$$
\end{teorema}

The results of Mellit or Ballandras (for the specialization $q=1$) do not generalize immediately to the twisted case. Indeed, the Higgs bundles are of additive nature and the group of connected components for the adjoint orbits are trivial. 
 \bigskip

As a further evidence we prove the following result (see Theorem \ref{casek=4}).

\begin{teorema}Conjecture \ref{IH-conjecture-thm-intro} is true when $g=0$, $n=2$,  $k=4$ and when the conjugacy classes are all semisimple regular.
 \end{teorema}
\bigskip

The groups of the form $A(\mathcal{C})$ are subgroups of the same group $\bm\mu_n$. Therefore the group 
$$
H(\bm{\mathcal{C}}):=\{(y_1,\dots,y_k)\in A(\bm{\mathcal{C}})\,|\, y_1\cdots y_k=1\},
$$
is well-defined and plays an important role to establish the main results of our paper. Indeed, we use the fact that it acts on some generic $\Gl_n$-character stacks $\mathcal{M}_{\overline{\bm{C}}}$ with local monodromies in the Zariski closure of the $\Gl_n$-conjugacy classes $C_1,\dots,C_k$ above the $\PGl_n$- conjugacy classes $\mathcal{C}_1,\dots,\mathcal{C}_k$.  The group $H(\bm{\mathcal{C}})$ is a subgroup of the so-called Weyl group of the $\Gl_n$-character stack $\mathcal{M}_{\overline{\bm{C}}}$. The whole Weyl group does not act on the $\Gl_n$-character stack itself but it does act on its cohomology thanks to Ballandras' result \cite[Theorem 5.5, Corollary 5.6]{ballandras}. The two actions of $H(\bm{\mathcal{C}})$ on cohomology should coincide. While our action of $H(\bm{\mathcal{C}})$ preserves the weight filtration on cohomology (because it acts on the stack itself), the action of the Weyl group defined in \cite{ballandras} uses analytic methods and it is unclear whether it preserves the weight filtration on cohomology.

On the corresponding moduli spaces of $\Gl_n$-parabolic Higgs bundles, the action of $H(\bm{\mathcal{C}})$ should correspond to the action of \textit{Hecke correspondences}, also called \textit{elementary transformations}. For a definition of the latter groups of automorphisms, see for instance \cite{alfaya-gomez}.

\subsection{Connection with Langlands duality over finite fields}

Here we explain how to relate the geometric results on character stacks to the representation theory of the finite reductive groups $\PGl_n(\F_q),\Sl_n(\F_q)$. As mentioned earlier, the relevant case is the case $g=0$ (i.e. $X=\mathbb{P}^1_{\C}$) which we assume in this section. We also omit $g$ from the notation throughout.

\bigskip

For any conjugacy class $\mathcal{C}$ of $\PGl_n(\overline{\F}_q)$ and $\chi \in \widehat{A(\mathcal{C})}$ defined over $\F_q$, we denote by
$$ {\bf X}_{\IC^{\bullet}_{\overline{\mathcal{C}},\mathcal{L}^\mathcal{C}_{\chi}}}: \PGl_n(\F_q) \to \overline{\Q}_{\ell} $$
the characteristic function of the intersection cohomology complex $\IC^{\bullet}_{\overline{\mathcal{C}},\mathcal{L}^{\mathcal{C}}_{\chi}}$ (see Formula (\ref{charIC})). 
\bigskip

If our conjugacy classes $\mathcal{C}_1,\dots,\mathcal{C}_k$ are over $\mathbb{C}$ we may choose a finitely generated $\Z$-subalgebra $R$ of $\mathbb{C}$ and  $R$-schemes $\mathcal{C}_1/_R,\dots,\mathcal{C}_k/_R$ giving back $\mathcal{C}_1,\dots,\mathcal{C}_k$ after scalar extension from $R$ to $\mathbb{C}$. Then we may choose $R$ "large enough" so that, for any finite field $\F_q$ and any ring homomorphism $R\rightarrow\F_q$, the $k$-tuple of conjugacy classes $\bm{\mathcal{C}}/_{\F_q}$ of $\PGl_n(\overline{\F}_q)$ obtained from $\bm{\mathcal{C}}$ by base change is generic and of same "Jordan type" as $\bm{\mathcal{C}}$ (with $A(\bm{\mathcal{C}}/_{\F_q})=A(\bm{\mathcal{C}})$). By notation abuse (and to alleviate the notation) we will denote again by $\bm{\mathcal{C}}$ the $k$-tuple $\bm{\mathcal{C}}/_{\F_q}$.
\bigskip

We prove the following result (see Theorem \ref{theorem-diagram}).

\begin{teorema}
\label{theorem-convolution-intro}
For any  $\chi \in A(\bm{\mathcal{C}})$, we have
\begin{equation}
\left\langle {\bf X}_{{\IC^{\bullet}_{\overline{\mathcal{C}}_1,\mathcal{L}^{\mathcal{C}_1}_{\chi_1}}}} \ast \cdots \ast {\bf X}_{{\IC^{\bullet}_{\overline{\mathcal{C}}_k,\mathcal{L}^{\mathcal{C}_k}_{\chi_k}}}},1_{1} \right\rangle_{\PGl_n(\F_q)}=IE\big(\mathcal{M}_{\overline{\bm{\mathcal{C}}}},\mathcal{E}_{\chi};q\big).  
\label{equation-convolution-intro}
\end{equation}    
where $*$ is the convolution product of functions on $\PGl_n(\F_q)$, and $1_1$ is the function that takes the value $1$ at $1$ and $0$ elsewhere.
\end{teorema}

\bigskip

\begin{remark}
The proof of Theorem \ref{theorem-convolution-intro} comes from a generalization of Katz's result \cite[Appendix A]{HRV} relating the E-polynomial of a variety $X/_{\C}$ to the count of points of its base change $X/_{\F_q}$ over $\F_q$. This result of Katz was used by Hausel, Letellier and Rodriguez-Villegas \cite{HA} to compute E-polynomials of generic $\Gl_n(\C)$-character varieties with semisimple local monodromies.

We extend Katz's result (see Theorem \ref{countingfq} and Theorem \ref{polynomial-count}), to relate $IE(X;q)$ to the characteristic function of the intersection cohomology complex of $X/_{\F_q}$. We will also need a twisted version  (see Theorem \ref{theoremtwistedE}) when we have an action of a  finite group on $X$.  
\end{remark}

\bigskip

Given a conjugacy class $\mathcal{C}$ of $\PGl_n(\overline{\F}_q)$ and $\chi \in A(\mathcal{C})$ (or equivalently an irreducible $\PGl_n$-equivariant local system on $\mathcal{C}$), in this paper we introduce a way to define a so-called \emph{character-sheaf} on $\Sl_n(\overline{\F}_q)$ denoted by $\mathcal{X}_{\mathcal{C},\chi}^{\Sl_n}$ (see \S\ref{CSSL}).

The theory of character-sheaves is due to Lusztig \cite{Lusztig}\cite{CS2}\cite{LUG} (see also \cite{mars} or \cite{LaumonCS} for a survey) and is very important as the characteristic functions of the character-sheaves on $G$ are closely related to the irreducible characters of the finite group $G(\F_q)$ (they mostly coincide for a general $G$ and always coincide when $G=\Gl_n$).

\bigskip

In \cref{section-sln}, we show the following result (see Theorem \ref{theoremmultiplicitiessln}).

\begin{teorema}
    \label{theoremmultiplicitiessln-intro}
For any $\chi \in \widehat{A(\bm{\mathcal{C}})}$, we have
\begin{equation}
\label{multiplicities-intro}
\left\langle {\bf X}_{\mathcal{X}_{\mathcal{C}_1,\chi_1}^{\Sl_n}}\cdots{\bf X}_{\mathcal{X}_{\mathcal{C}_k,\chi_k}^{\Sl_n}},1\right\rangle_{\Sl_n(\F_q)}=\dfrac{\iota(\bm{\mathcal{C}})}{|A(\bm{\mathcal{C}})|}\sum_{r \in R_{d_1,\dots,d_k}}\Delta_{r}^{s_{\chi}}\,\mathbb{H}_{\bm \omega_r}\left(0,\sqrt{q}\right).     
\end{equation} 
where $\cdot$ is the pointwise multiplication of functions on $\Sl_n(\F_q)$ and $1$ is the trivial character of $\Sl_n(\F_q)$.
\end{teorema}

\bigskip

Under Conjecture \ref{IH-conjecture-thm-intro}, the RHS of Formula (\ref{multiplicities-intro}) has the following interpretation in terms of the cohomology of $IH^{\bullet}_c(\mathcal{M}_{\overline{\bm{\mathcal{C}}}},\mathcal{E}_{\chi})$. We define the \emph{pure part} of $IH_c(\mathcal{M}_{\overline{\bm{\mathcal{C}}}},\mathcal{E}_{\chi};q,t)$ as
$$PIH_c\big(\mathcal{M}_{\overline{\bm{\mathcal{C}}}},\mathcal{E}_{\chi};q\big)\coloneqq \sum_i\dim(W^i_i/W^i_{i-1})q^{i/2}.$$

Formula (\ref{IHconjecture-intro}) together with Formula (\ref{multiplicities-intro}) implies that \begin{equation}
PIH_c\big(\mathcal{M}_{\overline{\bm{\mathcal{C}}}},\mathcal{E}_{\chi};q\big)=\dfrac{q^{{\frac{{\rm dim}\mathcal{M}_{\overline{\bm{\mathcal{C}}}}}{2}}}\iota(\bm{\mathcal{C}})}{|A(\bm{\mathcal{C}})|}\sum_{r \in R_{d_1,\dots,d_k}}\Delta_{r}^{s_{\chi}}\,\mathbb{H}_{\bm \omega_r}\left(0,\sqrt{q}\right)=q^{{\frac{{\rm dim}\mathcal{M}_{\overline{\bm{\mathcal{C}}}}}{2}}}\left\langle {\bf X}_{\mathcal{X}_{\mathcal{C}_1,\chi_1}^{\Sl_n}}\cdots{\bf X}_{\mathcal{X}_{\mathcal{C}_k,\chi_k}^{\Sl_n}},1\right\rangle_{\Sl_n}   
\end{equation}

We can summarize the main results of this paper with the following diagram

\begin{center}
\begin{tikzcd}
\label{diag11}
& & IH_c\left(\mathcal{M}_{\overline{\bm{\mathcal{C}}}},\mathcal{E}^{\bm{\mathcal{C}}}_{\chi};q,t\right)\ar[dd, "t\mapsto -1"']&\text{RHS Formula (\ref{IHconjecture-intro})}\arrow[dd, "\text{ "pure part"}"]\arrow[ddl,"t \mapsto -1"']\arrow[l, equal, "\text{Conjecture \ref{IH-conjecture-thm-intro}}"']\\
&&&\\
& & \left\langle {\bf X}_{{\IC^{\bullet}_{\overline{\mathcal{C}}_1,\mathcal{L}_{\chi_1}}}} \ast \cdots \ast {\bf X}_{{\IC^{\bullet}_{\overline{\mathcal{C}}_k,\mathcal{L}_{\chi_k}}}},1_{1} \right\rangle_{\PGl_n(\F_q)} &  q^{{\frac{{\rm dim}\mathcal{M}_{\overline{\bm{\mathcal{C}}}}}{2}}}\left\langle {\bf X}_{\mathcal{X}_{\mathcal{C}_1,\chi_1}^{\Sl_n}}\cdots{\bf X}_{\mathcal{X}_{\mathcal{C}_k,\chi_k}^{\Sl_n}},1\right\rangle_{\Sl_n(\F_q)}
\end{tikzcd}
\end{center}

Therefore the mixed Hodge series of $\PGl_n$-character stacks interpolate the generic structure coefficients of the following two based rings : 

(i) the ring of class functions on $\PGl_n(\F_q)$ equipped with the convolution product and with basis the characteristic functions of the intersection cohomology complexes of local systems on the conjugacy classes of $\PGl_n(\overline{\F}_q)$,

(ii) the ring of class functions on $\Sl_n(\F_q)$ equipped with the pointwise multiplication and with basis the characteristic functions of the character sheaves on $\Sl_n(\overline{\F}_q)$.

\begin{oss}
Under the correspondence 

$$
\{\IC^\bullet_{\overline{\mathcal{C}},\mathcal{L}_\chi}\}_{\mathcal{C},\chi}\longrightarrow \{\mathcal{X}^{\Sl_n}_{\mathcal{C},\chi}\}_{\mathcal{C},\chi}$$
that makes the above diagram work, the identity element $1_1$ for the convolution product on functions on $\PGl_n(\F_q)$ (i.e. the function associated to the pair $(\mathcal{C},\chi)=(\{1\}, {\rm Id})$) corresponds to the identity element for the pointwise multiplication on functions on $\Sl_n(\F_q)$, i.e. to the trivial character of $\Sl_n(\F_q)$. This correspondence agrees with the first construction of Springer correspondence on Lie algebras using Fourier transforms \cite{Springer}. In the classical Springer correspondence on groups due to Borho-MacPherson \cite[\S 6.2]{shoji} (which uses the decomposition theorem instead of Fourier), the function $1_1$ would correspond to the Steinberg character of $\Sl_n(\F_q)$. Recall that one goes from the original construction of the Springer correspondence (using Fourier transforms) to Borho-MacPherson's construction by tensoring irreducible characters of Weyl groups by the sign character.
\end{oss}

\subsection{Comments on the "pure part" specialization}

A diagram similar to the above one is known in the case of $\Gl_n$ (see \S\ref{picture-GLn}). In particular, the pure part of the conjectural expression for the cohomology of $\Gl_n(\C)$-character stacks gives back the multiplicities for the character ring of the finite group $\Gl_n(\F_q)$. Recall that for the finite group $\Gl_n(\F_q)$ the characteristic functions of character sheaves coincide (up to sign) with the irreducible characters. This is not the case for the group $\Sl_n(\F_q)$, where the two bases are quite different, and an algorithm for finding the basis-change matrix was only discovered about 20 years ago, in two independent papers by Bonnaf\'e and Shoji \cite{Bonnafe}\cite{Shoji}.

\bigskip

In the case of $\Gl_n$, moreover, it is easier to understand the "pure part" specialization, thanks to the geometric interpretation of it in terms of quiver varieties, as we  now see.

Let $(C_1,\dots,C_k)$ be a generic $k$-tuple of conjugacy classes of $\Gl_n$. The pure part of the intersection cohomology of the $\Gl_n$-character stack with local monodromies in $\overline{C}_1,\dots,\overline{C}_k$ is conjectured to be the intersection cohomology of the quiver stack 

$$
\mathcal{Q}=\left[\{(x_1,\dots,x_k)\in\overline{\mathcal{O}}_1\times\dots\times\overline{\mathcal{O}}_k\,|\, x_1+\cdots+x_k=0\}/\Gl_n\right],
$$
where $(\mathcal{O}_1,\dots,\mathcal{O}_k)$ is a generic $k$-tuple of adjoint orbits of $\frak{gl}_n$ with same Jordan type as $C_1,\dots,C_k$. 

More precisely, the Poincar\'e series of these quiver stacks is computed, in \cite{HA} in the semisimple case and in  \cite{letellier2} for any adjoint orbits, and we can check that they agree with the pure part of the conjectured formula for the mixed Poincar\'e series of the corresponding character stacks.
\bigskip

Using Fourier transforms on $\frak{gl}_n$, it is proved \cite[Theorem 7.4.1]{letellier2} that the Poincar\'e series of these quiver stacks agree with the multiplicity $\langle \chi_1 \otimes \cdots \otimes \chi_k, 1 \rangle$ of the trivial character in the tensor product of irreducible characters $(\chi_1,\dots,\chi_k)$ of $\Gl_n(\F_q)$ of same Jordan type as $\mathcal{O}_1,\dots,\mathcal{O}_k$.  The proof uses the work of Springer \cite{Springer}, Kazhdan \cite{Kazhdan}, Lusztig \cite{Lu-Fourier} and Letellier \cite{Let} from which we get a precise relationship between the values of the irreducible characters of $\Gl_n(\F_q)$ and those of the characteristic functions of the Deligne-Fourier transform of the intersection cohomology complexes on the adjoint orbits of $\frak{gl}_n(\overline{\F}_q)$, see \cite[Theorem 6.9.1]{letellier2}.
\bigskip

For an arbitrary group, the relationship between the Lie algebra and the group is unclear as, unlike for groups,  stabilizers of semisimple adjoint orbits of Lie algebras  are always connected (when the characteristic of the field $K$ is zero or large enough).

\subsection{Fourier inversion formula}\label{Fourinv}

In the above diagram, put $K_i:=\IC^\bullet_{\overline{C}_i,\mathcal{L}_{\chi_i}}$ and let $\hat{K}_i$ be the corresponding character sheaf $\mathcal{X}_{\mathcal{C}_i,\chi_i}^{\Sl_n}$. 

Notice that the multiplicities in the above diagram are the Frobenius traces, respectively  of

$$
\mathcal{H}^\bullet(K_1*\cdots* K_k|_{[1/\PGl_n]}),\hspace{.5cm}\text{ and }\hspace{.5cm} H_c^\bullet([\Sl_n/\Sl_n],\hat{K}_1\otimes\cdots\otimes\hat{K}_k).$$ 

Consider a Deligne-Fourier transform $\hat{\cdot}:D_c^b(V)\rightarrow D_c^b (V^*), K\mapsto \hat{K}$ on the derived category of $\ell$-adic sheaves on a vector space $V$. By the Fourier inversion formula together with the fact that it transforms the convolution product into the tensor product, for any complexes $K_1,\dots, K_k$ on $V$, we get that 

$$
\mathcal{H}^\bullet(K_1*\cdots *K_k|_{\{0\}}),\hspace{.5cm}\text{ and }\hspace{.5cm} H_c^\bullet(V^*,\hat{K_1}\otimes\cdots\otimes\hat{K_k})
$$
are isomorphic (up to a Tate twist by ${\rm dim}(V)$) and so our conjectural picture is  a kind of multiplicative version of the Fourier inversion formula.

\bigskip

\subsection{Structure of the paper}

The paper is organised as follows. \begin{itemize}
    \item In Section $2$, we review the definitions and properties of intersection cohomology and generalize Katz’s result on $E$-polynomials of complex varieties \cite[Appendix A]{HRV} to the case of intersection cohomology, taking into account the twist induced by the action of a finite group. 
    \item In Section $3$, we introduce the combinatorics to describe the conjugacy classes of $\Gl_n$.
    \item In Section $4$, we recall the results on the cohomology of $\Gl_n$-character stacks and using the results of Section $2$, we extend them to a twisted setting.
    \item In Section $5$, we show our main results concerning the  cohomology of $\PGl_n$-character stacks.
    
    \item In Section $6$, we review the definition of Lusztig's character-sheaves and the correspondence between local systems on conjugacy classes of $G$ and character-sheaves on the dual group $G^\flat$.
    \item In Section $7$, we discuss the case of the dual pair $(\Sl_n,\PGl_n)$.
    \item In Section $8$, we show our main results concerning the connection between the representation theory of $(\Sl_n,\PGl_n)$ and the geometry of $\PGl_n$-character stacks.
    
    \item In Section $9$, we study the case $n=2$ to illustrate our main results and conjectures.
\end{itemize}

\subsection{Acknowledgements.}

The authors are very grateful to Luca Migliorini and Jean Michel for many useful discussions regarding this paper. A part of this work was done while the second author was visiting the Université Paris Cité. The second author would like to thank UPC for its generous support.

\section{Preliminaries on intersection cohomology and weight filtration}

In the following, $K$ is an algebraically closed field which is either $\C$ or $\overline{\F}_q$  and $X$ is an algebraic variety or a Deligne-Mumford stack over $K$. We denote by $D^b_c(X),\Perv(X)$ respectively the derived category of constructible sheaves/the abelian category of perverse sheaves on $X$ with coefficients in the field $\kappa$ with $\kappa=\C$ if $K=\C$ and $\kappa=\overline{\Q}_\ell$ if $K=\overline{\F}_q$ where $\ell \nmid q$.

We will also need the notion of \textit{Weil structure} when $K=\overline{\F}_q$ and $F:X \to X$ is a geometric Frobenius (or equivalently, an $\F_q$-stack $X_o$ such that $X=X_o\times_{\F_q}\overline{\F}_q$).

 An $F$-equivariant structure (or Weil structure) on $\mathcal{F}\in D^b_c(X)$ is then an isomorphism $$\phi:F^*(\mathcal{F}) \to \mathcal{F} .$$ 

\begin{remark}If $\mathcal{F}$ is the pullback of a complex $\mathcal{F}_o$ on $X_o$, then it admits a canonical $F$-equivariant structure, see for instance \cite[Chapter 1]{Weil}. 
\end{remark}
\bigskip

We say that $(\mathcal{F},\phi)$ is an $F$-\emph{equivariant complex} on $X$. Given two $F$-equivariant complexes $(\mathcal{F},\phi)$ and $(\mathcal{F}',\phi')$, the Frobenius $F$ acts on ${\rm Hom}(\mathcal{F},\mathcal{F}')$ as 

$$
f\mapsto \phi'\circ F^*(f)\circ\phi^{-1}.
$$
We denote by $D_c^b(X;F),\Perv(X;F)$ the category of $F$-equivariant complexes/$F$-equivariant perverse sheaves on $X$ with ${\rm Hom}(\mathcal{F},\mathcal{F'})^F$ as the set of morphisms $(\mathcal{F},\phi)\rightarrow(\mathcal{F}',\phi')$.

The characteristic function of $(\mathcal{F},\phi)\in D_c^b(X;F)$ is the function $\X_{\mathcal{F},\phi}:X^F \to \overline{\Q}_{\ell}$ defined by 

\begin{equation}
\X_{\mathcal{F},\phi}(x)\coloneqq \sum_{i \in \Z}(-1)^i \tr(\phi_x^i: \mathcal{H}^i_x(\mathcal{F}) \to \mathcal{H}^i_x(\mathcal{F})).
\label{charIC}\end{equation}

The function $\X_{\mathcal{F},\phi}$ does depend on the choice of the isomorphism $\phi$. However, in all the cases of relevance for this article, we can make a canonical choice of the isomorphism $\phi$ and we will often drop it from the notation. In particular, if $X$ is an algebraic group, we will always assume that $\phi_e=Id$. 

\begin{oss}
Given a Weil structure $\phi$, we denote by 
$$\phi^n \coloneqq \phi \circ (F^*)(\phi)  \cdots (F^*)^{n-1}$$
the Weil structure $\phi^n$ for the  $\F_{q^n}$-Frobenius $F^n$. The characteristic function $\X_{\mathcal{F},\phi^n}$ is thus a function on $X^{F^n}$.
\end{oss}

\subsection{Intersection cohomology}
\label{intersection-cohomology}
For an equidimensional variety $X$ and an open smooth subset $U \subseteq X$, for every local system $\mathcal{L}$ on $U$, we denote by $\IC^\bullet_{X,\mathcal{L}} $ the intersection cohomology complex on $X$ with coefficients in $\mathcal{L}$. In particular $\IC^\bullet_{X,\mathcal{L}}[{\rm dim}\, X]$ is  a perverse sheaf.

If $\mathcal{L}=\kappa$, we will simply write $\IC^\bullet_X$ instead of $\IC^\bullet_{X,\kappa}$. If $K=\overline{\F}_q$  and $F: X \to X$ is a geometric Frobenius, the intersection cohomology sheaf has a canonical Weil structure $\phi: F^*(\IC^\bullet_X) \to \IC^\bullet_X$.

\bigskip

Let $X$ be an equidimensional variety and  $X=X_1 \cup \cdots \cup X_r$ its decomposition into irreducible components and denote by $\nu$ the canonical finite map $$\nu:X_1 \bigsqcup \cdots \bigsqcup X_r \to X .$$ Put $X_i^{\circ}=X_i \setminus (\bigcup_{j \neq i}X_j)$ for each $i$. Notice that $X_i^{\circ}$ is a non-empty and dense open subset of $X_i$ and $X_i^{\circ} \cap X_j^{\circ}=\emptyset$ for each $i,j$.

Given a smooth $U \subseteq X$, we must have that $\displaystyle U=\bigsqcup_{i=1}^r (U \cap X_i^{\circ})$. Put $U_i \coloneqq U \cap  X_i^{\circ}$ for each $i=1,\dots,r$. The datum of a local system $\mathcal{L}$ on $U$ is thus equivalent to the datum of local systems $\mathcal{L}_i$ on $U_i$ for each $i=1,\dots,r$.

It is not hard to see that we have an equality \begin{equation}
\label{reducible-IC}
\IC^\bullet_{X,\mathcal{L}}=v_*\left(\bigoplus_{i=1}^r\IC^\bullet_{X_i,\mathcal{L}_i}\right).    
\end{equation}

\bigskip

We denote by $IH_c^{\bullet}(X)$ the (compactly supported) intersection cohomology of $X$, i.e. $IH_c^{\bullet}(X)\coloneqq \mathbb{H}_c^{\bullet}(X,\IC^\bullet_X)$. If $X$ is smooth, we have $IH_c^{\bullet}(X)=H^{\bullet}_c(X)$.

\bigskip

\begin{remark}
From Equation (\ref{reducible-IC}) we deduce that, if $X=X_1 \cup \cdots \cup X_r$ is the decomposition into irreducible components of an equidimensional variety, we have an equality \begin{equation}
    IH^{\bullet}_c(X)=\bigoplus IH^{\bullet}_c(X_i)
\end{equation}   
\end{remark}

\subsection{Weight filtrations}
\label{chapter-weight-filtration}
If $K=\C$, each intersection cohomology group is equipped with the weight filtration $W^k_{\bullet}IH_c^k(X)$, introduced by Saito in \cite{saito}. If $X$ is smooth, through the identification $IH_c^{\bullet}(X)=H_c^{\bullet}(X)$, it corresponds to the weight filtration introduced by Deligne in \cite{HodgeIII}.

If $K=\overline{\F}_q$ and we assume to have a geometric Frobenius $F: X \to X$, we have a weight filtration $W^k_{\bullet}IH^k_c(X)$,   where $W^k_{m}IH^k_c(X)$  is the subspace on which the  eigenvalues of the Frobenius $F$ are of absolute value $\leq q^{\frac{m}{2}}$.


We define the mixed intersection cohomology Poincar\'e polynomial $IH_c(X;q,t) \in \Z[\sqrt{q},t]$

$$
IH_c(X;q,t)\coloneqq\sum_{i,k}{\rm dim}(W^k_i/W^k_{i-1})q^{i/2}t^k.
$$
Then $IH_c(X;1,t)=\sum_k{\rm dim}\, IH_c^k(X)\, t^k$ is the (compactly supported) intersection cohomology Poincar\'e polynomial and $IH_c(X;q,-1)$ is the so-called $IE$-polynomial denoted by $IE(X;q)$. In the cases of relevance for this article, $IH_c(X,q,t)$ and $IE(X,q)$ will be actual polynomial in $q$, i.e. $W^k_i/W^k_{i-1}=0$ unless $i$ is even.

We also define the pure part $PIH_c(X;q)$ as  $$PIH_c(X;q)\coloneqq \sum_{k}\dim(W^k_k/W^k_{k-1})q^{k/2}.$$

\bigskip

Given a variety $X/_{\overline{\F}_q}$ with Frobenius $F:X \to X$, say that $X$ has the \textit{IC-polynomial} property with IC-polynomial $P_X(t) \in \Z[t]$ if, for any $n \in \N_{>0}$, \begin{equation}
\sum_{x \in X^{F^n}}\X_{\IC^\bullet_X,\phi^n}(x)=P_X(q^n).    
\end{equation}

\begin{oss}
Notice that, if $X$ is smooth, we have  $$\displaystyle \sum_{x \in X^{F^n}}\X_{\IC^\bullet_X,\phi^n}(x)=|X^{F^n}|$$ and thus $X$ has the IC-polynomial property if $X$  has polynomial count in the classical sense with counting polynomial $P_X(t)$, see for instance \cite[Appendix by Katz]{HRV}, \cite[Section 2.2]{LRV}.    
\end{oss}

\bigskip

We have the following.

\begin{teorema}
\label{countingfq}
If $X/_{\overline{\F}_q}$ has the \textit{IC-polynomial property}  with IC-polynomial $P_X(t)$, we have
\begin{equation}
P_X(q)=IE(X;q)
\end{equation}
\end{teorema}

The proof of this theorem is very similar to that of \cite[Theorem 2.8]{LRV}. We give it here for completeness. 

\begin{proof}
By the trace formula, for any $r$, we have \begin{equation}
\label{grothendiecktrace0}
    P_X(q^r)=\sum_{x \in X^{F^r}}\X_{\IC^\bullet_X,\phi^r}=\sum_{k}(-1)^k \tr\left(F^r\,|\,IH^k_c(X)\right) .\end{equation}

Let $\lambda_{i,k,1}q^{\frac{i}{2}},\dots,\lambda_{i,k,s_{k,i}}q^{\frac{i}{2}}$  be the eigenvalues, counted with multiplicities, of $F$  on $W^k_i/W^k_{i-1}$. We thus have, for any $r \geq 1$, 

$$\tr\left(F^r\,|\,W^k_{i}/W^k_{i-1}\right)=\sum_{h=1}^{s_{k,i}}(\lambda_{i,k,h})^rq^{\frac{ri}{2}}$$ and thus $$\sum_{k}(-1)^k \tr\left(F^r\,|\,IH^k_c(X)\right)=\sum_{i}\left(\sum_k (-1)^k \sum_{h=1}^{s_{k,i}}(\lambda_{i,k,h})^r\right)q^{\frac{ir}{2}}  .$$ 

If $\displaystyle P_X(t)=\sum_{i}c_{i}t^i\in\Z[t]$, from Formula (\ref{grothendiecktrace0}) we deduce that 
\begin{equation}
\sum_k (-1)^k \sum_{h=1}^{s_{k,i}}(\lambda_{i,k,h})^r=\begin{cases}
c_{\frac{i}{2}} \text{ if } i \text{ is even }\\
0 \text{ otherwise}
\end{cases}.
\end{equation}

From \cite[Lemma 2.9]{LRV}, we deduce that $$\dim \left(W^k_i/W^k_{i-1}\right)=\begin{cases}
c_{\frac{i}{2}} \text{ if } i \text{ is even }\\
0 \text{ otherwise}
\end{cases}.$$    
\end{proof}

\bigskip

For a variety $X/_{\C}$, we say that $X$ has the IC-polynomial property with IC-polynomial $P_X(t)$ if there exists a finitely generated $\Z$-subalgebra $R \subseteq \C$ and a separated $R$-scheme $X_R$ of finite type such that 
$$
X_R \times_R \C \cong X,
$$
and such that, for any ring homomorphism $f:R \to \F_q$, the variety $X^f=X_R\times_R\overline{\F}_q$ has the IC-polynomial property with IC-polynomial $P_X(t)$. 

We have the following result.
\begin{teorema}\label{polynomial-count} If $X/_\C$ has the IC-polynomial property with IC-polynomial $P_X(t)$, then 
\begin{equation}
\label{katzformula}
IE(X;q)=P_X(q).
\end{equation}
\label{Katz}\end{teorema}

\begin{oss}
Letellier \cite[Theorem 3.3.2]{letellier2} shows Theorem \ref{Katz} above under some additional hypothesis on the variety $X$ but keeping also track of the Hodge filtration on intersection cohomology. If we do not bother about the Hodge filtration (as it is the case in this paper), Theorem \ref{polynomial-count} above is a consequence of Theorem \ref{countingfq} and Theorem \ref{theorem-ICextension} below.

Although Theorem \ref{theorem-ICextension}  seems to be well known to the experts, we were not able to locate a proof in the literature. 
\end{oss}

\begin{teorema}
\label{theorem-ICextension}  
Given $X/_{\C}$ and $R \subseteq \C$ as above, there exists an open subscheme $U \subseteq \spec(R)$ such that, for any $f:R \to \F_q$, we have 
\begin{equation}
IH_c(X;q,t)=IH_c(X^f;q,t)    
\end{equation}
\end{teorema}

\begin{proof}[Sketch of proof]
Let $\psi:\widetilde{X} \to X$ be a resolution of singularities. Restricting to an open subset $V=\spec(R') \subseteq \spec(R)$ if necessary, we can assume that $\widetilde{X}$ and $\psi$ are defined over $R$, i.e. there exists a smooth $R$-scheme $\widetilde{X}_R$ and a projective map $\psi_R:\widetilde{X}_R \to X_R$ which give back $\psi$ after extension of scalars. 


\bigskip

Fix  an isomorphism $\C \cong \overline{\Q}_{\ell}$ and identify $H^*_c(\widetilde{X},\C) \cong H^*_c(\widetilde{X},\overline{\Q}_{\ell})$ through this isomorphism. Recall that, there exists $U \subseteq \spec(R)$ such that, if $\Imm(f) \in U$ we have a natural isomorphism \begin{equation}
 \label{comparisonisom}   
H^*_c(\widetilde{X}^f) \cong H^*_c(\widetilde{X}) \end{equation} which preserves weight filtration on both sides.

Indeed, let $\sigma:\widetilde{X}_R \to \spec(R)$ be the structural morphism. The complex $\sigma_!\overline{\Q}_{\ell}$ is constructible, see for instance \cite[Chapter 2]{Deligne-etale}. In particular, there exists a non-empty open $U \subseteq \spec(R)$ on which  $\sigma_!\overline{\Q}_{\ell}$ is constant. Denote by $\xi:\spec(\C) \to \spec(R)$ the (geometric) generic point of $\spec(R)$ coming from the embedding $R \subseteq \C$ and, for any $f:R \to \F_q$, denote by $\xi_f:\spec(\overline{\F}_q) \to \spec(R)$ the corresponding geometric point. 

If $\Imm(\xi_{f}) \in U$, from the fact that $\sigma_!\overline{\Q}_{\ell}$ is constant and from the proper base change theorem, we have the following chain of isomorphisms:
\begin{equation}
\label{isomorphism-cohomology-constant}
     H^{\bullet}_c(\widetilde{X}^f) \cong (\sigma_!\overline{\Q}_{\ell})_{\xi_f} \cong (\sigma_!\overline{\Q}_{\ell})_{\xi} \cong H^{\bullet}_c(\widetilde{X}) .\end{equation}

The results of \cite[Theorem 14]{Deligne} show that the isomorphism (\ref{comparisonisom}) preserves the weight filtration on both sides. We may choose $R$ as above such that we have the isomorphism (\ref{comparisonisom}) for all ring homomorphisms $f: R \to \F_q$.
\bigskip

 The decomposition theorem for the map $\psi: \widetilde{X} \to X$ implies that we have a (non-canonical) splitting  \begin{equation}\label{DT}
 \psi_!(\overline{\Q}_{\ell}) \cong \bigoplus_{a \in \Z}\mathcal{P}_a[-a],\end{equation} where each $\mathcal{P}_a$ is a (semisimple) perverse sheaf over $X$. We have thus an identity $$H^{\bullet}_c(\widetilde{X}) \cong \bigoplus_{a \in \Z}H^{\bullet-a}_c(X,\mathcal{P}_a) .$$
  
 For each $a \in \Z$, put $\Gr^P_aH^{\bullet}_c(\widetilde{X})\coloneqq H^{\bullet-a}_c(X,\mathcal{P}_a)$.

 Since each $\mathcal{P}_a$ is semisimple, we can refine the decomposition (\ref{DT})  writing for each $a$, $$\mathcal{P}_a \cong \mathcal{P}_{a,j_{a,1}} \oplus \cdots \oplus \mathcal{P}_{a,j_{a,s_a}}$$ where $\mathcal{P}_{a,h}$ are simple perverse sheaves. Recall that, for a simple perverse sheaf $\mathcal{P}$, there is an associated irreducible closed subvariety $\Supp(\mathcal{P})$ called its \textit{support}. Put $\Supp(\psi)\coloneqq \{\Supp(\mathcal{P}_{a,h})\}_{a,j}$.

For each $T \in \Supp(\psi)$, let $$\Gr^{\bullet,P}_{a,T}\coloneqq \bigoplus_{\Supp(\mathcal{P}_{a,j})=T} H^{\bullet-a}_c(X,\mathcal{P}_{a,j}) .$$ Notice that we have an isomorphism which preserves the weight filtration \begin{equation}
\label{DT2}
IH^{\bullet}_c(X)\cong \Gr^{\bullet,P}_{0,X}
\end{equation}

For each $T \in \Supp(\psi)$, De Cataldo and Migliorini \cite[Section 1.3.3]{DCM} introduce a variety $\mathcal{T}$ and a proper map $r_{T}:\mathcal{T} \to \widetilde{X}$ such that
\begin{equation}
\label{DT3}
\Gr^{\bullet,P}_{0,X}=\displaystyle\bigcap_{T \neq X} \Ker(r^*_{T})
\end{equation} 
where $r^*_T:H^{\bullet}_c(\widetilde{X}) \to H^{\bullet}_c(\mathcal{T})$ is the corresponding map in cohomology. We deduce that we have an identification which preserves weight filtration \begin{equation}
    \label{DT4}
    IH^{\bullet}_c(X)=\bigcap_{T \neq X} \Ker(r^*_{T}) \subseteq H^{\bullet}_c(\widetilde{X}).   
\end{equation}
\bigskip
 
Denote by $\psi^f:\widetilde{X}^f \to X^f$ the corresponding map of $\F_q$ varieties and  write the corresponding decomposition $$\psi^f_!(\overline{\Q}_{\ell}) \cong \bigoplus_{a \in \Z}\mathcal{P}^{f}_a[-a],$$ where each $\mathcal{P}^{f}_{a}$ is a semisimple perverse sheaf on $X^f$ defined over $\F_q$.

The arguments of \cite{DCM} show that, for each $a,j$, the subvariety $\Supp(\mathcal{P}_{a,j})$ is defined over the field of fraction of $R$. Restricting to an open subset $V=\spec(R') \subseteq \spec(R)$ if necessary, we can assume that $\Supp(\mathcal{P}_{a,j})$ is defined over $R$ for all $a,j$. The results of \cite[Lemma 6.2.6]{BBD} show that, for each $a,j$, there exists a corresponding simple  perverse sheaf $\mathcal{P}^{f}_{a,j}$ on $X^f$ defined over $\F_q$ such that $$\mathcal{P}^f_{a} \cong \bigoplus \mathcal{P}^f_{a,j} .$$

Moreover, for each $a,j$, we have that $\Supp(\mathcal{P}^f_{a,j})=\Supp(\mathcal{P}_{a,j})^f$ and the latter is defined over $\F_q$. For each $T^f \in \Supp(\psi^f)$, we have a corresponding $\mathcal{T}^f$ and, by a similar argument, we have an identification which preserves weight filtrations \begin{equation}
    \label{DT5}
    IH^{\bullet}_c(X^f)=\displaystyle\bigcap_{T \neq X} \Ker(r^*_{T^f}) \subseteq H^{\bullet}_c(\widetilde{X}^f)
\end{equation} 

This implies that the isomorphism (\ref{comparisonisom}) restricts to an isomorphism \begin{equation}
    \label{comparisonIH}
    IH^{\bullet}_c(X)\cong IH^{\bullet}_c(X^f)
\end{equation}
which preserves weight filtration on both sides.
\end{proof}

\subsection{$W$-equivariance}\label{W-equiv-section}

For more details on this section see \cite{laumonletellier2}. Let $W$ be a finite group acting on the right on a variety $X$. A $W$-\emph{equivariant complex} on $X$ is a pair $(K,\theta)$ with $K\in D_c^b(X)$ and $\theta=(\theta_w)_{w\in W}$ of isomorphisms

$$
\theta_w:w^*(K)\simeq K
$$
such that

(1) $\theta_{ww'}=\theta_w\circ w^*(\theta_{w'})$ for all $w,w'\in W$, and

(2) $\theta_1=1_K$ where $1_K:K\rightarrow K$ is the identity morphism.
\bigskip

If $(K,\theta)$ and $(K',\theta')$ are two $W$-equivariant complexes on $X$, then $W$ acts on ${\rm Hom}(K,K')$ as

$$
w\cdot f=\theta'_w\circ w^*(f)\circ(\theta_w)^{-1}
$$
for all $w\in W$ and $f\in{\rm Hom}(K,K')$. A morphism $(K,\theta)\rightarrow (K',\theta')$ is an element of ${\rm Hom}(K,K')^W$. We denote by $D_c^b(X,W)$ the category of $W$-equivariant complexes on $X$.
\bigskip

\begin{oss}
\label{remark-trivial-sheaf}
The constant sheaf $\kappa$ has a canonical $W$-equivariant structure $\theta$, coming from the canonical isomorphisms $$\theta_w:w^*(\kappa) \to \kappa .$$
\end{oss}

\bigskip
Assume given a $W$-invariant smooth open $U \subseteq X$. For any $W$-equivariant local system $(\mathcal{F},\theta)$, the object $\IC^{\bullet}_{X,\mathcal{L}}$ has a corresponding $W$-equivariant structure $\alpha$ which extends $\theta$.
Notice that, in particular, from Remark \ref{remark-trivial-sheaf}, the intersection cohomology complex $\IC_{X}^{\bullet}$ has a canonical $W$-equivariant structure.

\bigskip

If $W$ acts trivially on $X$, then a $W$-equivariant complex on $X$ is  a pair $(K,\tilde{\theta})$ where $\tilde{\theta}$ is an action of $W$ on $K$, i.e. $\tilde{\theta}$ is a group homomorphism

$$
\tilde{\theta}: W\rightarrow{\rm Aut}(K).
$$

Then we have a decomposition

$$
K=\bigoplus_{\chi\in\widehat{W}}K(\chi)
$$
where $K(\chi)\rightarrow K$ is the kernel of the idempotent $1-e(\chi)\in{\rm End}(K)$ with

$$
e(\chi)=\frac{\chi(1)}{|W|}\sum_{w\in W}\overline{\chi(w)}\,\tilde{\theta}(w).
$$

\bigskip

\begin{remark}
\label{equiv-pushforward-remark}
Let $X$ and $Y$ be two varieties, with a $W$-action on $X$ (on the right) and $f:X \to Y$ a $W$-invariant morphism, i.e. $f( x\cdot w)=f(x)$ for each $x \in X$ and $w \in W$.

Given $(\mathcal{F},\theta) \in D^b_c(X,W)$, the complex $f_*(\mathcal{F})$ is endowed with a $W$-action $$\tilde{\theta}:W\rightarrow{\rm Aut}(f_*(\mathcal{F}))$$
as follows. Notice that, for any $w \in W$, we have $w^*(\mathcal{F})=(w^{-1})_*\mathcal{F}$. We thus have
\begin{equation}
\label{equiv-identi}
    f_*w^*(\mathcal{F})=f_*(w^{-1})_*(\mathcal{F})=(f \circ w^{-1})_*(\mathcal{F})=f_*(\mathcal{F})
,\end{equation} since $f \circ w^{-1}=f$. Define thus $\tilde{\theta}_w \in \Aut(f_*\mathcal{F}) $ as 
$$\tilde{\theta}(w) \coloneqq f_*(\theta_w),$$
through the identifications (\ref{equiv-identi}) above. 
\end{remark}

\bigskip

Assume now that $K=\overline{\F}_q$ and let $F:X\rightarrow X$ be a geometric Frobenius which commutes with the action of $W$. Let $K\in D_c^b(X)$ be equipped with an action $\tilde{\theta}:W\rightarrow{\rm Aut}(K)$. Assume given an $F$-equivariant structure $\varphi:F^*(K)\rightarrow K$ such that the following diagram commutes for all $w\in W$

\begin{equation}
\xymatrix{F^*(K)\ar[d]^{\varphi}\ar[rr]^{F^*(\tilde{\theta}(w))}&&F^*(K)\ar[d]^{\varphi}\\
K\ar[rr]^{\tilde{\theta}(w)}&&K}
\label{compat}\end{equation}

Then $\varphi$ restricts to an $F$-equivariant structure

$$
\varphi(\chi):F^*(K(\chi))\rightarrow K(\chi)
$$
for all $\chi\in \widehat{W}$ and we have

$$
{\bf X}_{K,\varphi}=\sum_{\chi\in \widehat{W}}{\bf X}_{K(\chi),\varphi(\chi)}.
$$

\begin{remark}If $W$ is abelian then

\begin{equation}
{\bf X}_{K,\tilde{\theta}(w)\circ\varphi}=\sum_{\chi\in\widehat{W}}\chi(w)\, {\bf X}_{K(\chi),\varphi(\chi)}
\label{W-AB}\end{equation}
and by the orthogonality relation we have

\begin{equation}
{\bf X}_{K(\chi),\varphi(\chi)}=\frac{1}{|W|}\sum_{w\in W}\chi(w){\bf X}_{K,\tilde{\theta}(w)\circ\varphi}.
\label{invert}\end{equation}
\end{remark}

\subsection{Equivariant category and quotient stacks}

Given a finite group $W$ acting on the right on a variety $X$, we can reformulate the definition and properties of $W$-equivariant objects on $X$ in the language of quotient stacks. Let $[X/W]$ be the quotient stack of $X$ by $W$.
\bigskip

We denote by $\pi^X_W:X \to [X/W]$ the canonical projection map. Recall that this map is a Galois covering with  Galois group $W$. If $X=\spec(K)$, we put $B(W)\coloneqq[\spec(K)/W]$ for the classifying space of $W$-torsors and we put simply $\pi_W:\spec(K) \to B(W)$.

\bigskip

The pullback $(\pi^X_W)^*$ induces equivalences of categories 

$$\Perv(X,W) \cong \Perv([X/W]),$$ $$D^b_c(X,W) \cong D^b_c([X/W]).$$

More generally, for any subgroup $W' \subseteq W$, we have a canonical map 

$$\pi^X_{W',W}:[X/W'] \to [X/W].$$

\bigskip

\begin{esempio}
\label{classifying}
If $X=\spec(K)$, a $W$-equivariant perverse sheaf is a perverse sheaf on the point $\spec(K)$ with an action on $W$. A perverse sheaf on $\spec(K)$ is a finite-dimensional $\kappa$-vector space and we have thus an equivalence of categories 

\begin{equation}\label{equivalence-categories}
\Perv(B(W))=\Perv(\spec(K),W)=\Rep_{\kappa}(W) .\end{equation} Therefore, for any irreducible character $\chi \in \widehat{W}$, we have an associated irreducible local system $\mathcal{L}^{B(W)}_{\chi}$ on $B(W)$. 

For any subgroup $W' \subseteq W$ and the corresponding map $\pi_{W',W}:B(W') \to B(W)$, through the identifications \ref{equivalence-categories}), the functor $(\pi_{W',W})_*:\Perv(B(W')) \to \Perv(B(W))$ corresponds to the functor $\Ind_{W'}^W:\Rep_{\kappa}(W') \to \Rep_{\kappa}(W)$.

Since $\Ind_{\{e\}}^W(1)=\kappa[W]$, we have an isomorphism of perverse sheaves on $B(W)$ 

\begin{equation}\label{decomposition-BW}
(\pi_{W})_*(\kappa) \cong \bigoplus_{\chi \in \widehat{W}}V_\chi\otimes \mathcal{L}^{B(W)}_{\chi}   
\end{equation}
\end{esempio}
where $V_\chi$ denotes an irreducible $\kappa$-module affording the character $\chi$ and $\mathcal{L}^{B(W)}_\chi$ the irreducible local system on $B(W)$ introduced above.

Moreover

$$
\left((\pi_W)_*(\kappa)\right)(\chi)=V_\chi\otimes \mathcal{L}^{B(W)}_\chi.
$$
\bigskip

Assume given another variety $Y$ with a $W$-action and a $W$-equivariant morphism $f:X \to Y$. Then we have a canonical morphism $\overline{f}:[X/W] \to [Y/W]$. In particular, for any $X$, we have a canonical morphism $\psi_X:[X/W] \to B(W)$. For any $\chi \in \widehat{W}$, we put 

\begin{equation}\label{Lchi}
\mathcal{L}^{[X/W]}_\chi\coloneqq \psi_X^*(\mathcal{L}^{B(W)}_{\chi})
\end{equation}
which is a local system on $[X/W]$.

Notice that, more precisely, forgetting the equivariant structure, $\mathcal{L}^{[X/W]}_\chi$ is a constant sheaf on $X$ of rank $\deg(\chi)$. It is however irreducible as an object in $\Perv(X,W)$.

\bigskip

Since the diagram $$\xymatrix{X\ar[d]^{\pi_W^X}\ar[rr]^{}&&\spec(K)\ar[d]^{\pi_W}\\[X/W]\ar[rr]^{\psi_X}&&B(W)}$$ is cartesian, from the proper base change theorem, we have an isomorphism  \begin{equation}
    \label{decomposition-equivariant}
    (\pi_W^X)_*(\kappa) \cong \bigoplus_{\chi \in \widehat{W}}V_\chi\otimes \mathcal{L}^{[X/W]}_\chi.
\end{equation}

Finally, if $(K,\theta)\in D_c^b(X,W)$ with corresponding complex $\overline{K}\in D_c^b ([X/W])$, then by the projection formula

\begin{equation}\label{pushequiv}
(\pi_W^X)_*(K,\theta)=\bigoplus_{\chi\in\widehat{W}}V_\chi\otimes K^{[X/W]}_\chi   
\end{equation}
where $K^{[X/W]}_\chi:=\overline{K}\otimes \mathcal{L}^{[X/W]}_\chi$.
\bigskip

\subsection{Mackey formula for finite quotient stacks}
\label{finitequotientstacks}
Fix a subgroup $W' \subseteq W$ and a simple perverse sheaf $\mathcal{F}$ on $X$ equipped with a $W'$-equivariant structure. Assume also that for any $w \in W \setminus W'$, we have  $w^*(\mathcal{F}) \not \cong \mathcal{F}$. 

The aim of this section is to prove the following result.

\begin{prop} We have a natural decomposition
$$
(\pi^X_W)_*(\mathcal{F})=\bigoplus_{\chi\in\widehat{W'}}V_\chi\otimes \mathcal{F}^{[X/W]}_\chi
$$
 for some distinct simple perverse sheaves $\mathcal{F}^{[X/W]}_\chi$ indexed by the irreducible characters of $W'$. 
\end{prop}
\bigskip

We have 

$$
(\pi^X_W)_{*}(\mathcal{F})=(\pi^X_{W',W})_*((\pi^X_{W'})_{*}(\mathcal{F})).
$$
Since $\mathcal{F}$ is $W'$-equivariant, we can apply Formula (\ref{pushequiv}) and we have \begin{equation}
(\pi^X_{W'})_{*}(\mathcal{F})= \bigoplus_{\chi \in \widehat{W'}}V_\chi\otimes \mathcal{F}^{[X/W']}_\chi   \end{equation} 
and so

\begin{equation}
(\pi^X_W)_{*}(\mathcal{F}) \cong \bigoplus_{\chi \in \widehat{W'}}V_\chi\otimes \mathcal{F}^{[X/W]}_\chi    
\end{equation}
where $\mathcal{F}^{[X/W]}_\chi:=(\pi^X_{W',W})_*\left(\mathcal{F}^{[X/W']}_{\chi}\right)$.
\bigskip

\begin{lemma}
\label{irreducible-components-finitestack}
For any $\chi \in \widehat{W'}$, the perverse sheaf $\mathcal{F}^{[X/W]}_\chi$ is irreducible. Moreover, if $\chi \neq \chi'$, then

$$
\mathcal{F}^{[X/W]}_\chi\not \cong \mathcal{F}^{[X/W]}_{\chi'}.
$$
\end{lemma}

\begin{proof}
 From adjunction, we have an isomorphism \begin{equation}
\label{hom0}\Hom(\mathcal{F}^{[X/W]}_\chi,\mathcal{F}^{[X/W]}_\chi) \cong \Hom\left((\pi_{W',W})^*(\mathcal{F}^{[X/W]}_\chi),\mathcal{F}^{[X/W']}_\chi\right) .\end{equation}

We now describe the perverse sheaf $(\pi^X_{W',W})^*(\mathcal{F}^{[X/W]}_\chi)$. Fix a representative $s \in W$ for every $\overline{s} \in W'\backslash W/W'$ and put $W'_s\coloneqq W' \cap s^{-1}W's$. We have a cartesian diagram $$\xymatrix{\displaystyle\bigsqcup_{\overline{s}\in W'\backslash W/W'}[X/W_s']\ar[d]^-{g}\ar[rr]^-{f}&&[X/W']\ar[d]^{\pi_{W',W}^X}\\[X/W']\ar[rr]^{\pi_{W',W}^X}&&[X/W]}\label{diagramMackey}.$$

Where the maps $f,g$ are defined as follows. We have $$f=\bigsqcup_{\overline{s} \in W'\backslash W/W'} \pi^X_{W_s',W'} .$$ For any $\overline{s}$, consider the embedding of groups $W_s' \to W'$ which sends $w \to sws^{-1}$. The latter embedding induces a morphism of quotient stacks $g_{\overline{s}}:[X/W_s'] \to [X/W']$ and we have $$g=\bigsqcup_{\overline{s} \in W'\backslash W/W'} g_{\overline{s}} .$$

By the proper base change theorem, we thus have an isomorphism \begin{equation}
(\pi^X_{W',W})^*(\mathcal{F}^{[X/W]}_\chi) \cong g_*f^*(\mathcal{F}^{[X/W']}_\chi)     
\end{equation}

and thus an isomorphism \begin{equation}
\Hom\left((\pi^X_{W',W})^*(\mathcal{F}^{[X/W]}_\chi),\mathcal{F}^{[X/W']}_\chi\right) \cong \Hom\left(g_*f^*(\mathcal{F}^{[X/W']}_\chi) ,\mathcal{F}^{[X/W']}_\chi\right).    
\end{equation}

Moreover, since $p$ is an étale morphism, from base change, $g$ is étale too and we have thus $g^*=g^{!}$. Therefore, we have \begin{equation}
    \Hom\left(g_*f^*(\mathcal{F}^{[X/W']}_\chi),\mathcal{F}^{[X/W']}_\chi\right) \cong \Hom\left(f^*(\mathcal{F}^{[X/W']}_\chi),g^*(\mathcal{F}^{[X/W']}_\chi)\right).
\end{equation}

We now describe the two perverse sheaves $f^*(\mathcal{F}^{[X/W']}_\chi)$ and $g^*(\mathcal{F}^{[X/W']}_\chi)$.

\bigskip

Notice that, an element $F$ of $$\displaystyle\Perv\left(\bigsqcup_{\overline{s} \in W'\backslash W/W'}[X/W_s']\right)$$ consists of  a $W'_s$-equivariant perverse sheaf $F_{\overline{s}} \in \Perv([X/W_s'])$ for every $\overline{s} \in W'\backslash W/W'$.

Moreover, under this correspondence, for any two such perverse sheaves $F,F'$  we have \begin{equation}
\label{hom}
    \Hom(F,F') \cong \bigoplus_{\overline{s}\in W'\backslash W/W'}\Hom_{[X/W_s']}(F_{\overline{s}},F'_{\overline{s}}).
\end{equation}
    
\bigskip

We now describe more generally the functors $f^*,g^*$. Given a perverse sheaf $\overline{K}=(K,\theta) \in \Perv([X/W'])$, under the correspondence introduced above, for every $\overline{s} \in W'\backslash W/W'$, we have  $$f^*(\overline{K})_{\overline{s}}=(\pi^X_{W'_s,W'})^*(\overline{K}).$$

\vspace{2 pt}

To describe $g^*$, notice that, for each $\overline{s} \in W'\backslash W/W'$,  the perverse sheaf $(s^{-1})^*(K)$ can be endowed with the $W_s'$-equivariant structure $\theta^{\overline{s}}=(s^{-1})^*(\theta)$, i.e. for $w\in W'_s$ of the form $w=s^{-1}w's$ with $w'\in W'$, we have

$$\theta^{\overline{s}}= (s^{-1})^*(\theta_{w'}) .$$
Then for every $\overline{s} \in W'\backslash W/W'$, we have $$g^{*}(\overline{K})_{\overline{s}}= ((s^{-1})^*K,\theta^{\overline{s}}).$$

If $\overline{K}=\mathcal{F}^{[X/W']}_\chi$, then, by definition of $\mathcal{F}_\chi^{[X/W']}$ (see the end of the previous section), we have$$g^*(\overline{K})_{\overline{s}}=g^*(\overline{\mathcal{F}})_{\overline{s}}\otimes g^*(\mathcal{L}^{[X/W']}_{\chi})_{\overline{s}}$$
where $\overline{\mathcal{F}}$ is the object of $\Perv([X/W'])$ corresponding to $\mathcal{F}$ equipped with its $W'$-equivariant structure.
\bigskip

From Formula (\ref{hom}), we deduce that

\begin{align*}
\Hom\left(f^*(\mathcal{F}^{[X/W']}_\chi)\right.,&\left.g^*(\mathcal{F}^{[X/W']}_\chi)\right)\\ &=\bigoplus_{\overline{s}\in W'\backslash W/W'} \Hom\left((\pi^X_{W'_s,W'})^*(\mathcal{F}^{[X/W']}_\chi),g^*(\overline{\mathcal{F}})_{\overline{s}} \otimes g^*(\mathcal{L}^{[X/W']}_{\chi})_{\overline{s}}\right).    
\end{align*}

Notice that, for any $\overline{s} \in W'\backslash W/W'$, we have that 
$$ (\pi_{W_s'}^X)^*((\pi^X_{W'_s,W'})^*(\mathcal{F}^{[X/W']}_\chi))=\mathcal{F}^{\deg(\chi)}$$ and $$(\pi_{W_s'}^X)^*(g^*(\overline{\mathcal{F}})_{\overline{s}} \otimes g^*(\mathcal{L}^{[X/W']}_{\chi})_{\overline{s}})=(s^{-1})^*(\mathcal{F})^{\deg(\chi)} .$$

Unraveling the definitions, we see that, for each $\overline{s}\in W'\backslash W/W'$, we have an inclusion \begin{equation}
\label{hom2}
\Hom\left((\pi^X_{W'_s,W'})^*(\mathcal{F}^{[X/W']}_\chi),g^*(\overline{\mathcal{F}})_{\overline{s}} \otimes g^*(\mathcal{L}^{[X/W']}_{\chi})_{\overline{s}} \right) \subseteq \Hom\left(\mathcal{F}^{\deg(\chi)},(s^{-1})^*(\mathcal{F})^{\deg(\chi)}\right)    
\end{equation}

For every $\overline{s} \in W'\backslash W/W'$ such that $s \notin W'$, we have  $(s^{-1})^*(\mathcal{F}) \not \cong \mathcal{F}$. Since $\mathcal{F}$ is a simple perverse sheaf, we deduce that $\Hom(\mathcal{F}^{\deg(\chi)},(s^{-1})^*(\mathcal{F})^{\deg(\chi)})=\{0\}$ and thus we have

\begin{equation}
\Hom\left(f^*(\mathcal{F}^{[X/W']}_\chi),g^*(\mathcal{F}^{[X/W']}_\chi)\right) \cong \End\left(\mathcal{F}^{[X/W']}_\chi\right).   
\end{equation}
From Formula (\ref{hom0}), we deduce that \begin{equation}
    \End\left(\mathcal{F}^{[X/W]}_\chi\right)\cong \End\left(\mathcal{F}^{[X/W']}_\chi\right)
\end{equation}
Notice that $\End(\mathcal{F}^{[X/W']}_\chi)=\kappa$, since $\mathcal{F}^{[X/W']}_\chi$ is a simple perverse sheaf on $[X/W']$. From the semisimplicity of $(\pi^X_W)_*(\mathcal{F})$, we deduce that  $\mathcal{F}^{[X/W]}_\chi$ is simple too.   

\bigskip

A similar argument shows that, if $\chi \neq \chi'$, we have 
$$\Hom\left((\pi^X_{W',W})_*(\mathcal{F}^{[X/W']}_{\chi}),(\pi^X_{W',W})_*(\mathcal{F}^{[X/W']}_{\chi'})\right)=\{0\}$$
and thus $\mathcal{F}^{[X/W]}_{\chi}$ and $\mathcal{F}^{[X/W]}_{\chi'}$ are not isomorphic.
\end{proof}
\bigskip

\subsection{Finite maps and intersection cohomology complexes}
\label{finite-maps}
Let $X,Y$ be two equidimensional varieties and let $f:X \to Y$ be a  surjective morphism. Recall that   $f$ is quasi-finite if for each $y \in Y$, the fiber $f^{-1}(y)$ is finite (in which case $\dim X =\dim Y$).

\bigskip

In the rest of the paper, all varieties $X,Y$ and finite maps $f:X \to Y$ will respect the following.
\bigskip

\begin{assum}
\label{assumptionmap}
\begin{itemize}
    \item The varieties $X,Y$ are equidimensional and have the same number of irreducible components.
    \item Given the decomposition into irreducible components $X=X_1 \cup \cdots \cup X_r$ and $Y=Y_1 \cup \cdots \cup Y_r$, we have $f(X_i)=Y_i$ and $f:X_i \to Y_i$ is a finite surjective map for each $i=1,\dots,r$.
\item There exists a finite abelian group $A$ that acts on $X$ and such that $f$ is $A$-invariant.

\item  There exists a smooth open subset $U \subseteq Y$ such that the restriction $f:f^{-1}(U) \to U$ is an $A$-covering.

\end{itemize}

\end{assum}

\bigskip

In this situation, consider an $A$-equivariant local system $(\mathcal{E},\theta)$ over $f^{-1}(U)$. The complex $\IC^{\bullet}_{X,\mathcal{E}}$ is naturally equipped with an  $A$-equivariant structure. As the map $f$ is $A$-invariant,  the local system $f_*(\mathcal{E})$ and the complex $\IC^{\bullet}_{Y,f_*(\mathcal{E})}$) are both equipped with an action of $A$. 

As the restriction of $f$ to $f^{-1}(U)$ is Galois we have

\bigskip

$$f_*(\mathcal{E})=\displaystyle\bigoplus_{\chi \in \widehat{A}}V_\chi\otimes \mathcal{E}_{\chi}^U$$
by Formula (\ref{pushequiv}) and so $$\IC^{\bullet}_{Y,f_*(\mathcal{E})}=\bigoplus_{\chi \in \widehat{A}}V_\chi\otimes \IC^{\bullet}_{Y,\mathcal{E}_{\chi}^U}.$$

\bigskip
We have the following.

\begin{lemma}
\label{corollary-small-map}
Let $f:X \to Y$ be a finite map satisfying  Assumption \ref{assumptionmap} and let $\mathcal{E}$ be an $A$-equivariant local system on $f^{-1}(U)$. We have an isomorphism $$f_*(\IC^\bullet_{X,\mathcal{E}}) \cong \IC^\bullet_{Y,f_*(\mathcal{E})}$$ which respects  the action of $A$.

In particular, for any $\chi \in \widehat{A}$, we have isomorphisms of perverse sheaves \begin{equation}\label{equiv-localsytem-final}
\left(f_*(\IC^\bullet_{X,\mathcal{E}})\right)(\chi) \cong (\IC^\bullet_{Y,f_*\mathcal{E}})(\chi) \cong V_\chi\otimes\IC^\bullet_{Y,\mathcal{E}^U_\chi}\cong \IC^\bullet_{Y,\mathcal{E}^U_\chi}.  
\end{equation}

\end{lemma}

The last isomorphism in (\ref{equiv-localsytem-final}) is resulting from the fact that $A$ is abelian and so $V_\chi$ is one-dimensional.

\bigskip







Assume now that $K=\overline{\F}_q$ and that $X,Y$ are equipped with corresponding geometric Frobenius $F$ which commute with $f$ and with the action of $A$. Let $(\mathcal{E},\theta)$ be an $A$-equivariant local system on $f^{-1}(U)$ equipped  with an $F$-equivariant structure $\varphi:F^*(\mathcal{E})\simeq\mathcal{E}$ such that the following diagram commutes for all $\zeta\in A$

\begin{equation}
\xymatrix{\zeta^*F^*(\mathcal{E})\ar[d]_{\zeta^*(\varphi)}\ar[rr]^{F^*(\theta_\zeta)}&&F^*(\mathcal{E})\ar[d]^\varphi\\
\zeta^*(\mathcal{E})\ar[rr]^{\theta_\zeta}&&\mathcal{E}}
\label{W}\end{equation}
Denote again by $\varphi:F^*(f_*(\mathcal{E}))\simeq f_*(\mathcal{E})$ the $F$-equivariant structure on $f_*(\mathcal{E})$ induced by $\varphi:F^*(\mathcal{E})\simeq\mathcal{E}$. 

Then Diagram (\ref{compat}), with $K=f_*(\mathcal{E})$ and $\tilde{\theta}=f_*(\theta)$, commutes.

The same diagrams with $f_*(\IC^\bullet_{X,\mathcal{E}})=\IC^\bullet_{Y,f_*(\mathcal{E})}$ instead of $f_*(\mathcal{E})$ also commute.
\bigskip

By Formula (\ref{invert}) combined with (\ref{equiv-localsytem-final}) we have

\begin{equation}
{\bf X}_{\IC^\bullet_{Y,\mathcal{E}^U_\chi},\varphi_\chi}=\frac{1}{|A|}\sum_{\zeta\in A}\chi(\zeta)\,{\bf X}_{f_*(\IC^\bullet_{X,\mathcal{E}}),\tilde{\theta}(\zeta)\circ\varphi}
\label{theta}\end{equation}
where $\varphi_\chi$ is the $F$-equivariant structure $\varphi(\chi)$ on $f_*(\IC^\bullet_{X,\mathcal{E}})(\chi)\cong\IC^\bullet_{Y,\mathcal{E}^U_\chi}$.

For each $\zeta \in A$ we have a $\zeta F$-equivariant structure on $\mathcal{E}$

$$
\varphi_\zeta:(\zeta F)^*(\mathcal{E})\rightarrow\mathcal{E}
$$
obtained from Diagram (\ref{W}) by composing $\varphi$ with $F^*(\theta_\zeta)$. Then

\begin{equation}
{\bf X}_{f_*(\IC^\bullet_{X,\mathcal{E}}),\tilde{\theta}(\zeta)\circ\varphi}=(f^{\zeta F})_*({\bf X}_{\IC^\bullet_{X,\mathcal{E}},\varphi_ \zeta})
\label{equa-wF}\end{equation}
where $f^{\zeta F}:X^{\zeta F}\rightarrow Y^F$.

\subsection{Twisted intersection cohomology polynomials}\label{twisted}

Assume that $X$ is a $K$-variety endowed with an action of a finite group $W$. If $K=\overline{\F}_q$ and $F:X\rightarrow X$ is a geometric Frobenius, we assume that the action of $W$ commutes with $F$.

\bigskip

Recall that, for each $w \in W$, we have a canonical isomorphism $$\alpha_w:w^*(\IC^\bullet_X) \to \IC^\bullet_X ,$$ from which we get an action of $W$ on the intersection cohomology groups $IH^{\bullet}_c(X)$. The group action preserves the weight filtration on $IH_c^\bullet(X)$ and for $w\in W$, we define the $w$-twisted mixed Poincar\'e polynomial for the intersection cohomology as

$$
IH_c^w(X;q,t):=\sum_{i,k}{\rm Tr}\left(w\mid W^k_i/W^k_{i-1}\right)\, q^{i/2}t^k.
$$

If $X/_{\overline{\F}_q}$ and we have a geometric Frobenius $F:X \to X$, we define $\phi^w:(wF)^*\IC^\bullet_X \to \IC^\bullet_X$ as follows \begin{center}
    \begin{tikzcd}
     F^*w^*(\IC^\bullet_X)=w^*F^*(\IC^\bullet_X) \arrow[bend left=30]{rr}{\phi^{w}} \arrow[r,"w^*(\phi)"] &w^*(\IC^\bullet_X) \arrow[r,"\alpha_w"] &\IC^\bullet_X   
    \end{tikzcd}
\end{center}

Similarly, define $(\phi^n)^w:(wF^n)^*\IC^\bullet_X \to \IC^\bullet_X$ as $(\phi^n)^w \coloneqq \alpha_w \circ w^*(\phi^n)$. If $K=\overline{\F}_q,$ we say that the pair $(X/_{\overline{\F}_q}, W)$ has the IC-polynomial property  with (twisted) IC-polynomials $\{P_w(t)\}_{w\in W}$ in $\mathbb{Z}[t]$ if

$$
\sum_{x \in X^{wF^r}}\X_{\IC^\bullet_X,(\phi^r)^w}=P_w(q^r).
$$
for all integers $r \geq 1$.

We have the following twisted analogue of Theorem \ref{countingfq}.

\begin{teorema}
\label{theoremtwistedEFq}
If $(X/_{\overline{\F}_q},W)$ has the (twisted) IC-polynomial property with (twisted) IC-polynomials $\{P_w(t)\}_{w\in W}$ for any $w\in W$ we have

\begin{equation}
\label{katztwistedformula}
IE^w(X;q)=P_w(q).
\end{equation}
where $IE^{w}(X;q)\coloneqq IH^w_c(X;q,-1)$.
\end{teorema}

The proof of this Theorem is very similar to that of \cite[Theorem 2.8]{LRV}. We give it below after Theorem \ref{theoremtwistedE}.

\bigskip

\begin{oss}
Given $(X/_{\overline{\F}_q},W)$ as above, for any  $w \in W$, the map $wF:X \to X$ is a Frobenius morphism and gives thus another $\F_q$-structure of the $\overline{\F}_q$-variety $X$.

In particular, there exists an $\F_q$-scheme $X_w$ with an isomorphism $X_w \times_{\F_q} \overline{\F}_q \cong X$ such that, through this isomorphism, the geometric Frobenius $F_w$ of $X_w$ is identified with $wF$.

In general, the polynomial $P_w(t)$ is not the IC-polynomial of $X_w$ with the geometric Frobenius $F_w$.

\vspace{4 pt}

For a concrete example, consider, $X=\overline{\F}_q^*$ ($q$ odd) with $F(z)=z^q$ and $W=\bm \mu_2=\{1,\sigma\}$, with the action $\sigma \cdot x=x^{-1}$. The pair $(\overline{\F}_q^*,\bm \mu_2)$ has polynomial count with counting polynomials $P_1(t)=(t-1)$ and $P_{\sigma}(t)=(t+1)$. Indeed, we have $\sigma F(x)=x^{-q}$ and thus $X^{\sigma F}=\bm \mu_{q+1}$. Notice that, if $4$ does not divide $q+1$, i.e. if $-1$ is not a square in $\F_q^*$, we can consider $X_{\sigma}=\spec(\F_q[s,t]/(s^2+t^2=1))$, with the isomorphism $X_{\sigma} \times_{\F_q}\overline{\F}_q \to \overline{\F}_q^*$ given by $(s,t) \to s+it$.

We have that $$\#X_{\sigma}(\F_{q^r})=\begin{cases}
q^r+1 \text{ if } r \text{ is odd}\\
q^r-1 \text{ if } r \text{ is even}
\end{cases} $$
\end{oss}

\bigskip

If $K=\C$, we say that $(X/_{\C},W)$ has the (twisted) IC-polynomial property with (twisted) IC-polynomials $\{P_w(t)\}_{w\in W}$  if there exists a finitely generated $\mathbb{Z}$-subalgebra $R$ of $\C$, a separated  $R$-scheme $X_R$ equipped with a $W$-action which gives back $X$ with its $W$-action after scalar extensions from $R$ to $\C$, such that for any ring homomorphism $f: R \to \mathbb{F}_q$,  $(X^f,W)$ has the (twisted) IC-polynomial property with (twisted) IC-polynomials $\{P_{w}(t)\}_{w \in W}$.

\vspace{2 pt}

\bigskip

We have the following twisted version of Theorem \ref{Katz}.

\begin{teorema}
\label{theoremtwistedE}
Assume that $(X/_{\C},W)$ has the (twisted) IC-polynomial property with (twisted) IC-polynomials $\{P_w(t)\}_{w\in W}$. Then for any $w\in W$ we have

\begin{equation}
\label{katztwistedformulaC}
IE^w(X;q)=P_w(q).
\end{equation}
\end{teorema}

\vspace{2 pt}

Theorem \ref{theoremtwistedE} above can be deduced from Theorem \ref{theoremtwistedEFq} as follows. Consider a variety $X/_{\C}$ and $R$ as above and $U \subseteq \spec(R)$ as in Theorem \ref{theorem-ICextension}. 

Take a $W$-equivariant resolution of singularities $\widetilde{X} \to X$ defined over $V \subseteq \spec(R)$ as in the proof of Theorem \ref{theorem-ICextension}. From the $W$-equivariance of $\widetilde{X} \to X$, we see that all the constructions in the proof of Theorem \ref{theorem-ICextension} are $W$-equivariant and, in particular, the isomorphism (\ref{comparisonIH}) commute with the $W$-action on both sides.

Moreover, since the action of $W$ on $IH_c^{\bullet}(X)$ is defined over the rationals, through the isomorphism (\ref{comparisonIH}) , we have $$\tr\left(\,w\mid W_{i}^kIH^k_c(X)/W^k_{i-1}IH^k_c(X)\right)=\tr\left(w\mid W_{i}^kIH^k_c(X^f)/W^k_{i-1}IH^k_c(X^f)\right)$$ and so $$IH^{w}_c(X;q,t)=IH^{w}_c(X^f;q,t) $$ from which we get $$IE^w(X;q)=IE^w(X^f;q) .$$

\begin{proof}[Proof of Theorem \ref{theoremtwistedEFq}]
\label{prooftwistedEFq}
By the trace formula, for any $r$, we have \begin{equation}
\label{grothendiecktrace}
    P_w(q^r)=\sum_{x \in X^{wF^r}}\X_{\IC^\bullet_X,(\phi^r)^w}=\sum_{k}(-1)^k \tr\left(wF^r\,|\,IH^k_c(X,\overline{\Q}_{\ell})\right) .\end{equation}

Let $\lambda_{i,k,1}q^{\frac{i}{2}},\dots,\lambda_{i,k,s_{k,i}}q^{\frac{i}{2}}$ and $\alpha^w_{i,k,1},\dots,\alpha^w_{i,k,s_{k,i}}$ be the eigenvalues, counted with multiplicities, of $F$ and $w$ on $W^k_i/W^k_{i-1}$. Since $w$ and $F$ commute, up to reordering, we can assume that, for any $r \geq 1$, 

$$\tr\left(wF^r\,|\,W^k_{i}/W^k_{i-1}\right)=\sum_{h=1}^{s_{k,i}}\alpha^w_{i,k,h}(\lambda_{i,k,h})^rq^{\frac{ri}{2}}$$ and thus $$\sum_{k}(-1)^k \tr\left(wF^r\,|\,IH^k_c(X)\right)=\sum_{i}\left(\sum_k (-1)^k \sum_{h=1}^{s_{k,i}}\alpha^w_{i,k,h}(\lambda_{i,k,h})^r\right)q^{\frac{ir}{2}}  .$$ 

If $\displaystyle P_w(t)=\sum_{i}t^i c_{w,i}$, from Formula (\ref{grothendiecktrace}) we deduce that 
\begin{equation}
\sum_k (-1)^k \sum_{h=1}^{s_{k,i}}\alpha^w_{i,k,h}(\lambda_{i,k,h})^r=\begin{cases}
c_{w,\frac{i}{2}} \text{ if } i \text{ is even }\\
0 \text{ otherwise}
\end{cases}.
\end{equation}

From \cite[Lemma 2.9]{LRV}, we deduce that we have $$\sum_k (-1)^k \sum_{h=1}^{s_{k,i}}\alpha^w_{i,k,h}=\begin{cases}
c_{w,\frac{i}{2}} \text{ if } i \text{ is even }\\
0 \text{ otherwise}
\end{cases}.$$ Since $\displaystyle \sum_k (-1)^k \sum_{h=1}^{s_{k,i}}\alpha^w_{i,k,h}=\sum_k (-1)^k \tr(w\mid W^{k}_{i}/W^k_{i-1})$, we have the desired equality (\ref{katztwistedformula}).

\end{proof}

\section{Partitions, types and conjugacy classes of $\mathrm{GL}_n$}

\subsection{Partition and types}\label{type}

Let $\mathcal{P}$ be the set of all partitions and $\mathcal{P}^* \subseteq \mathcal{P}$ the subset of nonzero partitions. A partition $\lambda$ will be denoted by $\lambda=(\lambda_1,\lambda_2 \dots ,\lambda_h)$ with $\lambda_1 \geq \lambda_2 \geq \dots \geq \lambda_h$ or  by $\lambda=(1^{m_1},2^{m_2},\dots )$ where $m_k$ is the number of occurrances of the number $k$ in the partition $\lambda$. We will denote by $\lambda'$ the partition conjugate to $\lambda$.

The \textit{size} of $\lambda$ is $\displaystyle |\lambda|=\sum_{i} \lambda_i$ and its length $l(\lambda)$ is the largest $i$ such that $\lambda_i \neq 0$. For each $n \in \N$, we denote by $\mathcal{P}_n$ the subset of partitions of size $n$. We consider the dominance ordering on $\mathcal{P}$. Say that $\lambda\unlhd \mu$ if, for any $i$, we have $$\sum_{j=1}^i \lambda_i \leq \sum_{j=1}^i \mu_i .$$
\bigskip

A \emph{type} $\omega$ is a function $\omega:\N_{>0} \times \mathcal{P}^* \to \N$ with finite support.

It will be then convenient to write $\omega$ as

$$
\omega=\{(d_i,\omega^i)^{m_i}\}$$
where $m_i\neq 0$ is the image of $(d_i,\omega^i)\in\N_{>0} \times \mathcal{P}^* $ by the function $\omega$.

The size $|\omega|$ of a type $\omega=\{(d_i,\omega^i)^{m_i}\}$ is defined as $$|\omega|=\sum_im_id_i|\omega^i|.$$
We denote by $\mathbb{T}$ the set of types and by $\mathbb{T}_n$ the subset of types of size $n$. We let $\mathbb{T}^{\circ}_n$ be the subset of $\mathbb{T}_n$ of the types of the form $\omega=\{(1,\omega^i)^{m_i}\}$. 

For any $d \in \N_{>0}$ and $\omega=\{(d_i,\omega^i)^{m_i}\} \in \mathbb{T}$, define $$\psi_d(\omega)\coloneqq \{(dd_i,\omega^i)^{m_i}\}.$$

For a positive integer $s \in \N_{>0}$ and $\omega \in \mathbb{T}$, we define $$s \omega :\N_{>0} \times \mathcal{P}^* \to \N $$ $$(d,\lambda) \to s\omega((d,\lambda)) .$$

Given $\omega=\{(d_i,\omega^i)^{m_i}\}$, denote by $\omega'\coloneqq \{(d_i,\omega^{i}{'})^{m_i}\}$ its dual.

\bigskip

\begin{esempio}
The dual type of $\omega=\{(d_i,(1^{n_i}))^{m_i}\}$, is $\omega'=\{(d_i,(n_i))^{m_i}\}$.    
\end{esempio}

\subsection{Conjugacy classes of $\mathrm{GL}_n(K)$ and types}

We start by fixing the following notation. Given $z \in K^*$ and $m \in \N$, denote by $J(z,m)$ the upper Jordan triangular matrix of size $m$ and having $z$ on the diagonal entries. 

 Recall that $\mathcal{P}_n$ is in bijection with the unipotent conjugacy classes in the following way. To $\lambda \in \mathcal{P}_n$, we associate the block diagonal matrix $J(1,\lambda)$ having blocks on the diagonal $(J(1,\lambda_1),\cdots ,J(1,\lambda_h))$. Similarly, for any $z \in K^*$, we denote by $J(z,\lambda)$ the block diagonal matrix having blocks on the diagonal $(J(z,\lambda_1),\cdots ,J(z,\lambda_h))$.

Let $\mathcal{P}(K)$ be the set of maps $f:K^* \to \mathcal{P}$ with finite support. Given $f \in \mathcal{P}(K)$ put $$|f|\coloneqq \sum_{z \in K^*} |f(z)| ,$$ the \textit{size} of a function $f$.
We denote by $\mathcal{P}_n(K) \subseteq \mathcal{P}(K)$ the subset of functions of size $n$. The set $\mathcal{P}_n(K)$ is in bijection with the conjugacy classes of $\Gl_n(K)$ in the following way.

Consider $f \in \mathcal{P}_n(K)$. Let $\Imm(f)=\{\lambda^1,\dots,\lambda^s\}$ and put $I_j=f^{-1}(\lambda^j)$ for each $j=1,\dots,s$. Let $c_j=|I_j|$  and $I_j=\{z_{j,1},\dots,z_{j,c_j}\}$ for each $j=1,\dots,s$. To $f$ we associate now the conjugacy class of the block diagonal matrix $M_f$ having diagonal blocks $$M_f=(J(z_{1,1},\lambda^1),\dots,J(z_{1,c_1},\lambda^1),\dots,J(z_{s,1},\lambda^s),\dots,J(z_{s,c_s},\lambda^s)) .$$



\bigskip

To a conjugacy class $C \subseteq \Gl_n$ with associated function $f \in \mathcal{P}_n(K)$, we associate the following type $\omega_C \in \mathbb{T}^{\circ}_n$ defined as  $$\omega_C((1,\lambda)) \coloneqq |f^{-1}(\lambda)| .$$

\subsection{Conjugacy classes of $\mathrm{GL}_n(\mathbb{F}_q)$ and types}
\label{typesFq}
Recall that the conjugacy classes of $\Gl_n(\F_q)$ are parametrized by types in the following way, see for instance \cite[Paragraph 6.8]{letellier2}. We start by fixing the following notation. For an element $z \in \overline{\F}_q^*$, we let $d_z$ be the size of the $F$-orbit $\{z,z^q,z^{q^2},\dots\}$ of $z$.

\bigskip

We have a bijection $$\{F-\text{stable conjugacy classes of } \Gl_n(\overline{\F}_q)\} \longleftrightarrow \{\text{Conjugacy classes of }\Gl_n(\F_q)\}$$ $$C \mapsto C^F .$$

\bigskip

\begin{oss}
For a linear algebraic group $G/_{\F_q}$, it is not always true that the conjugacy classes of $G(\F_q)$ are in bijection with the $F$-stable conjugacy classes of $G(\overline{\F}_q)$. It is already not the case for $\PGl_n$, see \cref{orbital}.
\end{oss}

\bigskip

Consider thus an $F$-stable conjugacy class $C \subseteq \Gl_n(\overline{\F}_q)$ and the associated function $f \in \mathcal{P}_n(\overline{\F}_q)$. The function $f$ is then $F$-stable, i.e. $f(z)=f(F(z))$ for every $z \in \overline{\F}_q^\times$. In particular, each set $I_j$ is stable for the action of the Frobenius. 


We define the type $\omega_{C^F}$ as the function defined by $$\omega_{C^F}((d,\lambda)) \coloneqq \dfrac{|\{z \in f^{-1}(\lambda) \ | \ d_z=d\}|}{d}$$
(which is the number of Frobenius orbits of $\overline{\F}_q^\times$ of size $d$ contained in the support of $f$).

\begin{oss}
Notice that $\omega_{C^F}\in \mathbb{T}_n^{\circ}$ if and only if all the eigenvalues of $C$ are all contained in $\F_q^\times$. In this case, we will say that $C$ is \textit{split}.
\end{oss}

\subsection{An important example}\label{important}

Let $C$ be a conjugacy class of $\Gl_n(\F_q)$ with eigenvalues in $\F_q^*$ and let $f$ be the associated function $\overline{\F}_q^*\rightarrow\mathcal{P}$. 

We denote by $\omega=\{(1,\omega^i)^{m_i}\}$ the type of $C$. The multiplicity $m_i$ of $(1,\omega^i)$ equals the cardinality of $f^{-1}(\omega^i)$. 
\bigskip

Let $z\in\F_q^\times$ be such that $zC=C$. The function $f$ is thus stable by $z$, i.e.

$$
f(zg)=f(g)
$$
for any $g\in C$ and so the multiplication by $z$ permutes the eigenvalues of $C$ in each fiber $f^{-1}(\omega^i)$. The order $o(z)$ of $z$ divides thus the multiplicities $m_1,\dots,m_r$ and so $z$ is an $n$-th root of unity since 

$$
\sum_{i=1}^rm_i|\omega^i|=n.
$$
Then

\begin{equation}
\omega_{o(z)}:=\{(o(z),\omega^i)^{m_i/o(z)}\}
\label{type-important}\end{equation}
is the type of the $\Gl_n(\F_q)$-conjugacy class $O_z:=\alpha C$ where $\alpha\in\overline{\F}_q^\times$ is such that

$$
F(\alpha)=z\alpha.
$$
Indeed, if $x_1,\dots,x_{m_i}$ denotes the eigenvalues of $C$ in $f^{-1}(\omega^i)$, then the Frobenius $F$ preserves the set $\{\alpha x_1,\dots,\alpha x_{m_i}\}$ and the $F$-orbits are all of size $o(z)$.

\subsection{Zariski closure of conjugacy classes}

For two conjugacy classes $C',C$ of $\Gl_n$, we say that $C'\leq C$ if $C' \subseteq \overline{C}$. If $C$ corresponds to $f \in \mathcal{P}_n(K)$ and $C'$ to $f'\in\mathcal{P}_n(K)$ via the correspondence introduced above, we have that $C' \leq C$ if and only if $f'(z) \unlhd f(z)$ for all $z \in K^*$.

Recall moreover that the closure of each conjugacy class of $\Gl_n$ is a union of conjugacy classes, i.e. we have the following stratification $$\overline{C}=\bigsqcup_{C'\leq C} C' .$$

Notice that, for any $C$, there exists a unique closed conjugacy class $C_{ss} \leq C$, which is the conjugacy class of the semisimple part of any element of $C$. Moreover, $C_{ss} \leq C'$ for any $C'\leq C$.

\subsection{Symmetric functions}

Let $\Lambda(\mathbf{x})$ be the ring of symmetric functions over the rational functions $\Q(z,w)$ in the infinite set of variables $\mathbf{x}=\{x_1,x_2,\dots\}$. For each $n \in \N$, consider  the complete symmetric function $h_n(\mathbf{x}) \in \Lambda(\mathbf{x})$ and power sum $p_n(\mathbf{x}) \in \Lambda(\mathbf{x})$ defined as 

 \begin{align*}&h_n(\mathbf{x})=\sum_{1\leq i_1\leq i_2 \leq \dots\leq i_n}x_{i_1}x_{i_2}\cdots x_{i_n},\\&p_n(\mathbf{x})=x_1^n+x_2^n+\cdots.\end{align*}

For each $\lambda=(\lambda_1,\dots,\lambda_h) \in \mathcal{P}$, we have the corresponding symmetric functions  $$h_{\lambda}(\mathbf{x}):=h_{\lambda_1}(\mathbf{x})\cdots h_{\lambda_h}(\mathbf{x}), $$ $$p_{\lambda}(\mathbf{x}):=p_{\lambda_1}(\mathbf{x})\cdots p_{\lambda_h}(\mathbf{x}) .$$

The families of functions $\{h_{\lambda}(\mathbf{x})\}_{\lambda \in \mathcal{P}}, \{p_{\lambda}(\mathbf{x})\}_{\lambda \in \mathcal{P}}$ are both basis of the $\Q(z,w)$-vector space $\Lambda(\mathbf{x})$, or equivalently, the families of functions $\{h_n(\mathbf{x})\}_{n \in \N}, \{p_n(\mathbf{x})\}_{n \in \N}$ freely generate the ring $\Lambda(\mathbf{x})$. We define the map $\psi_d:\Lambda(\mathbf{x}) \to \Lambda(\mathbf{x})$ as the only $\Q(z,w)$-algebras morphism such that $$\psi_d(p_n(\mathbf{x}))=p_{nd}(\mathbf{x}) .$$

\vspace{8 pt}

Recall that another important basis of the ring $\Lambda(\mathbf{x})$ is given by the Schur functions $\{s_{\lambda}(\mathbf{x})\}_{\lambda \in \mathcal{P}}$. On the ring $\Lambda(\mathbf{x})$ we consider the bilinear product $\langle\,,\, \rangle$ making the Schur functions orthonormal, i.e $$\langle s_{\lambda}(\mathbf{x}),s_{\mu}(\mathbf{x})\rangle=\delta_{\lambda,\mu}.$$

For a type $\omega=\{(d_i,\omega^i)^{m_i}\} \in \mathbb{T}$, we put $$s_{\omega}(\mathbf{x})\coloneqq \prod_i \psi_{d_i}(s_{\omega^i}(\mathbf{x}))^{m_i}=\prod_is_{\omega^i}(\mathbf{x}^{d_i})^{m_i}.$$
where, for a positive integer $d$, $\mathbf{x}^d$ stands for the set of variables $\{x_1^d,x_2^d,\dots\}$.

\bigskip

We will need also the following multivariable version of the definitions above. For any $k \in \N$, consider $k$-sets of infinite variables $\mathbf{x}_1=\{x_{1,1},\dots,\},\dots,\mathbf{x}_k=\{x_{k,1},\dots,\}$ and denote by $\Lambda(\mathbf{x}_1,\dots,\mathbf{x}_k)$ the ring of functions which are symmetric in each set of variables. We have $$\Lambda(\mathbf{x}_1,\dots,\mathbf{x}_k)=\Lambda(\mathbf{x}_1)\otimes \cdots \otimes \Lambda(\mathbf{x}_k) .$$ We endow the ring $\Lambda(\mathbf{x}_1,\dots,\mathbf{x}_k)$ with the bilinear form  $$\left < f_1(\mathbf{x_1})\cdots f_k(\mathbf{x_k}),g_1(\mathbf{x_1})\cdots g_k(\mathbf{x_k})\right>=\prod_{i=1}^k \left<f_i,g_i\right> .$$

For any multitype $\bm \omega=(\omega_1,\dots,\omega_k) \in \mathbb{T}^k$, put $$s_{\bm \omega}\coloneqq s_{\omega_1}(\mathbf{x}_1)\cdots s_{\omega_k}(\mathbf{x}_k) .$$

\bigskip

In this last paragraph, we introduce the rational functions that compute the cohomology of generic $\Gl_n$-character stacks. Put $\Lambda(\mathbf{x}_1,\dots,\mathbf{x}_k)[[T]]$ the ring of formal power series. For $g \in \N$, we define the element $\Omega_g(z,w) \in \Lambda(\mathbf{x}_1,\dots,\mathbf{x}_k)[[T]]$ defined as follows $$\Omega_g(z,w) \coloneqq \sum_{\lambda \in \mathcal{P}}T^{|\lambda|}\mathcal{H}_{g,\lambda}(z,w) \prod_{i=1}^k \widetilde{H}_{\lambda}(\mathbf{x}_i,z^2,w^2) ,$$ where $\widetilde{H}_{\lambda}(\mathbf{x},z,w)$ are the (modified) Macdonald polynomials, for a definition see \cite[I.11]{garsia-haiman} and $$\mathcal{H}_{g,\lambda}(z,w)=\prod_{s \in \lambda}\dfrac{(z^{2a(s)+1}-w^{2l(s)+1})^{2g}}{(z^{2a(s)+2}-w^{2l(s)})(z^{2a(s)}-w^{2l(s)+2})} $$ is the \textit{hook function}.  

\bigskip

For any $\bm \omega=(\omega_1,\dots,\omega_k) \in \mathbb{T}_n^k$ with $\omega_i=\{(d_j,\omega^j_i)^{m_j}\}$, put $$\bm \omega'\coloneqq(\omega_1',\dots,\omega_k') \in \mathbb{T}_n^k .$$ We define the following rational function $\mathbb{H}_{g,\bm \omega}(z,w)$ $$\mathbb{H}_{g,\bm \omega}(z,w)\coloneqq (1-z^2)(w^2-1)(-1)^{r(\bm \omega)}\left\langle \Coeff_{T^n}(\Omega(z,w)),s_{\bm \omega'} \right\rangle ,$$ where $$r(\bm \omega)=k|\bm \omega|+\sum_{i,j} m_i|\omega^j_i| .$$

In the following, we will put $$\mathbb{H}_{\bm \omega}(z,w)=\mathbb{H}_{0,\bm \omega}(z,w) .$$

\section{$\mathrm{GL}_n$-character stacks}\label{chapterGln}

Unless specified, $K$ is an algebraically closed field, which for us is either $K=\overline{\F}_q$ or $\C$.  If $K=\overline{\F}_q$, we denote by $F:\Gl_n(K) \to \Gl_n(K)$ the standard Frobenius, i.e. $F((a_{i,j}))=(a_{i,j}^q)$. In the following, we recall certain geometric properties of $\Gl_n$-character stacks for Riemann surfaces and extend them to a twisted setting (for an action of a finite group). For simplicity, we focus on  the case of genus $0$, i.e. for the Riemann surface $\mathbb{P}^1_{\C}$, since this is the most relevant case for our representation theoretic results. We refer to the existing literature mentioned in this section for the case of higher genus : indeed the new results on twisted versions extend easily to higher genus as the non-trivial part in these twisted versions come from the conjugacy classes.

\subsection{Definition}

Given a $k$-tuple  $\bm C=(C_1,\dots,C_k)$ of conjugacy classes of $\Gl_n(K)$, we define the following affine variety $$X_{\overline{\bm C}}\coloneqq\{(X_1,\dots,X_k) \in \overline{C_1} \times \cdots \times \overline{C_k} \ | \ X_1\cdots X_k=I_n\} .$$

\bigskip
\begin{remark}Assume that $K=\mathbb{C}$. For a subset $D \subseteq \mathbb{P}^1_{\C}$ with $D=\{d_1,\dots,d_k\}$, we can identify $X_{\overline{\bm C}}$ with the variety of representations of $\pi_1(\mathbb{P}^1_{\C}\setminus D)$ or, equivalently, local systems on $\mathbb{P}^1_{\C} \setminus D$, such that the local monodromy around each $d_i$ belongs to the Zariski closure $\overline{C}_i$.
\end{remark}
\bigskip

Inside $X_{\overline{\bm C}}$ we have the (possibly empty) open subvariety $$X_{\bm C}\coloneqq\{(X_1,\dots,X_k) \in C_1 \times \cdots \times C_k \ | \ X_1\cdots X_k=I_n\} .$$

For a $k$-tuple $\bm C'=(C'_1,\dots,C'_k)$, say that $\bm C'\leq \bm C$ if $C'_i \leq C_i$ for each $i$. For any $k$-tuple $\bm C'\leq \bm C$, we have an inclusion $X_{\bm C'} \subseteq X_{\bm C}$ and, more generally, a decomposition into locally closed subsets $$X_{\overline{\bm C}}=\bigsqcup_{\bm C'\leq \bm C} X_{\bm C'} .$$

We consider the following character stack  $$\mathcal{M}_{\overline{\bm C}}\coloneqq [X_{\overline{\bm C}}/\PGl_n(K)] ,$$ its open substack $$\mathcal{M}_{\bm C} \coloneqq [X_{\bm C}/\PGl_n(K)] $$ and the corresponding decomposition into locally closed substacks $$\mathcal{M}_{\overline{\bm C}}=\bigsqcup_{\bm C'\leq \bm C}\mathcal{M}_{\bm C'} .$$

\subsection{Review on cohomology of generic character stacks}
\label{generic-GLn-paragraph}
In \cite[Definition 2.1.1]{HA}, the authors give the following definition of a generic $k$-tuple $\bm C$.

\begin{definizione}
\label{definition-genericity-GLn}
A $k$-tuple $\bm C=(C_1,\dots,C_k)$ of  conjugacy classes of $\Gl_n(K)$ is said be \emph{generic} if the two following conditions are satisfied :
\bigskip

(1) \begin{equation}\prod_{i=1}^k\det(C_i)=1.\label{prodet}\end{equation}

(2) If, for any $0<r<n$, we select $r$ eigenvalues of $C_i$ (for each $i$), then the product of the $kr$ selected eigenvalues is different from $1$. 

\end{definizione}
\bigskip

\begin{remark} Equation (\ref{prodet}) is necessary to have $X_{\bm C}\neq\emptyset$. 
\bigskip

The notion of generic $k$-tuples of conjugacy classes has been extended to any reductive group in \cite[Section 3.1]{KNWG}.
\end{remark}
\bigskip

Notice also that the $k$-tuple $\bm C$ is generic if and only if $\bm C^{ss}=(C_1^{ss},\dots,C_k^{ss})$ is generic. 

We thus have the following result.

\begin{lemma}
For two $k$-tuples $\bm C',\bm C$ such that $\bm C' \leq \bm C$, the $k$-tuple $\bm C'$ is generic if and only if $\bm C$ is generic.    
\end{lemma}

\bigskip

Recall the following result, see \cite[Proposition 3.4]{L}.

\begin{lemma}
If the characteristic of $K$ is $0$ or large enough, for any $\bm\omega=(\omega_1,\dots,\omega_k) \in (\mathbb{T}^{\circ}_n)^k$, there exists a generic $k$-tuple $\bm C=(C_1,\dots,C_k)$ of conjugacy classes of $\Gl_n(K)$ of type $\bm{\omega}$, i.e. $C_i$ is of type $\omega_i$ for all $i$.
\end{lemma}

In the following, we fix a generic $k$-tuple $\bm C$. We have the following results describing the geometry of the character stack $\mathcal{M}_{\overline{\bm C}}$, see \cite[Theorem 3.5, 3.8]{L}.

\begin{teorema}
\label{GLncharstacksdescription}
Assume that $X_{\overline{\bm C}} \neq \emptyset$.

    (i) The stack $\mathcal{M}_{\overline{\bm C}}$ is an affine variety (i.e. the canonical map $\mathcal{M}_{\overline{\bm C}} \to M_{\overline{\bm C}}$ is an isomorphism) which is irreducible and of dimension $$d_{\bm C}=-2n^2+2 +\sum_{i=1}^k \dim(C_i) .$$
    
     (ii) $\mathcal{M}_{\bm C}$ is a dense open and smooth subvariety of $\mathcal{M}_{\overline{\bm C}}$ (it is in particular non-empty).

\end{teorema}

\begin{oss}
\label{remark-type-Gln}
A combinatorial criterion for the non-emptyness of $X_{\overline{\bm C}}$ can be found in \cite[Section 3.2]{L}. This criterion depends only on the type of $\bm C$ and not on  the eigenvalues of the conjugacy classes $C_1,\dots,C_k$.
\end{oss}

\bigskip

In what follows, we let ${\bm \omega}$ be the type of ${\bm C}$ and we put 

$$
d_{\bm \omega}:=d_{\bm C}/2.
$$ 
We have the following results concerning the cohomology of generic $\Gl_n(K)$-character stacks.

\begin{teorema}\cite[Theorem 4.8]{L} The stack $\mathcal{M}_{\overline{\bm{C}}}$ has the IC-polynomial property and 
\begin{equation}
    IE(\mathcal{M}_{\overline{\bm C}};q)=q^{d_{\bm \omega}}\mathbb{H}_{\bm \omega}\left(\sqrt{q},\dfrac{1}{\sqrt{q}}\right).
\end{equation}   
\label{E-polynomial-GLn}\end{teorema}

We also have the following conjectural formula for the mixed Poincar\'e polynomial of $\mathcal{M}_{\overline{\bm C}}$.

\begin{conjecture}\cite[Conjecture 4.5]{L}
\label{mixed-polynomial-Gln} 
We have \begin{equation}
    IH_c(\mathcal{M}_{\overline{\bm C}};q,t)=(qt^2)^{d_{\bm \omega}}\mathbb{H}_{\bm \omega}\left(-t\sqrt{q},\dfrac{1}{\sqrt{q}}\right).
\end{equation}
\end{conjecture}

The conjecture is known in some cases when $\mathcal{M}_{\overline{\bm{C}}}$ is a surface (see \cite[\S 1.5.3]{HA} and \cite[Section 7]{L} for details).

This conjectural formula is true after the specialisation $t\mapsto -1$ by Theorem \ref{E-polynomial-GLn} (see \cite[Theorem 1.2.3]{HA} in the semisimple case). 

It is also proved after the specialisation $q\mapsto 1$ which gives the Poincar\'e series. In the case of semisimple conjugacy classes this is due to by A. Mellit \cite{Mellit} who followed a strategy used by O. Schiffmann to compute the Poincar\'e polynomial of the moduli space of semistable Higgs bundles over a smooth projective curve \cite{schiffmann}. For any conjugacy classes, this is due to M. Ballandras \cite{ballandras} who reduced the proof to the semisimple case using resolutions of some singular character varieties introduced in \cite{L}.

\subsection{Twisted mixed Poincar\'e polynomials}\label{section-twisted-mixed}

Let $\bm{C}=(C_1,\dots,C_k)$ be a generic $k$-tuple of conjugacy classes of $\Gl_n(K)$ of type $\bm{\omega}$. Let $y=(y_1,\dots,y_k)$ be a $k$-tuple of elements of $K^\times$ such that $y_iC_i=C_i$ for all $i=1,\dots,k$ and 

$$
y_1\cdots y_n=1.
$$
Notice that we have also $y_i\overline{C}_i=\overline{C}_i$ for all $i$ and so $y$ acts on $\mathcal{M}_{\overline{\bm{C}}}$ by multiplication on the coordinates. We will need to compute the twisted mixed Poincar\'e series

$$
IH_c^y(\mathcal{M}_{\overline{\bm{C}}};q,t).
$$

By the same argument as in \S\ref{important}, we see that $y_1,\dots,y_k$ are $n$-th roots of unity. Let $o(y_i)$ be the order of $y_i$, put

$$
o(y):=(o(y_1),\dots,o(y_k))
$$
and 

\begin{equation}\label{o(y)}
\bm{\omega}_{o(y)}=((\omega_1)_{o(y_1)},\dots,(\omega_k)_{o(y_k)})
\end{equation}
where $\omega_i$ is the type of $C_i$ and $(\omega_i)_{o(y_i)}$ is defined by Formula (\ref{type-important}). 
\bigskip

\begin{remark}\label{multi-important}Assume that $K=\overline{\F}_q$ with $n\mid(q-1)$ (so that $\bm\mu_n\subset\F_q^\times$) and that the eigenvalues of the conjugacy classes $C_i$ are in $\F_q^\times$. Then, as in \S \ref{important}, for each $i$, we consider the conjugacy class $(O_i)_{y_i}:=\alpha_i C_i$ of $\Gl_n(\overline{\F}_q)$ where $\alpha_i\in\overline{\F}_q^\times$ satisfies

$$
F(\alpha_i)=y_i\alpha_i.
$$
Then the $k$-tuple $((O_1)_{y_1}^F,\dots,(O_k)_{y_k}^F)$ of conjugacy classes of $\Gl_n(\F_q)$ is of type $\bm{\omega}_{o(y)}$. Notice that $\alpha_1\cdots\alpha_k\in\F_q^\times$ as $$F(\alpha_1 \cdots \alpha_k)=y_1\cdots y_k (\alpha_1 \cdots \alpha_k)=\alpha_1 \cdots \alpha_k .$$

The $(k+1)$-tuple

$$
((O_1)_{y_1},\dots,(O_k)_{y_k}, (\alpha_1\cdots\alpha_k)^{-1})$$
or equivalently the $k$-tuple

$$
\big((O_1)_{y_1},\dots,(O_{k-1})_{y_{k-1}},(\alpha_1\cdots\alpha_k)^{-1}(O_k)_{y_k}\big)
$$
is generic.

\end{remark}
\bigskip

We make the following conjecture.

\begin{conjecture}
\label{conjectureIHy}
    \begin{equation}
    IH_c^y(\mathcal{M}_{\overline{\bm C}};q,t)=(qt^2)^{d_{\bm \omega}}\mathbb{H}_{\bm{\omega}_{o(y)}}\left(-t\sqrt{q},\frac{1}{\sqrt{q}}\right).
    \label{conjectural-non-degenerate-twisted}\end{equation}
\end{conjecture}

Consider the finite group $A(C_i)=\{z\in \bm{\mu}_n\,|\, zC_i=C_i\}$ and 

$$
H(\bm{C}):=\{(z_1,\dots,z_k)\in A(C_1)\times\cdots\times A(C_k)\,|\, z_1\cdots z_k=1\}.
$$

\begin{teorema}\label{twisted-teo}The pair $(\mathcal{M}_{\overline{\bm{C}}},H(\bm{C}))$ has the (twisted) IC-polynomial property with twisted IC-polynomials

$$
\left\{q^{d_{\bm \omega}}\mathbb{H}_{\bm{\omega}_{o(y)}}\left(\sqrt{q},\frac{1}{\sqrt{q}}\right)\right\}_{y\in H(\bm{C})}.
$$
In particular, for all $y\in H(\bm{C})$

\begin{equation}
IE^y(\mathcal{M}_{\overline{\bm{C}}};q):=IH^y(\mathcal{M}_{\overline{\bm{C}}};q,-1)=q^{d_{\bm \omega}}\mathbb{H}_{\bm{\omega}_{o(y)}}\left(\sqrt{q},\frac{1}{\sqrt{q}}\right).
\label{non-degenerate-twisted}\end{equation}

\end{teorema}

The second assertion is a consequence of Theorem \ref{theoremtwistedEFq} (when $K=\overline{\mathbb{F}}_q$) and Theorem \ref{theoremtwistedE} (when $K=\mathbb{C}$). 
\bigskip

Formula (\ref{non-degenerate-twisted}) is the specialization $t\mapsto -1$ of the conjectural Formula (\ref{conjectural-non-degenerate-twisted}).

\begin{proof}
Thanks to the discussion in \S \ref{twisted}, it is enough to show Theorem \ref{twisted-teo} in the case of $K=\overline{\F}_q$. 

We thus assume that our conjugacy classes $C_1,\dots,C_k$ are in $\Gl_n(\overline{\F}_q)$ with eigenvalues in $\F_q^\times$. 

From Remark \ref{multi-important} the multi-type $\bm{\omega}_{o(y)}$
is the type of the $k$-tuple $((O_1)_{y_1}^F,\dots,(O_k)_{y_k}^F)$ of conjugacy classes of $\Gl_n(\F_q)$.

Moreover, we have the following commutative diagram $$
\xymatrix{{C}_i\ar[d]_{y_iF}\ar[rr]^{f_i}&&O_{y_i}\ar[d]^F\\
{ C}_i\ar[rr]^{f_i}&&O_{y_i}}
$$
where $f_i$ is the multiplication by the scalar $\alpha_i$.

Therefore via $\prod_i f_i$, the pair $(\mathcal{M}_{\bm C},yF)$ can be identified with $(\mathcal{M}_{\bm O},F)$ where ${ O}$ is the $(k+1)$-tuple of conjugacy classes
$$
\bm {O}=\left(({ O}_1)_{y_1},\dots,(O_k)_{y_k}, \{(\alpha_1\cdots\alpha_k)^{-1} I_n\}\right)$$
which is generic.
\bigskip

We thus have
$$
\sum_{x \in \mathcal{M}^{yF}_{\bm C}}\X_{\IC^{\bullet}_{\mathcal{M}_{\overline{\bm C}}},y\circ\phi}=\sum_{x \in \mathcal{M}^{F}_{\overline{\bm O}}}\X_{\IC^{\bullet}_{\mathcal{M}_{\overline{\bm O}}},\phi}.
$$

By \cite[Theorem 4.14]{L}, the right-hand side of the above formula equals 

$$
q^{d_{\bm\omega}}\mathbb{H}_{\bm{\omega}_{o(y)}}\left(\sqrt{q},\frac{1}{\sqrt{q}}\right),
$$
hence the result.

\end{proof}

\subsection{An explicit verification in the case of $4$ punctures}\label{k=4}

In this section, we verify explicitly Conjecture \ref{conjectureIHy} in the case where $n=2$, $k=4$ and when  the conjugacy classes $C_i$ are regular semisimple conjugacy classes. In this case, it is known that the variety $\mathcal{M}_{\bm C}$ is an affine Del Pezzo surface and that Conjecture \ref{mixed-polynomial-Gln} holds, see for instance \cite{etingof-et-al} and \cite{HA}. We quickly recall these results here. Assume for the sake of simplicity that $\det(C_i)=1$  and let  $a_i=\tr(C_i)$ for each $i$.

The character variety $\mathcal{M}_{\bm C}$ is isomorphic to the affine cubic surface $S$ defined as follows \begin{equation}
\label{affinecubicsurface}
S \coloneqq \{(x,y,z) \in \mathbb{A}^3_K \ | \ xyz+x^2+y^2+z^2-(a_1a_2+a_3a_4)x- \end{equation}
\begin{equation}
-(a_2a_3+a_1a_4)y-(a_1a_3+a_2a_4)z+a_1a_2a_3a_4+a_1^2+a_2^2+a_3^2+a_4^2-4=0 \}.\end{equation}

The isomorphism is obtained by quotienting the morphism $X_{C} \to S$ given by $$(X_1,X_2,X_3,X_4) \to (\tr(X_2X_3),\tr(X_1X_3),\tr(X_1X_2)) .$$ 
Let $\overline{S}$ be the compactification of $S$ in $\mathbb{P}^3_K$, i.e. 
\begin{equation}
 \label{projcubicsurface}
\overline{S}\coloneqq \{[x,y,z,w] \in \mathbb{P}^3_K \ | \ xyz+wx^2+wy^2+wz^2-(a_1a_2+a_3a_4)xw^2- \end{equation}
\begin{equation}
-(a_2a_3+a_1a_4)yw^2-(a_1a_3+a_2a_4)zw^2+(a_1a_2a_3a_4+a_1^2+a_2^2+a_3^2+a_4^2-4=0)w^3 \}.\end{equation}
Notice that $\overline{S}=S \bigsqcup \Delta$, where $\Delta=\{[x,y,z,w] \in \mathbb{P}^3_K \ | \ w=0, xyz=0 \}$ is a triangle at infinity. 

\bigskip

For $k=4$, $n=2$  and the $C_i$s regular semisimple,  we obtain $$\mathbb{H}_{\bm \omega}(z,w)=z^2+w^2+4 .$$ Being a smooth projective cubic surface, $\overline{S}$ can be obtained as the blowup of $\mathbb{P}^2_K$ at $6$ points. From the long sequence in cohomology for the decomposition $\overline{S}=S \bigsqcup \Delta$, we can see that $H_c(S;q,t)=q^2t^4+4qt^2+t^2$ and this proves that Conjecture \ref{mixed-polynomial-Gln} is true in this case, see \cite{HA} for more details.

\bigskip

We now verify Conjecture \ref{conjectureIHy} in certain degenerate cases for $k=4$. For the sake of simplicity, assume that $C_1,C_2,C_3$ are the conjugacy class of $$ \begin{pmatrix}
    i &0\\
    0 &-i
\end{pmatrix}$$
and $C_4$ is any (regular semisimple) non-degenerate conjugacy class such that the corresponding $\bm C$ is generic. In this case, we have that $a_1=a_2=a_3=0$ and  thus Formula (\ref{affinecubicsurface}) has the following simpler form $$xyz+x^2+y^2+z^2+a_4^2-4=0 .$$

Notice that $A(C_1)=A(C_2)=A(C_3)={\bm \mu}_2$. In this case, we thus have $$H_{\bm C}=\{y=(y_1,y_2,y_3) \in {\bm\mu}_2^3 \ | \ y_1y_2y_3=1\}=\{(1,1,1),(-1,-1,1),(-1,1,-1),(1,-1,-1)\} .$$ We will describe the polynomial $H^y_c(\mathcal{M}_{\bm C};q,t)$ for $y=(-1,-1,1)$. A similar computation applies for the elements $(-1,1,-1),(1,-1,-1)$.

\bigskip

From Formula (\ref{non-degenerate-twisted}) and Lemma \ref{lemma-combinatorial-2}, we know that $E^y(\mathcal{M}_{\bm C};q)=q\mathbb{A}_2\left(\sqrt{q},\dfrac{1}{\sqrt{q}}\right)$. For $k=4$, we have that $\mathbb{A}_2(z,w)=z^2+w^2$ and thus $E^y(\mathcal{M}_{\bm C};q)=q^2+1$. This implies that $\tr(y\mid W^4_4/W^4_3)=1$, $\tr(y\mid W^2_2/W^2_1)=0$ and $\tr(y\mid W^2_0/W^2_{-1})=1$, and therefore \begin{equation}
\label{verificationk4}
    H^y_c(\mathcal{M}_{\bm C};q,t)=qt^2+t^2=qt^2\mathbb{A}_2\left(-t\sqrt{q},\dfrac{1}{\sqrt{q}}\right).
\end{equation}

We could have obtained the same result in a direct and more geometric way, looking at the action of $y$ on $S$. The latter extends indeed to an action of $y$ on $\overline{S}$ given by $y([x,y,z,w])=[-x,-y,z,w]$. Using the long exact sequence in cohomology for $\overline{S}=S \bigsqcup \Delta$ and that the action of $y$ on $H^{\bullet}(\Delta)$ is trivial, we get similarly Formula (\ref{verificationk4}) above.

\section{$\mathrm{PGL}_n$-character stacks}

In this section, $K$ is an algebraically closed field. We put $\PGl_n=\PGl_n(K)$. If $K=\overline{\F}_q$, $F:\PGl_n\rightarrow\PGl_n$ denotes the Frobenius $(a_{i,j})\mapsto (a_{i,j}^q)$ and we assume that $n\mid q-1$ or equivalently that $\bm{\mu}_n\subseteq\F_q^\times$. Again, for simplicity, we will focus on the case where $g=0$, which is the most relevant case for our results in representation theory. However, most of our results can be immediately generalized to the case where $g>0$: for more details on these generalizations, see Conjecture \ref{IHg>0} and Theorem \ref{IEg>0}.

\subsection{$\mathrm{PGL}_n$-conjugacy classes}

Let $p_n:\Gl_n \to \PGl_n$ be the canonical projection map. Consider a conjugacy class $\mathcal{C} \subseteq \PGl_n$ and a conjugacy class $C \subseteq \Gl_n$ such that $$p_n(C)=\mathcal{C} .$$

Put

$$
A(\mathcal{C})\coloneqq A(C)=\{\lambda \in K^\times \ | \ \lambda \cdot C= C \}
$$
and put

$$
d(\mathcal{C}):=|A(\mathcal{C})|.
$$
Notice that $A(\mathcal{C})$ does not depend on the choice of $C$ and is finite. More precisely, we have $A(\mathcal{C}) \subseteq \bm\mu_n$ (see \S \ref{important}). This group has the following geometric interpretation.

\begin{lemma}
For any $\overline{x} \in \mathcal{C}$, there is an isomorphism \begin{equation}
\label{connected-components}
    A(\mathcal{C}) \cong \pi_0(C_{\PGl_n}(\overline{x}))
\end{equation}
\end{lemma}

\begin{proof}
Let $x \in C$ such that $p_n(x)=\overline{x}$. There is a short exact sequence of algebraic groups:

\begin{center}
    \begin{tikzcd}
1 \arrow[r,""] &p_n(C_{\Gl_n}(x)) \arrow[r,""] &C_{\PGl_n}(\overline{x}) \arrow[r,"\theta"] &A(\mathcal{C}) \arrow[r,""] &1         
    \end{tikzcd},
\end{center}
where the map $\theta$ is defined as follows. Given $h \in C_{\PGl_n}(\overline{x})$, pick $g \in \Gl_n$ such that $p_n(g)=h$. Since $h \in C_{\PGl_n}(\overline{x})$, we have $gxg^{-1}=\lambda x$ for a certain $\lambda \in K^*$. We put $\theta(h)=\lambda.$ This does not depend on the choice of $g$.

Since $p_n(C_{\Gl_n}(x))$ is connected and $A(\mathcal{C})$ is finite, we deduce the isomorphism (\ref{connected-components}).
\end{proof}

 Notice that, being a finite subgroup of $\bm\mu_n$, the group $A(\mathcal{C})$ is cyclic. If $d(\mathcal{C})=|A(\mathcal{C})| \neq 1$ (i.e. if the centralizer of an element $\overline{x} \in \mathcal{C}$ is not connected), we say that the conjugacy class $\mathcal{C}$ is \textit{degenerate}, otherwise we say that it is \textit{non degenerate}.

\begin{remark}
\label{type-division}
Assume that $K=\overline{\F}_q$ and that the eigenvalues of $\mathcal{C},C$ are all contained in $\F_q^\times$. Put $\omega \in \mathbb{T}_n^{\circ}$ for the type of $C$. Since $A(\mathcal{C})$ is cyclic, from \cref{important}, we see that, for each $\lambda \in \mathcal{P}^*$, the integer $d(\mathcal{C})$ divides $\omega((1,\lambda))$.

In particular, there is a well defined type $\dfrac{\omega}{d(\mathcal{C})} \in \mathbb{T}^{\circ}_n$, with $$\dfrac{\omega}{d(\mathcal{C})}\left((1,\lambda)\right)\coloneqq\dfrac{\omega((1,\lambda))}{d(\mathcal{C})} .$$

We give a similar definition of $\dfrac{\omega}{r}$ for any $r$ such that $r\mid d(\mathcal{C})$.
\end{remark}

\subsection{Local systems on $\mathrm{PGL}_n$-conjugacy classes}\label{Localsystemson}

The projection $p_n:C \to \mathcal{C}$ is a $A(\mathcal{C})$-Galois covering. We deduce that if $\mathcal{C}$ is non degenerate, the projection $p_n:C \to \mathcal{C}$ is an isomorphism. For each $z \in A(\mathcal{C})$, we denote be $\sigma_z:C \to C$ the corresponding Galois automorphism where $\sigma_z(g)=zg$.

\bigskip

We have  a decomposition, see (\ref{decomposition-equivariant})

\begin{equation}
    (p_n)_*\kappa\cong \bigoplus_{\chi \in \widehat{A(\mathcal{C})}}V_\chi\otimes \mathcal{L}^\mathcal{C}_{\chi}
\label{Gal}\end{equation}
 and $\{\mathcal{L}^\mathcal{C}_\chi\}_\chi$ is the set of isomorphism classes of irreducible $\PGl_n$-equivariant local systems on $\mathcal{C}$.

\begin{oss}
More generally, for any connected algebraic group $G$, any element $x$ in some conjugacy class $C$ of $G$, the irreducible $G$-equivariant local systems on $C$ are parametrized by the irreducible representations of the finite group $\pi_0(C_{G}(x))$, see for instance \cite[Lemma 8.4.11]{ginzburg}.
\end{oss}
\bigskip

As done before with $\Gl_n$, for any two conjugacy classes $\mathcal{C},\mathcal{C}'$ of $\PGl_n$ we write $\mathcal{C}'\leq \mathcal{C}$ if $\mathcal{C}'\subseteq \overline{\mathcal{C}}$. 

\begin{oss}
\label{remark-cloture}
For two conjugacy classes $\mathcal{C}' \leq \mathcal{C}$ of $\PGl_n$, there is a priori no definite relation of inclusion between $A(\mathcal{C})$ and $A(\mathcal{C}')$. Consider for example $n=4$ and the following elements $y_1,y_2,y_3\in \Gl_4$:

$$y_1= (J(-1,(2)),J(1,(2)))$$ $$y_2=(J(-1,(1^2)),J(1,(2)) $$ $$y_3=(J(1,(1^2)),J(-1,(1^2)) ,$$ the corresponding projections $x_i=p_4(y_i)$ in $\PGl_4$ and their conjugacy classes $\mathcal{C}_{x_i}$. On the one side, we have $\mathcal{C}_{x_3} \leq \mathcal{C}_{x_2} \leq \mathcal{C}_{x_1}$. On the other side, we have $A(\mathcal{C}_{x_1})=A(\mathcal{C}_{x_3})=\bm\mu_2$ and $A(\mathcal{C}_{x_2})=\{1\}$. Notice however that it is always true that $$A(\mathcal{C}) \subseteq A(\mathcal{C}^{ss}) .$$

\end{oss}

\bigskip

We have the following stratification: $$\overline{\mathcal{C}}=\bigsqcup_{\mathcal{C}'\leq \mathcal{C}} \mathcal{C}'.$$ Moreover, for any $\mathcal{C}'\leq \mathcal{C}$, there exists a unique $C'\leq C$ such that $p(C')=\mathcal{C}'$.

\begin{oss}
Notice that $A(\mathcal{C})$ acts on $\overline{C}$, since for each $\lambda \in K^*$ such that $\lambda \cdot C=C$, we have $\lambda \cdot \overline{C}=\overline{C}$, i.e. $A(\mathcal{C})$ acts on $\overline{C}$. The map $p_n:\overline{C} \to \overline{\mathcal{C}}$ is $A(\mathcal{C})$-invariant. However, from Remark \ref{remark-cloture} we deduce that the latter map is not a Galois covering in general.

\end{oss}

We can still deduce the following.

\begin{prop}
\label{conjclass-closure}
The map $p_n:\overline{C} \to \overline{\mathcal{C}}$ is a finite map and a $A(\mathcal{C})$-Galois covering over the open subset $\mathcal{C}$.
\end{prop}

\begin{proof}
Consider the stratification $$\displaystyle \overline{\mathcal{C}}=\bigsqcup_{\mathcal{C}'\leq \mathcal{C}} \mathcal{C}.$$We have  $p_n^{-1}(\mathcal{C}')=C'$ and $p_n:C' \to \mathcal{C}'$ is an $A(\mathcal{C}')$-Galois covering and so  is quasi-finite. 

Define in a similar way to what we did for $\Gl_n$, the conjugacy class $\mathcal{C}_{ss}$. We have $\overline{C} \subseteq A(\mathcal{C}_{ss}) \cdot \overline{C}$. The map $$p_n':A(\mathcal{C}_{ss}) \cdot \overline{C} \to \overline{\mathcal{C}}$$ is an $A(\mathcal{C}_{ss})$-Galois covering and thus proper. The map $p_n$ is the composition of $p_n'$ and the closed embedding $\overline{C} \subseteq A(\mathcal{C}_{ss}) \cdot \overline{C}$ and thus is also proper. We deduce that $p_n:\overline{C} \to \overline{\mathcal{C}}$ is finite.
\end{proof}

\bigskip

We deduce the following result from Lemma \ref{corollary-small-map} and Formula (\ref{Gal}).

\begin{prop}
\label{IC-conjclasses}
We have an isomorphism \begin{equation}
    (p_n)_*(\IC^\bullet_{\overline{C}})=\bigoplus_{\chi \in \widehat{A(\mathcal{C})}} V_\chi\otimes\IC^\bullet_{\overline{\mathcal{C}},\mathcal{L}^{\mathcal{C}}_{\chi}}
\end{equation}    
\end{prop}

\subsection{Geometry of $\mathrm{PGL}_n$-character stacks}
\label{section-geometry}

Fix a $k$-tuple $\bm{\mathcal{C}}=(\mathcal{C}_1,\dots,\mathcal{C}_k)$ of  conjugacy classes of $\PGl_n$ and a $k$-tuple $\bm C=(C_1,\dots,C_k)$ of conjugacy classes of $\Gl_n$ such that $p_n(C_i)=\mathcal{C}_i$ for each $i=1,\dots,k$. Fix now a $n$-th root 
$$
\lambda_{\bm C}=\sqrt[n]{\prod_{i=1}^k \det(C_i)}
$$
and consider the following affine algebraic variety

$$
X_{\overline{\bm {\mathcal{C}}}}\coloneqq \left\{(X_1,\dots,X_k) \in \overline{\mathcal{C}_1} \times \cdots \times \overline{\mathcal{C}_k} \ | \ X_1 \cdots X_k=1\right\} $$

and its open subvariety $$X_{\bm{\mathcal{C}}}=X_{\overline{\bm{\mathcal{C}}}} \cap \left(\mathcal{C}_1 \times \cdots \times \mathcal{C}_k\right) .$$

For any $\zeta \in \bm\mu_n$, put

$$
\bm C(\zeta):=(C_1,\dots,C_k,\zeta \lambda_{\bm C}^{-1}I_n).
$$
Then

$$
X_{\overline{\bm{C}(\zeta)}}:=\left\{(X_1,\dots,X_k)\in\overline{C}_1\times\cdots\times\overline{C}_k\,|\, X_1\cdots X_k=\zeta^{-1}\lambda_{\bm{C}}\right\}.
$$
We have a decomposition
$$(p_n^k)^{-1}(X_{\overline{\bm{\mathcal{C}}}})= \bigsqcup_{\zeta \in \bm\mu_n} X_{\overline{\bm C(\zeta)}}
,$$
i.e. the following diagram is cartesian \begin{center}
    \begin{tikzcd}
  \displaystyle\bigsqcup_{\zeta \in \bm\mu_n} X_{\overline{\bm C(\zeta)}} \arrow[r," "] \arrow[d,"p"] & \overline{C_1} \times \cdots \times \overline{C_k} \arrow[d," p^k"] \\
    X_{\overline{\bm{\mathcal{C}}}} \arrow[r,"i_{\overline{\bm{\mathcal{C}}}}"] & \overline{\mathcal{C}_1} \times \cdots \times \overline{\mathcal{C}_k}.
    \end{tikzcd}
\end{center}

Notice that $\PGl_n$ acts diagonally by conjugation on each of the above varieties. We consider the $\PGl_n$-character stack
 
$$\mathcal{M}_{\overline{\bm{\mathcal{C}}}}\coloneqq [X_{\overline{\bm{\mathcal{C}}}} /\PGl_n] $$

and its open substack $$ \mathcal{M}_{\bm{\mathcal{C}}}=[X_{\bm{\mathcal{C}}}/\PGl_n] .$$

\begin{definizione}
\label{definitiongenericity}
The $k$-tuple $\bm{\mathcal{C}}$ of conjugacy classes of $\PGl_n$ is said to be \emph{generic} if, for every $\zeta \in \bm\mu_n$, the $(k+1)$-tuples $\bm C(\zeta)$ are generic in the sense of Definition \ref{definition-genericity-GLn}.
\end{definizione}

Recall that, under the genericity assumption, for each $\zeta \in \bm\mu_n$, the variety $X_{\bm C(\zeta)}$ is non empty if and only if $X_{\overline{\bm C(\zeta)}}$ is non empty by Theorem \ref{GLncharstacksdescription}(ii).

Under these assumptions, we have thus the following stratifications for the representation varieties and the corresponding character stacks: 
$$X_{\overline{\bm C(\zeta)}}=\bigsqcup_{\bm C'\leq \bm C}X_{\bm C'(\zeta)} \ \ \ \  \text{ and }  X_{\overline{\bm{\mathcal{C}}}}=\bigsqcup_{\bm{\mathcal{C}}'\leq \bm{\mathcal{C}}} X_{\bm{\mathcal{C}}'} $$

\begin{equation}
\label{stratification-character-stacks}
\mathcal{M}_{\overline{\bm C(\zeta)}}=\bigsqcup_{\bm C'\leq \bm C}\mathcal{M}_{\bm C'(\zeta)} \ \ \ \  \text{ and }  \mathcal{M}_{\overline{\bm{\mathcal{C}}}}=\bigsqcup_{\bm{\mathcal{C}'}\leq \bm{\mathcal{C}}} \mathcal{M}_{\bm{\mathcal{C}'}}.
\end{equation}

\begin{oss}
In \cite[Definition 12]{KNWG}, the authors give a definition of a generic $k$-tuple of conjugacy classes for any reductive group $G$. It is not hard to see that their definition agrees with our definition in the case of $\PGl_n$.
\end{oss}

\subsection{Description of $\mathrm{PGL}_n$-character stacks}
\label{description-section}
 Put $$A(\bm{\mathcal{C}})\coloneqq \prod_{i=1}^k A(\mathcal{C}_i) .$$
Since each $A(\mathcal{C}_i)$ is a subgroup of $\bm\mu_n$, we can define its subgroup $$H(\bm{\mathcal{C}})\coloneqq \{(y_1,\dots,y_k) \in A(\mathcal{C}_1) \times \cdots \times A(\mathcal{C}_k) \ | \ y_1\cdots y_k=1\} $$  and $$H'(\bm{\mathcal{C}})\coloneqq \langle A(\mathcal{C}_1),\dots,A(\mathcal{C}_k) \rangle \subseteq \bm\mu_n \subseteq K^\times.$$

\bigskip
\begin{oss}
The subgroup $H'(\bm{\mathcal{C}})$ is trivial if and only if the classes $\mathcal{C}_1,\dots,\mathcal{C}_k$ are non degenerate. On the other side, we can have $H(\bm{\mathcal{C}})=\{1\}$ even if some of the classes $\mathcal{C}_1,\dots,\mathcal{C}_k$ are degenerate. 
\end{oss}

\bigskip

Notice that there is a short exact sequence of abelian groups \begin{center}
\begin{tikzcd}
\label{sequence}
1 \arrow[r,""] &H(\bm{\mathcal{C}}) \arrow[r] &A(\bm{\mathcal{C}}) \arrow[r,"\psi"] &H'(\bm{\mathcal{C}}) \arrow[r,""] &1
\end{tikzcd}
\end{center}
where the second arrow is the canonical inclusion and $\psi((y_1,\dots,y_k))=y_1\cdots y_k$. 

Put 

$$d'(\bm{{\mathcal{C}}})\coloneqq |H'(\bm{\mathcal{C}})|$$
i.e. $H'(\bm{\mathcal{C}})=\bm\mu_{d'(\bm{\mathcal{C}})}$. Notice that $$d'(\bm{\mathcal{C}})=\lcm(d(\mathcal{C}_1),\dots,d(\mathcal{C}_k)),$$where we recall that $d(\mathcal{C}_i)=|A(\mathcal{C}_i)|$.
\bigskip

Let moreover 

$$
\iota(\bm{\mathcal{C}})\coloneqq \dfrac{n}{d'(\bm{\mathcal{C}})}
$$
and let $\zeta_n$ be a generator of  $\bm\mu_n\subset K^\times$ so that $\bm\mu_n=\{1,\zeta_n,\dots,(\zeta_n)^{n-1}\}$. We have  \begin{equation}
\label{H'}H'(\bm{\mathcal{C}})=\left.\left\{(\zeta_n)^{k \iota(\bm{\mathcal{C}})}\,\right|\, k=0,\dots, d'(\bm{\mathcal{C}})-1\right\}.\end{equation}  

For each $i=1,\dots,k$, put 

\begin{equation}
d'(\mathcal{C}_i)\coloneqq \dfrac{d'(\bm{\mathcal{C}})}{d(\mathcal{C}_i)}.
\label{d'}\end{equation}

We will need the following result.

\begin{lemma}
\label{splitting-exact-sequence}
The above short exact sequence splits and we have an isomorphism $$A(\bm{\mathcal{C}}) \cong H(\bm{\mathcal{C}}) \times H'(\bm{\mathcal{C}}) .$$
\end{lemma}

\begin{proof}

Notice that, for each $i=1,\dots,k$ and $z \in H'(\bm{\mathcal{C}})$, we have that $$ z^{d'(\mathcal{C}_i)} \in \bm\mu_{d(\mathcal{C}_i)}=A(\mathcal{C}_i) .$$ Moreover,  we have

\begin{equation}
    \label{gcd}
    \gcd(d'(\mathcal{C}_1),\dots,d'(\mathcal{C}_k))=1
\end{equation}
from which we deduce that there exist $\gamma_1,\dots,\gamma_k \in \Z$ such that $$\sum_{i} \gamma_i d'(\mathcal{C}_i)=1 .$$

Define $$\Psi:A(\bm{\mathcal{C}}) \to H(\bm{\mathcal{C}}) \times H'(\bm{\mathcal{C}}) $$ $$(y_1,\dots,y_k)\mapsto\left(\left(\dfrac{y_1}{\psi(y_1,\dots,y_k)^{\gamma_1 d'(\mathcal{C}_1)}},\dots,\dfrac{y_k}{\psi(y_1,\dots,y_k)^{\gamma_k d'(\mathcal{C}_k)}}\right), \psi(y_1,\dots,y_k)\right) .$$

It is not hard to see that $\Psi$ is an isomorphism.
\end{proof}

\bigskip

Consider the subset

$$
I(\bm{\mathcal{C}}):=\left\{(\zeta_n)^j\,|\,j=0,\dots,\iota(\bm{\mathcal{C}})-1\right\} \subseteq \bm{\mu}_n.
$$
We have the following.

\begin{prop}
\label{geometry-character-stacks}
  The morphism $$\displaystyle \overline{p}: \bigsqcup_{\zeta \in I(\bm{\mathcal{C}})}\mathcal{M}_{\overline{\bm C(\zeta)}} \to \mathcal{M}_{\overline{\bm{\mathcal{C}}}}$$ is finite and is a $H(\bm{\mathcal{C}})$-Galois covering over $\mathcal{M}_{\bm{\mathcal{C}}}$.
\end{prop}

\begin{remark}If the conjugacy classes $\mathcal{C}_i$ are all non-degenerate, then $I(\bm{\mathcal{C}})={\bm{\mu}}_n$, $H(\bm{\mathcal{C}})=1$ and so $\overline{p}$ is birational and, more precisely, its restriction
$$\overline{p}: \bigsqcup_{\zeta \in \bm{\mu}_n}\mathcal{M}_{\bm C(\zeta)}\cong \mathcal{M}_{\bm{\mathcal{C}}} $$ is an isomorphism.

\label{rem-non-degenerate}\end{remark}
\begin{proof}
From Proposition  \ref{conjclass-closure}, the map $p^k:\overline{C}_1 \times \cdots \times \overline{C}_k\to \overline{\mathcal{C}}_1 \times \cdots \times \overline{\mathcal{C}}_k$ is a finite map and an $A(\bm{\mathcal{C}})$-Galois covering over $\mathcal{C}_1 \times \cdots \times \mathcal{C}_k$. By base change we deduce that: \begin{equation}
p:\displaystyle  \bigsqcup_{\zeta \in \bm\mu_n} X_{\overline{\bm C(\zeta)}} \to X_{\overline{\bm{\mathcal{C}}}}
\end{equation}
is a finite map and an $A(\bm{\mathcal{C}})$-Galois covering over $X_{\bm{\mathcal{C}}}$. Therefore, we see that $\overline{p}$ is finite.

Identify $H'(\bm{\mathcal{C}})$ with a subgroup of $A(\bm{\mathcal{C}})$ through the isomorphism of Lemma \ref{splitting-exact-sequence}. For each $z \in H'(\bm{\mathcal{C}})$, let $\sigma_z$ be the corresponding Galois automorphism $$\sigma_z:\bigsqcup_{\zeta \in \bm\mu_n} X_{\bm C(\zeta)} \to \bigsqcup_{\zeta \in \bm\mu_n} X_{\bm C(\zeta)} .$$

For each $\zeta \in \bm\mu_n$, we have that $\sigma_z\left( X_{\bm C(\zeta)}\right)= X_{\bm C(z \zeta)}$. We deduce that there is an isomorphism $$\left(\bigsqcup_{\zeta \in \bm\mu_n} X_{\bm C(\zeta)}\right)/H'(\bm{\mathcal{C}}) \cong  \bigsqcup_{\zeta \in \nu(\mathcal{C})}X_{\bm C(\zeta)}  $$ and thus that $\overline{p}$ is a Galois covering over $X_{\mathcal{C}}$.

Taking the quotient by $\PGl_n$, we obtain the corresponding properties for the map $$\overline{p}: \bigsqcup_{\zeta \in I(\bm{\mathcal{C}})}\mathcal{M}_{\overline{\bm C(\zeta)}} \to \mathcal{M}_{\overline{\bm{\mathcal{C}}}} .$$

\end{proof}

\bigskip

\begin{oss}
\label{remark-irreducible-components}
The action of $H(\bm{\mathcal{C}})$  on $\displaystyle \bigsqcup_{\zeta \in I(\bm{\mathcal{C}})}\mathcal{M}_{\overline{\bm C(\zeta)}}$ leaves $\mathcal{M}_{\overline{\bm C(\zeta)}}$ invariant for each $\zeta \in \nu(\mathcal{C})$. Since $\overline{p}$ is an $H(\bm{\mathcal{C}})$-Galois covering over $\mathcal{M}_{\bm{\mathcal{C}}}$, we deduce that \begin{equation}
\label{irreducible-comp}
\overline{p}\left(\mathcal{M}_{\bm C(\zeta)}\right) \cap \overline{p}\left(\mathcal{M}_{\bm C(\zeta')}\right)=\emptyset\end{equation} for every $\zeta \neq \zeta'$.

Since $\overline{p}$ is finite, for every $\zeta \in I(\bm{\mathcal{C}})$, the image $\overline{p}\left(\mathcal{M}_{\bm C(\zeta)}\right)$ is an irreducible closed substack of dimension $d_{\bm C}$. From Formula (\ref{irreducible-comp}) we deduce that the $\overline{p}\left(\mathcal{M}_{\bm C(\zeta)}\right)$'s are exactly the irreducible component of $\mathcal{M}_{\overline{\bm {\mathcal{C}}}}$. In particular, the stack $\mathcal{M}_{\overline{\bm {\mathcal{C}}}}$ has $\iota(\bm{\mathcal{C}})$ irreducible components and $\overline{p}$ respects the Assumption \ref{assumptionmap}.

\end{oss}

\bigskip

From Proposition \ref{geometry-character-stacks} and Remark \ref{remark-irreducible-components} we deduce the following.

\begin{prop}
\label{proposition-DM}
For a generic $k$-tuple $\mathcal{C}$ of $\PGl_n$-conjugacy classes, the stack $\mathcal{M}_{\overline{\bm{\mathcal{C}}}}$ is an equidimensional Deligne-Mumford stack of dimension $d_{\bm C}$ with $\iota(\bm{\mathcal{C}})$ irreducible components. The substack $\mathcal{M}_{\bm{\mathcal{C}}}$ is a smooth and (everywhere) dense open substack.    
\end{prop}

\subsection{Cohomology of local systems on $\mathrm{PGL}_n$-character stacks}
\label{paragraph-cohomology}
In this section, we give our main results concerning the cohomology of $\PGl_n$-character stacks. 

\vspace{2 pt}
 
We choose a generic $k$-tuple $\bm{\mathcal{C}}=(\mathcal{C}_1,\dots,\mathcal{C}_k)$ of conjugacy classes $\PGl_n(K)$ and we let $\bm{C}=(C_1,\dots,C_k)$ be a $k$-tuple of conjugacy classes of $\Gl_n(K)$ above $\bm{\mathcal{C}}$. 
\bigskip

If $K=\overline{\F}_q$, we recall that $q-1$ is assumed to be divisible by $n$ (i.e. $\bm{\mu}_n\subseteq \F_q^\times$). In this case we further assume that the conjugacy classes $\mathcal{C}_i$ are $F$-stable with eigenvalues in $\F_q^\times$. We then choose the conjugacy classes $C_i$ to be  $F$-stable with eigenvalues in $\F_q^\times$. We also assume that  that $\lambda_{\bm C} \in \F_q^\times$, or, equivalently, that $$\prod_{i} \det(C_i) \in (\F_q^\times)^n ,$$ where $(\F_q^\times)^n \subseteq \F_q^\times$ is the subgroup of $n$-th powers. 
Under these assumptions, the constructions of \cref{section-geometry} are all compatible with $F$, i.e. are all defined over $\F_q$. 

 \bigskip

In the rest of the chapter, we put $\bm \omega$ to denote the multi-type $(\omega_1,\dots,\omega_k)$ where $\omega_i\in\mathbb{T}^o_n$ denotes the type of the conjugacy class $C_i\subset\Gl_n(K)$.

\subsubsection{Non-degenerate case}

Assume that $\mathcal{C}_1,\dots,\mathcal{C}_k$ are all non-degenerate. From Remark \ref{rem-non-degenerate} and Lemma \ref{corollary-small-map}, we deduce that \begin{equation}
    \label{IC-nondegenerate}
    \overline{p}_*\left(\bigoplus_{\zeta \in \bm\mu_n}\IC^\bullet_{ \mathcal{M}_{\overline{\bm C(\zeta)}}}\right) \cong \IC^\bullet_{\mathcal{M}_{\overline{\bm{\mathcal{C}}}}}.
\end{equation}

Taking global sections, we have the following.

\begin{prop}
If the $\PGl_n$-conjugacy classes $\mathcal{C}_1,\dots,\mathcal{C}_k$ are all non-degenerate, we have an isomorphism:

\begin{equation}
    \bigoplus_{\zeta \in \bm\mu_n} IH_c^*(\mathcal{M}_{\overline{\bm C(\zeta)}}) \cong IH_c^*(\mathcal{M}_{\overline{\bm{\mathcal{C}}}}).
\end{equation}
and thus 
\begin{equation}
IH_c(\mathcal{M}_{\overline{\bm{\mathcal{C}}}},q,t)=\sum_{\zeta \in \bm\mu_n} IH_c(\mathcal{M}_{\overline{\bm C(\zeta)}},q,t).    
\end{equation}
\end{prop}

The following conjecture and theorem are thus consequences of Theorem \ref{E-polynomial-GLn} and Conjecture \ref{mixed-polynomial-Gln}.   

\begin{conjecture}
\label{conj-nondegenerate}
If  $\mathcal{C}_1,\dots,\mathcal{C}_k$ are all non-degenerate, we have \begin{equation}
    IH_c(\mathcal{M}_{\overline{\bm{\mathcal{C}}}};q,t)=n(qt^2)^{d_{\bm\omega}}\mathbb{H}_{\bm\omega}\left(-t\sqrt{q},\dfrac{1}{\sqrt{q}}\right).
\end{equation}
\end{conjecture}

\begin{remark}If we choose the conjugacy classes $C_i$ in $\Sl_n$, then we can regard $\mathcal{M}_{\overline{\bm C}}$ as an $\Sl_n$-character stack and the above conjecture together with Conjecture \ref{mixed-polynomial-Gln} says that, up to multiplication by the cardinality of the center of $\Sl_n$, the mixed Poincar\'e series of the $\PGl_n$ and $\Sl_n$-character stacks agree.

\label{mirror}    
\end{remark}

\begin{teorema}
\label{theorem-nondegenerate}
If $\mathcal{C}_1,\dots,\mathcal{C}_k$ are all non-degenerate, we have \begin{equation}
    IE(\mathcal{M}_{\overline{\bm{\mathcal{C}}}};q)=n(qt^2)^{d_{\bm \omega}}\mathbb{H}_{\bm\omega}\left(\sqrt{q},\dfrac{1}{\sqrt{q}}\right).
\end{equation}    
\end{teorema}

\begin{remark}Assume that  $\mathcal{C}_1,\dots,\mathcal{C}_k$ are all semisimple conjugacy classes. From \cite[Theorem 5.3.10]{HA1} we know  that the coefficient of the highest power of $q$ in $\mathbb{H}_{\bm \omega}\left(\frac{1}{\sqrt{q}},\sqrt{q}\right)$ equals $1$. Therefore the coefficient of the highest power of $q$ in $E(\mathcal{M}_{\bm{\mathcal{C}}};q)$ equals $n$ which is also the number of connected components of the center of the dual group $\Sl_n$ of $\PGl_n$. This has been previously observed for an arbitrary connected reductive group \cite[Remark 3 (iii)]{KNWG}.
    
\end{remark}

\subsubsection{The degenerate case}
\label{localsystemsPGLncohomology}

Recall that the irreducible $\PGl_n(K)$-equivariant local systems on a conjugacy class $\mathcal{C}$ of $\PGl_n(K)$ are parametrized by the irreducible character of $A(\mathcal{C})$ and we denote by $\mathcal{L}_\chi^\mathcal{C}$ the irreducible local system on $\mathcal{C}$ corresponding to $\chi\in\widehat{A(\mathcal{C})}$ (see \S \ref{Localsystemson}).
\bigskip

Consider a character $\chi \in \widehat{A(\bm{\mathcal{C}})}$, where $\chi=\chi_1 \boxtimes \cdots \boxtimes \chi_k$ with $\chi_i \in \widehat{A(\mathcal{C}_i)}$ for each $i=1,\dots,k$. Let $\mathcal{F}^{\bm{\mathcal{C}}}_{\chi}\coloneqq \mathcal{L}^{\mathcal{C}_1}_{\chi_1} \boxtimes \cdots \boxtimes \mathcal{L}^{\mathcal{C}_k}_{\chi_k}$ be the corresponding local system on $\mathcal{C}_1 \times \cdots \times \mathcal{C}_k$ and  let $$\mathcal{E}^{\bm{\mathcal{C}}}_{\chi} \coloneqq  i_{\bm {\mathcal{C}}}^*(\mathcal{F}^{\bm{\mathcal{C}}}_{\chi}) $$ be its restriction to $X_{\bm{\mathcal{C}}}$. 

The local systems $\mathcal{E}^{\bm{\mathcal{C}}}_{\chi}$  are the ones coming from the Galois covering $$p:\bigsqcup_{\zeta \in \bm\mu_n}X_{\bm C(\zeta)} \to X_{\bm{\mathcal{C}}} ,$$ i.e.  we have an isomorphism \begin{equation}
    p_*(\kappa) \cong \bigoplus_{\chi \in \widehat{A(\bm{\mathcal{C}})}} V_\chi\otimes \mathcal{E}^{\bm{\mathcal{C}}}_{\chi}.
\end{equation}

We will need the following lemma.

\begin{lemma}For any $\chi\in \widehat{A(\bm{\mathcal{C}})}$
\label{restriction}
we have an isomorphism \begin{equation}
\label{isomrestriction0}
\IC^{\bullet}_{X_{\overline{\bm{\mathcal{C}}}},\mathcal{E}^{\bm{\mathcal{C}}}_{\chi}}=i_{\overline{\bm {\mathcal{C}}}}^*\left(\IC^{\bullet}_{\overline{\mathcal{C}_1}\times \cdots \times \overline{\mathcal{C}_k},\mathcal{F}^{\bm{\mathcal{C}}}_{\chi}}\right).    
\end{equation}
\end{lemma}

\begin{proof}
From Lemma \ref{corollary-small-map}, we have isomorphisms \begin{equation}
\label{isomrestriction1}
\IC^{\bullet}_{X_{\overline{\bm{\mathcal{C}}}},\mathcal{E}_{\chi}^{\bm{\mathcal{C}}}} \cong p_*\big(\bigoplus_{\zeta \in \bm\mu_n }\IC^{\bullet}_{X_{\overline{\bm C(\zeta)}}}\big)(\chi)   
\end{equation}
and 
\begin{equation}
\label{isomrestriction2}
i_{\overline{\bm {\mathcal{C}}}}^*\left(\IC^{\bullet}_{\overline{\mathcal{C}_1}\times \cdots \times \overline{\mathcal{C}_k},\mathcal{F}^{\bm{\mathcal{C}}}_{\chi}}\right) \cong i_{\overline{\bm {\mathcal{C}}}}^*\left(p^k_*\left(\IC^{\bullet}_{\overline{C_1} \times \cdots \times \overline{C_k}}\right)(\chi)\right) \cong  \left(i_{\overline{\bm {\mathcal{C}}}}^*\left(p^k_*(\IC^{\bullet}_{\overline{C_1} \times \cdots \times \overline{C_k}})\right)\right)(\chi)   
\end{equation}
where the last isomorphism of (\ref{isomrestriction2}) comes from the exactness of the functor $i_{\overline{\bm{\mathcal{C}}}}^*$. 

To construct an isomorphism (\ref{isomrestriction0}) it is enough to find an $A(\bm {\mathcal{C}})$-isomorphism between the complexes 
$$
p_*\left(\displaystyle \bigoplus_{\zeta \in \bm\mu_n }\IC^{\bullet}_{X_{\overline{\bm C(\zeta)}}}\right)\hspace{.5cm}\text{ and }\hspace{.5cm} i_{\bm {\mathcal{C}}}^*\left(p^k_*(\IC^{\bullet}_{\overline{C_1} \times \cdots \times \overline{C_k}})\right).
$$

From \cite[Theorem 4.10, Proposition 4.11]{L} we have an isomorphism
\begin{equation}
\label{restrictionktuples}
 i^*_{\bm C(\zeta)}\left(\IC^{\bullet}_{\overline{C_1} \times \cdots \times \overline{C_k}}\right) \cong  \IC^{\bullet}_{X_{\overline{\bm C(\zeta)}}}  
\end{equation}
for any $\zeta\in\bm\mu$. This isomorphism is $H(\bm{\mathcal{C}})$-equivariant.

Applying the functor $p_*$, we get an isomorphism \begin{equation}
 p_*\left(\bigoplus_{\zeta \in \bm\mu_n }\IC^{\bullet}_{X_{\overline{\bm C(\zeta)}}}\right)  \cong  p_*\left(\bigoplus_{\zeta \in \bm\mu_n}i^*_{\bm C(\zeta)}(\IC^{\bullet}_{\overline{C_1} \times \cdots \times \overline{C_k}})\right) 
\end{equation}
which commutes with the $A(\bm{\mathcal{C}})$ action on both sides.

Moreover, from the proper base change theorem, we get an $A(\bm{\mathcal{C}})$-equivariant isomorphism \begin{equation}
\label{isomorphismrestriction12}
\displaystyle p_ *\left(\bigoplus_{\zeta \in \bm\mu_n }\IC^{\bullet}_{X_{\overline{\bm C(\zeta)}}}\right) \cong i_{\bm {\mathcal{C}}}^*\left(p^k_*(\IC^{\bullet}_{\overline{C_1} \times \cdots \times \overline{C_k}})\right).    
\end{equation}

\end{proof}

\bigskip

Notice that, for each $\chi \in \widehat{A(\bm{\mathcal{C}})}$, the local system $\mathcal{E}_{\chi}^{\bm{\mathcal{C}}}$ on $X_{\bm{\mathcal{C}}}$ is $\PGl_n$-equivariant  and thus induces a unique local system, denoted again by $\mathcal{E}_{\chi}^{\bm{\mathcal{C}}}$, on the character stack $\mathcal{M}_{\bm{\mathcal{C}}}$.

The local systems $\mathcal{E}^{\bm{\mathcal{C}}}_{\chi}$ are the ones coming from the Galois covering $$\overline{p}:\bigsqcup_{\zeta \in \bm\mu_n}\mathcal{M}_{\bm C(\zeta)} \to \mathcal{M}_{\bm{\mathcal{C}}} ,$$ i.e.  we have an isomorphism \begin{equation}
    \overline{p}_*(\kappa) \cong \bigoplus_{\chi \in \widehat{A(\bm{\mathcal{C}})}} V_\chi\otimes\mathcal{E}^{\bm{\mathcal{C}}}_{\chi}.
\end{equation}

In this section, we will study the mixed Poincaré polynomial for the intersection cohomology $IH_c\big(\mathcal{M}_{\overline{\bm{\mathcal{C}}}},\mathcal{E}^{\bm{\mathcal{C}}}_{\chi};q,t\big)$.

\bigskip

Let $\Res$ be the restriction morphism $$\Res:\widehat{A(\bm{\mathcal{C}})} \to \widehat{H(\bm{\mathcal{C}})} .$$ In the following, we identify $\widehat{H(\bm{\mathcal{C}})}$ with a subset of $\widehat{A(\bm{\mathcal{C}})}$ through the isomorphism of Lemma \ref{splitting-exact-sequence}. Notice that, through this identification, for each $\chi \in \widehat{H(\bm{\mathcal{C}})}$, we have $\Res(\chi)=\chi$.

\bigskip 
From the properties of the map $\overline{p}$, we see that \begin{equation}
    \mathcal{E}^{\bm{\mathcal{C}}}_{\chi}\cong \mathcal{E}^{\bm{\mathcal{C}}}_{\chi'}
\end{equation}
if $\Res(\chi)=\Res(\chi')$. 

It is thus enough to describe the intersection cohomology $IH_c\big(\mathcal{M}_{\overline{\bm{\mathcal{C}}}},\mathcal{E}^{\bm{\mathcal{C}}}_{\chi};q,t\big)$ for the characters $\chi \in \widehat{H(\bm{\mathcal{C}})}$. 

Since the map $\displaystyle \overline{p}:\bigsqcup_{\zeta \in I(\bm{\mathcal{C}})}\mathcal{M}_{\bm{C}(\zeta)} \to \mathcal{M}_{\bm{\mathcal{C}}}$ is a Galois covering with Galois group $H(\bm{\mathcal{C}})$ by Proposition \ref{geometry-character-stacks}, we have $$\overline{p}_*(\kappa) \cong \bigoplus_{\chi \in \widehat{H(\bm{\mathcal{C}})}} V_\chi\otimes\mathcal{E}^{\bm{\mathcal{C}}}_{\chi} .$$

From Lemma \ref{corollary-small-map}, we deduce that we have \begin{equation}
   \overline{p}_*\left(\bigoplus_{\zeta \in I(\bm{\mathcal{C}})} \IC^\bullet_{\mathcal{M}_{\overline{\bm C(\zeta)}}}\right) \cong \bigoplus_{\chi \in \widehat{H(\bm{\mathcal{C}})}}V_\chi\otimes\IC^\bullet_{\mathcal{M}_{\overline{\bm{\mathcal{C}}}},\mathcal{E}^{\bm{\mathcal{C}}}_{\chi}}
\end{equation}

Taking hypercohomology, we have thus an equality 

\begin{equation}
\label{decomposition-cohomology}
\bigoplus_{\zeta \in I(\bm{\mathcal{C}})} IH^*_c\big(\mathcal{M}_{\overline{\bm C(\zeta)}}\big) \cong \bigoplus_{\chi \in \widehat{H(\bm{\mathcal{C}})}} V_\chi\otimes IH_c^*\big(\mathcal{M}_{\overline{\bm{\mathcal{C}}}},\mathcal{E}^{\bm{\mathcal{C}}}_{\chi}\big), 
\end{equation}
Since $H(\bm{\mathcal{C}})$ is abelian, the multiplicity space $V_\chi$ is of dimension $1$ and so the space  $IH_c\big(\mathcal{M}_{\overline{\bm{\mathcal{C}}}},\mathcal{E}^{\bm{\mathcal{C}}}_{\chi}\big)$ can be identified with the subspace of $\displaystyle \bigoplus_{\zeta \in I(\bm{\mathcal{C}})} IH^*_c\big(\mathcal{M}_{\overline{\bm C(\zeta)}}\big)$ on which $H(\bm{\mathcal{C}})$ acts by the character $\chi$. 
\bigskip

From the inversion formula in the character ring of $H(\bm{\mathcal{C}})$, we have

\begin{equation}
\label{inversion-formula-cohomology}
IH_c\big(\mathcal{M}_{\overline{\bm{\mathcal{C}}}},\mathcal{E}^{\bm{\mathcal{C}}}_{\chi};q,t\big)=\sum_{\zeta \in I(\bm{\mathcal{C}}) }\dfrac{1}{|H(\bm{\mathcal{C}})|}\sum_{y \in H(\bm{\mathcal{C}})}IH^y_c(\mathcal{M}_{\overline{\bm C(\zeta)}};q,t)\chi(y).    
\end{equation}

\bigskip

\begin{teorema}
\label{theoremEI}
Let $\zeta \in I(\bm{\mathcal{C}})$, $y=(y_1,\dots,y_k) \in H(\bm{\mathcal{C}})$ and for each $i=1,\dots,k$, denote by $o(y_i)$ the order of $y_i$. We have \begin{equation}
\label{formulaEI-twisted}
    IE^{y}(\mathcal{M}_{\overline{\bm C(\zeta)}};q)=q^{d_{\bm{\omega}}}\mathbb{H}_{\bm \omega_{o(y)}}\left(\sqrt{q},\dfrac{1}{\sqrt{q}}\right),
\end{equation}    
where $\bm{\omega}_{o(y)}$ is the multi-type defined by Formula (\ref{o(y)}). 
\end{teorema}
\begin{proof} The theorem follows from Theorem \ref{twisted-teo} as the character stack $\mathcal{M}_{\overline{\bm{C}(\zeta)}}$ is isomorphic to the character stack defined from the generic $k$-tuple of conjugacy classes $(C_1,\dots,C_{k-1},\zeta\lambda_{\bm{C}}^{-1}C_k)$ which is of same type $\bm{\omega}$ as $\bm{C}$(types do not depend on eigenvalues).
\end{proof}

Notice that the right-hand side of the formula does not depend on $\zeta\in I(\bm{\mathcal{C}})$.
\bigskip

We conjecture the following identity:

\begin{conjecture}
\label{conjIHy}
For every $\zeta \in I(\bm{\mathcal{C}}) $ and every $y \in H(\bm{\mathcal{C}})$, we have \begin{equation}\label{formulaHI-twisted}
    IH^{y}_c(\mathcal{M}_{\overline{\bm C(\zeta)}};q,t)=(qt^2)^{d_{\bm \omega}}\mathbb{H}_{\bm \omega_{o(y)}}\left(-t\sqrt{q},\dfrac{1}{\sqrt{q}}\right) 
    \end{equation}
\end{conjecture}

\bigskip

In what follows, put $\iota=\iota(\bm{\mathcal{C}})$, $d'=d'(\bm{\mathcal{C}})$, $d_i=d(\mathcal{C}_i)=|A(\mathcal{C}_i)|$ and $d_i'=\iota/d_i$ (see \S \ref{description-section}). Moreover, put

\begin{equation}
\zeta_{\iota}=(\zeta_n)^{\frac{n}{\iota}}
\label{zeta}\end{equation}
where $\zeta_n$ is a fixed generator of $\bm\mu_n\subset K^\times$. 

Notice that $H'(\bm{\mathcal{C}})=\langle \zeta_\iota \rangle \text{    and     } A(\mathcal{C}_i)=\langle \zeta_\iota^{d_i'} \rangle \text{ for all } i.$ 

For any $\chi_1 \boxtimes \cdots \boxtimes \chi_k \in \widehat{A(\bm{\mathcal{C}})}$ and any $i \in \{1,\dots,k\}$, let $s_{\chi_i} \in \{0,\dots,d_i-1\}$ be the integer such that $$\chi_i((\zeta_\iota)^{d_i'})=(\zeta_\iota)^{s_{\chi_i} d'_i} $$where, by notation abuse, $\zeta_\iota\in\kappa^\times$ is defined by (\ref{zeta}) with $\zeta_n$ a fixed generator of $\bm\mu_n\subset \kappa^\times$.

Put

$$s_{\chi} \coloneqq (s_{\chi_1},\dots,s_{\chi_k}).$$

\bigskip

Let $\phi:\N_{>0} \to \N$ be the Euler function. For any $m \in \N_{>0}$, denote by $C_m$ the coefficient of  $x^{\phi(m)-1}$ in the $m$-th cyclotomic polynomial in the variable $x$, i.e. 
\begin{equation}C_m=\sum_{\gcd(r,m)=1}(\zeta_m)^r,
\label{cyclo}\end{equation} where $\zeta_m$ is a primitive $m$-th root of unity (in $K^\times$ or $\kappa^\times$). We will need the following.

\begin{lemma}
\label{lemma-cyclotomic}
Let $\chi \in \widehat{\bm\mu_m}$ be a character and $s \in \{0,\dots,m-1\}$, such that $\chi(\zeta_m)=(\zeta_m)^s$.  For any $l \mid m$, we have that:
\begin{equation}
\label{formula-cyclotomic}
\sum_{\substack{h \in \bm\mu_m \\ \ord(h)=l}}\chi(h)=\phi(l)\dfrac{C_{\frac{l}{\gcd(l,s)}}}{\phi\left(\frac{l}{\gcd(l,s)}\right)}    
\end{equation}

\end{lemma}

\begin{proof}
We have  \begin{equation}\sum_{\substack{h \in \bm\mu_m \\ \ord(h)=l}}\chi(h)=\sum_{\substack{h \in \bm\mu_m \\ \ord(h)=l}}h^s=\sum_{\substack{e \in \{0,\dots,l\} \\ \gcd(e,l)=1}}(\zeta_m)^{\frac{mse}{l}}.
\end{equation}   

Let $l'=\dfrac{l}{\gcd(l,s)}$. Notice that, for each $e \in \{0,\dots,l\}$ such that $\gcd(e,l)=1$, the element $(\zeta_m)^{\frac{mse}{l}}$ is a primitive $l'$-th root of unity. More precisely, we have a surjective map 

$$
\theta:  \{e \in \{0,\dots,l\} \ | \ \gcd(e,l)=1\} \to \{\text{Primitive } l'-\text{th root of unity}\} $$

$$e \longmapsto  (\zeta_m)^{\frac{mse}{l}} .$$

Since each fiber of $\theta$ has cardinality $\dfrac{\phi(l)}{\phi(l')}$, we deduce formula (\ref{formula-cyclotomic}).
\end{proof}

\bigskip

For any $s=(s_1,\dots,s_k) \in \N^k$ and $r=(r_1,\dots,r_k)\in\N_{>0}^k$, put $$\Delta^{s}_r\coloneqq \phi(r_1) \cdots \phi(r_k)\sum_{j=0}^{d'-1} \prod_{i=1}^k \dfrac{C_{\frac{r_i}{\gcd(r_i,s_i+j)}}}{\phi\left(\frac{r_i}{\gcd(r_i,s_i+j)}\right)} .$$

Put $$R_{d_1,\dots,d_k}\coloneqq\left.\left\{r=(r_1,\dots,r_k) \in \N_{>0}^k \ \right| \ r_i \mid d_i \text{ for all } i\right\} .$$

Recall that $\bm\omega\in(\mathbb{T}_n^o)^k$ is the type of the $k$-tuple $\bm C$ of $\Gl_n(K)$-conjugacy classes above $\bm{\mathcal{C}}$. 

For any $r \in R_{d_1,\dots,d_k}$, put $$\bm \omega_r \coloneqq \left(\psi_{r_1}\left(\dfrac{\omega}{r_1}\right),\dots,\psi_{r_k}\left(\dfrac{\omega}{r_k}\right)\right) ,$$

where the $\dfrac{\omega_i}{r_i}$'s are the types introduced in Remark \ref{type-division}.

\bigskip

\begin{remark}
Notice that, for any $y \in A(\bm{\mathcal{C}})$, we have  $o(y) \in R_{d_1,\dots,d_k}$ and $\bm \omega_{o(y)}$ is the multi-type already introduced in \cref{section-twisted-mixed}.
\end{remark}

\bigskip

We have the following.

\begin{teorema}
\label{E-polynomial-general}
For any $\chi \in \widehat{A(\bm{\mathcal{C}})}$, we have
\begin{equation}
\label{edIE-theorem}
IE\left(\mathcal{M}_{\overline{\bm{\mathcal{C}}}},\mathcal{E}^{\bm{\mathcal{C}}}_{\chi};q\right)=\dfrac{q^{d_{\bm \omega}}\iota(\bm{\mathcal{C}})}{|A(\bm{\mathcal{C}})|}\sum_{r \in R_{d_1,\dots,d_k}}\Delta_{r}^{s_{\chi}}\,\mathbb{H}_{\bm \omega_r}\left(\sqrt{q},\dfrac{1}{\sqrt{q}}\right).    
\end{equation}
where $d_i=|A(\mathcal{C}_i)|$ and $\iota(\bm{\mathcal{C}})$ is the number of irreducible components of $\mathcal{M}_{\overline{\bm{\mathcal{C}}}}$ (see Proposition \ref{proposition-DM}).

\end{teorema}

\begin{proof}
From Formula (\ref{inversion-formula-cohomology}) and Formula (\ref{formulaEI-twisted}) we have

\begin{align*}
IE\left(\mathcal{M}_{\overline{\bm{\mathcal{C}}}},\mathcal{E}^{\bm{\mathcal{C}}}_{\chi};q\right)&=q^{d_{\bm{\omega}}}\sum_{\zeta\in I(\bm{\mathcal{C}})}\frac{1}{|H(\bm{\mathcal{C}})|}\sum_{y\in H(\bm{\mathcal{C}})}\mathbb{H}_{\bm{\omega}_{o(y)}}\left(\sqrt{q},\frac{1}{\sqrt{q}}\right)\chi(y)\\
&=\frac{q^{d_{\bm{\omega}}}\, |I(\bm{\mathcal{C}})|}{|H(\bm{\mathcal{C}})|}\sum_{y\in H(\bm{\mathcal{C}})}\mathbb{H}_{\bm{\omega}_{o(y)}}\left(\sqrt{q},\frac{1}{\sqrt{q}}\right)\chi(y)\\
&=\frac{q^{d_{\bm{\omega}}}\,\iota(\bm{\mathcal{C}})}{|H(\bm{\mathcal{C}})|}\sum_{r\in R_{d_1,\dots,d_k}}\mathbb{H}_{\bm{\omega}_r}\left(\sqrt{q},\frac{1}{\sqrt{q}}\right)\sum_{\substack{y\in H(\bm{\mathcal{C}})\\ o(y)=r}}\chi(y)
\end{align*}
To show Formula (\ref{edIE-theorem}) above, it is enough to show that, for any $r \in R_{d_1,\dots,d_k}$, we have 
\begin{equation}
\sum_{\substack{y \in H(\bm{\mathcal{C}}) \\ o(y)=r}}\chi(y)=\dfrac{\Delta_{r}^{s_{\chi}}}{|H'(\bm{\mathcal{C}})|}.    
\end{equation}

From Lemma \ref{splitting-exact-sequence}, we see that 
\begin{equation}
\sum_{\substack{y \in H(\bm{\mathcal{C}}) \\ o(y)=r}}\chi(y)=\dfrac{1}{|H'(\bm{\mathcal{C}})|}\sum_{\chi'\in \widehat{H'(\bm{\mathcal{C}})}}\,\sum_{\substack{y \in A(\bm{\mathcal{C}}) \\ o(y)=r}}(\chi\boxtimes\chi')(y)    
\end{equation}

Notice that, for any $\chi'\in \widehat{H'(\bm{\mathcal{C}})}$, we have 
\begin{equation}
\sum_{\substack{y \in A(\bm{\mathcal{C}}) \\ o(y)=r}}(\chi\boxtimes\chi')(y)=\prod_{i=1}^k \left(\sum_{\substack{y_i \in A(\mathcal{C}_i) \\ o(y_i)=r_i}}\chi_i\chi'(y_i)\right).    
\end{equation}

Formula (\ref{edIE-theorem}) is thus a consequence of Lemma \ref{lemma-cyclotomic}.

\end{proof}

\bigskip

We make the following conjecture.

\begin{conjecture}
\label{IH-conjecture-thm}
For any $\chi \in \widehat{A(\bm{\mathcal{C}})}$, we have
\begin{equation}
\label{IHconjecture}
IH_c\big(\mathcal{M}_{\overline{\bm{\mathcal{C}}}},\mathcal{E}^{\bm{\mathcal{C}}}_{\chi};q,t\big)=\dfrac{\iota(\bm{\mathcal{C}})(qt^2)^{d_{\bm \omega}}}{|A(\bm{\mathcal{C}})|}\sum_{r \in R_{d_1,\dots,d_k}}\Delta_{r}^{s_{\chi}}\,\mathbb{H}_{\bm \omega_r}\left(-t\sqrt{q},\dfrac{1}{\sqrt{q}}\right).    
\end{equation}    

\end{conjecture}

Our conjecture reduces to a conjectural statement on $\Gl_n$.

\begin{teorema}The above formula (\ref{IHconjecture}) is a consequence of Conjecture \ref{conjIHy}.

\label{redGL}\end{teorema}

\begin{proof}We use Formula (\ref{inversion-formula-cohomology}) and we proceed exactly as in the proof of Theorem \ref{E-polynomial-general} using the conjectural formula  (\ref{formulaHI-twisted}) instead of (\ref{formulaEI-twisted}).
\end{proof}

The following result follows from the above theorem together with the results of \S \ref{k=4}.

\begin{teorema}The formula (\ref{IHconjecture}) is true if $n=2$, $k=4$ and if the conjugacy classes are semisimple regular.
    
\label{casek=4}\end{teorema}

\subsection{$\mathrm{PGL}_n(\C)$-character stacks for Riemann surfaces of higher genus}

For any integer $g \geq 0$ and any Riemann surface $X$ of genus $g$, we can give a similar definition of the associated $\PGl_n$- representation variety $X_{\overline{\bm{\mathcal{C}}}}$ and $\PGl_n$-character stack $\mathcal{M}_{\overline{\bm{\mathcal{C}}}}$, i.e.  $$X_{\overline{\bm {\mathcal{C}}}}\coloneqq \left\{(A_1,B_1,\dots,A_g,B_g,X_1,\dots,X_k) \in \PGl_n^{2g} \times  \overline{\mathcal{C}_1} \times \cdots \times \overline{\mathcal{C}_k} \ \left| \ \prod_{i=1}^g [A_i,B_i]X_1 \cdots X_k=1\right\}\right. $$ and $$\mathcal{M}_{\overline{\bm{\mathcal{C}}}}=[X_{\overline{\bm {\mathcal{C}}}}/\PGl_n] $$ and similarly for the stacks $\mathcal{M}_{\overline{\bm C(\zeta)}}$. The stack $\mathcal{M}_{\overline{\bm{\mathcal{C}}}}$ parametrizes local systems on $X\setminus D$, where $D=\{x_1,\dots,x_k\}$ is a subset of $k$ points, with local monodromy around the points $x_1,\dots,x_k$ lying in $\overline{\mathcal{C}_1},\dots,\overline{\mathcal{C}}_k$ respectively. 

\bigskip

Given any $\chi \in \widehat{A(\bm{\mathcal{C})}}$, there is an associated local system $\mathcal{E}^{\bm{\mathcal{C}}}_{\chi}$ on $\mathcal{M}_{\overline{\bm{\mathcal{C}}}}$. Similarly to what was observed in \cite[Theorem 2.2.12]{HRV} in the case of central monodromy, it can be shown that 

\begin{equation}
\label{decomposition-cohomologyg1}
\bigoplus_{\zeta \in I(\bm{\mathcal{C}})} IH^*_c\big(\mathcal{M}_{\overline{\bm C(\zeta)}}\big) \cong \bigoplus_{\chi \in \widehat{H(\bm{\mathcal{C}})}} V_\chi\otimes IH_c^*\big(\mathcal{M}_{\overline{\bm{\mathcal{C}}}},\mathcal{E}^{\bm{\mathcal{C}}}_{\chi}\big) \otimes H^{\bullet}_c(\mathbb{G}_m)^{\otimes 2g}.
\end{equation}
\bigskip

A line of reasoning entirely analogous to that adopted in genus 0 shows thus that, for any $g >0$, we have

\begin{teorema}
\label{IEg>0}
For any generic $\bm{\mathcal{C}}$, we have the following:
$$ IE\left(\mathcal{M}_{\overline{\bm{\mathcal{C}}}},\mathcal{E}^{\bm{\mathcal{C}}}_{\chi};q\right)=\dfrac{q^{d_{\bm \omega}}\iota(\bm{\mathcal{C}})}{(q-1)^{2g}|A(\bm{\mathcal{C}})|}\sum_{r \in R_{d_1,\dots,d_k}}\Delta_{r}^{s_{\chi}}\,\mathbb{H}_{g,\bm \omega_r}\left(\sqrt{q},\dfrac{1}{\sqrt{q}}\right).  $$
\end{teorema}
Similarly, the extension of Conjecture \ref{IH-conjecture-thm} to genus $>0$ is the following one.
\begin{conjecture}
\label{IHg>0}
For any generic $\bm{\mathcal{C}}$, we have the following:
$$IH_c\big(\mathcal{M}_{\overline{\bm{\mathcal{C}}}},\mathcal{E}^{\bm{\mathcal{C}}}_{\chi};q,t\big)=\dfrac{\iota(\bm{\mathcal{C}})(qt^2)^{d_{\bm \omega}}}{(qt^2+t)^{2g}|A(\bm{\mathcal{C}})|}\sum_{r \in R_{d_1,\dots,d_k}}\Delta_{r}^{s_{\chi}}\,\mathbb{H}_{g,\bm \omega_r}\left(-t\sqrt{q},\dfrac{1}{\sqrt{q}}\right). $$
\end{conjecture}
\section{Geometric induction, character-sheaves, duality}

Assume that $G$ is a connected reductive algebraic group over $K$, $T$ is a maximal torus, $B\supset T$ a Borel subgroup and $W$ the Weyl group of $G$ with respect to $T$. 

Put

$$
\car:=T/\!/ W
$$
for the GIT quotient of $T$ by $W$.

\subsection{Geometric induction}\label{geometric-induction}

We consider the following morphism of correspondences

$$
\xymatrix{&&[B/B]\ar[rrd]^p\ar[lld]_{q'}\ar[d]^{(q',p)}&&\\
T&&S:=T\times_\car[G/G]\ar[rr]_-{\pr_2}\ar[ll]^-{\pr_1}&&[G/G]}
$$
where $[B/B]$ and $[G/G]$ denote the quotient stacks for the conjugation action.
\bigskip

We have functors between categories of perverse sheaves
(see \cite{BY}\cite[\S 2.9]{laumonletellier2})
$$
{\rm Ind}:=\Perv(T)\rightarrow\Perv([G/G]),\hspace{1cm}K\mapsto p_*q'{^!}(K)[{\rm dim}\, T)]({\rm dim}\, T).
$$

$$
\xymatrix{{\rm Res}:=\Perv([G/G])\rightarrow\Perv(T),\hspace{1cm}K\mapsto {^p }\mathcal{H}^0\left(q'_!p^*(K)[-{\rm dim}\, T](-{\rm dim}\, T)\right).}
$$
Since the morphism $(q',p)$ is small, we have

$$
(q',p)_!\overline{\Q}_\ell=\IC^\bullet_{S,\overline{\Q}_\ell}
$$
and so from the projection formulas we have

\begin{align*}
&{\rm Ind}(K)=\pr_{2\, *}\underline{\rm Hom}\left(\IC^\bullet_{S,\overline{\Q}_\ell},\pr_1^!(K)\right)[{\rm dim}\, T]({\rm dim}\, T).\\
&{\rm Res}(K)={^p}\mathcal{H}^0\left(\pr_{1\, !}\left(\IC^\bullet_{S,\overline{\Q}_\ell}\otimes\pr_2^*(K)\right)[-{\rm dim}\, T](-{\rm dim}\, T)\right).
\end{align*}

Consider the quotient map

$$
\pi:T\rightarrow [T/W]
$$
and the following commutative diagram

$$
\xymatrix{T\ar[d]_\pi&&S\ar[rr]^{\pr_2}\ar[ll]_{\pr_1}\ar[d]&&[G/G]\\
[T/W]&&[S/W]\ar[ll]_{\pr_1}\ar[rru]_{\pr_2}&&}
$$

Then (see \cite[Proposition 2.21]{laumonletellier2}) the functors ${\rm Ind}$ and ${\rm Res}$ factorise as

 $$
 {\rm Ind}=\Irm\circ \pi_*,\hspace{1cm}{\rm Res}=\pi^*\circ\Rrm
 $$
where

$$
\Irm:\Perv([T/W])\rightarrow\Perv([G/G]),\hspace{1cm}K\mapsto \pr_{2\, *}\underline{\rm Hom}\left(\IC^\bullet_{[S/W],\overline{\Q}_\ell},\pr_1^!(K)\right)[{\rm dim}\, T]({\rm dim}\, T).$$

$$
\Rrm:\Perv([G/G])\rightarrow\Perv([T/W]),\hspace{0.5cm}K\mapsto {^p}\mathcal{H}^0\left(\pr_{1\, !}\left(\IC^\bullet_{[S/W],\overline{\Q}_\ell}\otimes\pr_2^*(K)\right)[-{\rm dim}\, T](-{\rm dim}\, T)\right).$$
We have the following result (see \cite[\S 7.2]{laumonletellier2}).

\begin{teorema} The adjunction map

$$
\Rrm\circ\Irm\rightarrow 1
$$
is an isomorphism.

If $G$ is of type $A$ with connected center then $\Irm$ is an equivalence of categories with inverse functor $\Rrm$.
\label{I}\end{teorema}

\begin{remark} Let $(\overline{\Q}_\ell)_1$ be the skyscrapper sheaf on $T$ supported by $1$. We have a decomposition, see Decomposition (\ref{decomposition-BW})  

$$
\pi_*((\overline{\Q}_\ell)_1)=\bigoplus_{\chi\in\widehat{W}}V_\chi\otimes \mathcal{L}^{B(W)}_\chi.
$$
Then (reformulation of Borho-MacPherson's construction of Springer correspondence \cite[\S 6.2]{shoji})

$$
\Irm\big(\mathcal{L}_\chi^{B(W)}\big)=\IC^\bullet_{\overline{C}_\chi,\mathcal{E}_\chi}[{\rm dim}\, C_\chi]
$$
for some unipotent conjugacy class $C_\chi$ and some irreducible $G$-equivariant local system $\mathcal{E}_\chi$ on $C_\chi$ (if $\chi=1$ then $C_\chi$ is the unipotent regular conjugacy class and $\mathcal{E}_\chi$ the constant sheaf).
\label{Springerco}\end{remark}
\bigskip

Assume that $L\supset T$ is a Levi factor of some parabolic subgroup $P$ of $G$ and let $W_L$ be the Weyl group of $L$ with respect to $T$. 

Consider the correspondence

$$
\xymatrix{[L/L]&&[P/P]\ar[rr]^p\ar[ll]_q&&[G/G]}
$$
with corresponding induction functor (see \cite{BY})

$$
{\rm Ind}_{[L/L]}^{[G/G]}:\Perv([L/L])\rightarrow\Perv([G/G]),\hspace{.5cm}K\mapsto p_*q^!(K)
$$
The following diagram commutes

$$
\xymatrix{\Perv([L/L])\ar[rr]^{{\rm Ind}_{[L/L]}^{[G/G]}}&&\Perv([G/G])\\
\Perv([T/W_L])\ar[u]^{\Irm}\ar[rr]^{(\pi_L)_*}&&\Perv([T/W])\ar[u]_{\Irm}}
$$
where 

$$
\pi_L:[T/W_L]\rightarrow[T/W]
$$
is the map induced by the inclusion $W_L\subset W$.

\subsection{Character-sheaves}
 Assume that $G$ and $T$ are defined over $\F_q$, with geometric Frobenius $F:T \to T$. A \textit{Kummer local system} $\mathcal{E}$ is a $\overline{\Q}_{\ell}$-local system on $T$ such that $\mathcal{E}^{\otimes m}\simeq\overline{\Q}_\ell$ for some $m \in \N$ such that $(m,q)=1$. Notice that in particular every Kummer local system is of rank $1$ and thus simple.
 
 For any $F$-stable Kummer local system $\mathcal{E}$ on $T$, the characteristic function $\X_{\mathcal{E}}$ (with respect to the natural $F$-equivariant structure with is the identity on stalks at $1$) is a linear character of the finite group $T^F$ and any linear character of $T^F$ is obtained in this way, i.e. 
 
\begin{prop}\cite[Proposition 2.3.1]{mars}
\label{bijectionkummer}
The map $\mathcal{E} \to \X_{\mathcal{E}}$ is an isomorphism between the group of $F$-stable isomorphism classes of Kummer local systems on $T$ and the group $\widehat{T^F}$ of linear characters of $T^F$.  
\end{prop}

The Kummer local systems are the character-sheaves on $T$.

\begin{esempio}
\label{examplekummer}
Consider $T=\mathbb{G}_m$ with the Frobenius $F(x)=x^q$ for $x \in \mathbb{G}_m$. In this case, we have $T^F=\F_q^\times$. Consider a linear character $\alpha:\F_q^\times \to \C^*$ and let $n$ be the order of $\alpha$. In particular, $n$ divides $q-1$. Fix a surjection $q_n:\F_q^\times \to \Z/n\Z$ (by sending a generator $\zeta$ of the cyclic group $\F_q^\times$ to its subgroup of order $n$ generated by  $\zeta^{\frac{q-1}{n}}$). Since $\alpha^n=1$, there exists a linear character $\mu:\Z/n\Z \to \C^*$ such that $\mu \circ q_n=\alpha$.

Consider now the $\Z/n\Z$-Galois cover $f_n:\mathbb{G}_m \to \mathbb{G}_m$ given by $f_n(z)=z^n$. We have a splitting $$(f_n)_*(\overline{\Q}_{\ell})=\bigoplus_{\xi \in \widehat{\Z/n\Z}} \mathcal{E}_{\xi} .$$

Since $f_n$ commutes with $F$, the local systems $\mathcal{E}_{\xi}$ are defined over $\F_q$ and have thus a canonical $F$-equivariant structure such that $\X_{\mathcal{E}_{\mu}}=\alpha$.
\end{esempio}
\bigskip

The character-sheaves we are considering in this article are the direct summands of the perverse sheaves of the form ${\rm Ind}(\mathcal{E}[{\rm dim}\, T])$ where $\mathcal{E}$ runs over the Kummer local systems on $T$. Equivalently, they are the perverse sheaves on $G$ of the form $\Irm(\overline{\mathcal{E}}[{\rm dim}\,T])$ where $\overline{\mathcal{E}}$ is a direct summand of $\pi_*(\mathcal{E})$ for some Kummer local system $\mathcal{E}$ on $T$.
\bigskip

If $G=\Gl_n$ or $\PGl_n$, we obtain all the character-sheaves on $G$ defined by Lusztig in this way. However for $G=\Sl_n$ this is not true any more.

We denote by ${\rm CS}_o(G)$ the set of isomorphism classes of character-sheaves on $G$ obtained from a Kummer local system on $T$.

\subsection{Langlands correspondence over finite fields}\label{LLD}
We assume that $G$, $T$ and $B$ are defined over $\F_q$ with geometric Frobenius $F$. We assume that $T$ is split (i.e. $T^F\simeq (\F_q^\times)^{{\rm dim}\, T}$). The Frobenius $F$ acts trivially on the Weyl group with respect to $T$. Denote by $X(T)$ the character group and by $Y(T)$ the co-character group. 
\bigskip

Let $G^\flat$ be another connected reductive group together $\F_q$ and by notation abuse we still denote by $F$ the corresponding geometric Frobenius on $G^\flat$. We let $B^\flat$ be an $F$-stable Borel subgroup of $G^\flat$ containing $T^\flat$. 
\bigskip

We say that $(G,F)$ and $(G^\flat,F)$ are in \emph{duality} (see \cite[Definition 5.21]{DL}) if there exists an isomorphism $\tau:X(T)\rightarrow Y(T^\flat)$ which takes simple roots (with respect to $B$) to simple coroots (with respect to $B^\flat$) and which is compatible with the action of the Galois group ${\rm Gal}(\overline{\F}_q/\F_q)$. The isomorphism $\tau$ is then compatible with the actions of the Weyl groups.
\bigskip

Fix an isomorphism $\overline{\F}_q^\times\simeq(\Q/\Z)_{p'}$ (where $p$ is the characteristic of $\F_q$) and an identification of $(\Q/\Z)_{p'}$ with the $n$-th roots of unity of $\overline{\Q}_{\ell}^\times$, with $n \wedge p=1$. We obtain thus a fixed embedding $\overline{\F}_q^\times\hookrightarrow\overline{\Q}_\ell^\times$.

\begin{remark}
\label{remark-isom}
For each $n \in \N$, from  the choice of the isomorphism $\overline{\F}_q^\times\simeq(\Q/\Z)_{p'}$ above, we can define  an isomorphism $$ \psi_{q^n}:\Hom(\F_{q^n}^\times,\overline{\Q}_{\ell}^\times) \to \F_{q^n}^\times$$  as follows. Identify $1/q^n-1$ with the corresponding element of $\F_{q^n}^\times$, through the isomorphism $\overline{\F}_q^\times\simeq(\Q/\Z)_{p'}$ fixed above.

Notice that, for any $\alpha \in \Hom(\F_{q^n}^\times,\overline{\Q}_{\ell}^\times)$, the element $\alpha(1/q^n-1)$ belongs to the $q^n-1$-th roots of unity of $\overline{\Q}_{\ell}^\times$. We put thus

\begin{equation}
    \label{choice}
    \psi_{q^n}(\alpha)=\alpha\left(\dfrac{1}{q^n-1}\right)
\end{equation}
where we are identifying $\F_{q^n}^\times$ with the $(q^n-1)$-th roots of unity of $\overline{\Q}_\ell^\times$ with the embedding fixed above.
\end{remark}

\bigskip

We have  a surjective group homomorphism

$$
Y(T)\rightarrow T^F,\hspace{1cm}y\mapsto y\left(\dfrac{1}{q-1}\right)
$$
where we are identifying $1/q-1$ with the corresponding element of $\F_{q}^\times$, through the isomorphism $\overline{\F}_q^\times\simeq(\Q/\Z)_{p'}$ fixed above.
\bigskip

The restriction of the elements of $\Hom(T,\mathbb{G}_m)$ to $T^F$ defines a surjective morphism

$$
X(T)\rightarrow\widehat{T^F}
$$
where we are identifying $\overline{\F}_q^\times \subseteq \overline{\Q}_{\ell}^\times$, through the fixed embedding above. See \cite[\S 5]{DL} or \cite[Proposition 13.7]{DM} for more details.
\bigskip

Therefore, we deduce the following. 
\begin{prop}
There is an isomorphism
\begin{equation}
 \label{bijection-LLD}   
\Psi:T^F\simeq\widehat{T^\flat{^F}}.
\end{equation}
which is compatible with the action of the Weyl group $W$. 
\end{prop}

\bigskip

\begin{esempio}
\label{GLn-example}
Consider $G=\Gl_n$ with the Frobenius $F((a_{i,j}))=(a_{i,j}^q)$ and let $T \subseteq \Gl_n$ be the torus of diagonal matrices. In this case $(G^{\flat},F^{\flat})=(G,F)$ and $T^{\flat}=T$. The bijection (\ref{bijection-LLD}), or rather $\Psi^{-1}$, has the following explicit expression. Notice that $T^F=(\F_q^\times)^n$ and $\widehat{T^{\flat F}}=\Hom(\F_q^\times,\overline{\Q}_{\ell}^\times)^n$.

Using the isomorphisms introduced in Remark \ref{remark-isom}, for $(\alpha_1,\dots,\alpha_n) \in \Hom(\F_q^\times,\overline{\Q}_{\ell}^\times)^n$, we have that \begin{equation}
    \Psi^{-1}(\alpha_1,\dots,\alpha_n)=(\psi_{q}(\alpha_1),\cdots ,\psi_q(\alpha_n)).
\end{equation}
\end{esempio}

\bigskip

\bigskip

An element $s\in T^F$ defines an $F$-stable skyscraper sheaf $(\overline{\Q}_\ell)_s$ on $T$ but also, via the above isomorphism, a linear character of $T^\flat{^F}$, and so, by Proposition \ref{bijectionkummer}, an $F$-stable Kummer local system $\mathcal{E}_s$ on $T^\flat$.
\bigskip

We now define a bijection between the irreducible constituents of $\pi_*((\overline{\Q}_\ell)_s)$ and $\pi^\flat_*(\mathcal{E}_s)$ respectively on $[T/W]$ and $[T^\flat/W]$. We denote by $W_s$ the stabilizer of $s$ in $W$ and by $W_s^o$ the Weyl group of $C_G(s)^o$ with respect to $T$. Then 

$$
W_s/W_s^o=C_G(s)/C_G(s)^o.
$$
From \cref{finitequotientstacks} we obtain the irreducible constituents of $\pi_*((\overline{\Q}_\ell)_s)$ and $\pi^\flat_*(\mathcal{E}_s)$ as follows.

\bigskip

We decompose $\pi$ as 

$$
T\longrightarrow[T/W_s^o]\longrightarrow[T/W]
$$
The pushforward of $(\overline{\Q}_\ell)_s$ along the first arrow decomposes into irreducible local systems as follows

$$
\bigoplus_{\psi\in\widehat{W_s^o}}V_\psi\otimes(\overline{\Q}_\ell)_{s,\psi}$$
Given $\psi\in\widehat{W_s^o}$, denote by $W_{s,\psi}$ the stabilizer of $\psi$ in $W_s$. By \cref{finitequotientstacks}, the irreducible constituents of the pusforward of $(\overline{\Q}_\ell)_{s,\psi}$ along the map $[T/W_s^o]\rightarrow[T/W]$ are parametrized by the irreducible characters of $W_{s,\psi}/W_s^o$.
\bigskip

Denote by $(\overline{\Q}_\ell)_{s,\psi;\varphi}$ the irreducible constituent corresponding to $\varphi\in\widehat{W_{s,\psi}/W_s^o}$. Similarly we get local systems $\mathcal{E}_{s,\psi;\varphi}$ from the Kummer local system $\mathcal{E}_s$. 

We define the bijection between the irreducible consituents of $\pi_*((\overline{\Q}_\ell)_s)$ and $\pi^\flat_*(\mathcal{E}_s)$ by 

\begin{equation}(\overline{\Q}_\ell)_{s,\psi;\varphi}\mapsto\mathcal{E}_{s,\psi';\varphi}\label{Fourier-Spr}\end{equation}where $\psi'$ denote the tensor product of $\psi$ with the sign character of $W_s^o$. This bijection makes sense because of the following lemma.
\bigskip

\begin{lemma}We have

$$
W_{s,\psi}=W_{s,\psi'}.
$$
\end{lemma}

\begin{proof}We need to check that the sign character of $W_s^o$ is stabilized by any element of $W_s$. This follows from the fact that an element of $W_s$ maps a basis of the root system of $C_G(s)^o$ to an other basis. Since the basis of the root system are all in the same $W_s^o$-orbit, we deduce that the conjugation action of the elements of $W_s$ on $W_s^o$ decomposes as an inner automorphism of $W_s^o$ followed by an automorphism of the Coxeter graph.

\end{proof}

\bigskip

\begin{oss}
Given $s \in T^F$, the perverse sheaf $\Irm\left(\pi_*((\overline{\Q}_\ell)_s)\right)$  has support $[\overline{C}/G]$, where  $C$ is the conjugacy class of a regular element whose semisimple part is $s$ (i.e. $\overline{C}$ is the fiber at $s$ of the Chevalley map $G\rightarrow T/\!/W$). For instance, if $s=1$, the class $C$ is the regular unipotent conjugacy class. Notice that any irreducible $G$-stable closed substack of $[\overline{C}/G]$ is of the form $[\overline{C'}/G]$ with $C'\subseteq \overline{C}$ a conjugacy class.

In particular, an irreducible component of $\Irm\left(\pi_*((\overline{\Q}_\ell)_s)\right)$ must be of the form $\IC^{\bullet}_{\overline{C'},\xi}$, where $C'\subseteq \overline{C}$ and $\xi$ is an $F$-equivariant irreducible $G$-equivariant local system on $C'$.
\end{oss}

\bigskip

We denote by $({\rm LS}_o(G)^F)_{\rm split}$ the set of pairs of the form $(C,\zeta)$ where $C$ is an $F$-stable conjugacy class of $G$ with eigenvalues in $\F_q^\times$ and $\zeta$ an $F$-equivariant $G$-equivariant irreducible local system on $C$ such that the perverse sheaf $\IC^\bullet_{\overline{C},\zeta}[{\rm dim}\, C]$ is the image by 

$$
\Irm:\Perv([T/W])\rightarrow\Perv([G/G])
$$
of an irreducible constituent of $\pi_*((\overline{\Q})_s)$ for some $s\in T^F$.
\bigskip

 On the dual side, we denote by $({\rm CS}_o(G^\flat)^F)_{\rm split}$ the set of $F$-equivariant character sheaves on $G^\flat$ which are the image by 

$$
\Irm^\flat:\Perv([T^\flat/W])\rightarrow\Perv([G^\flat/G^\flat])
$$
of an irreducible constituent of $\pi^\flat_*(\mathcal{E}_s)$ for some $s\in T^F$.
\bigskip

Using the bijection (\ref{Fourier-Spr}) for all $s$, the functors $\Irm$ and $\Irm^\flat$ (with Theorem \ref{I} in mind), we obtain a bijection (Langlands correspondence over finite fields)

\begin{equation}
\label{langlandsfinite}
\mathfrak{c}_G:({\rm LS}_o(G)^F)_{\rm split}\longrightarrow({\rm CS}_o(G^\flat)^F)_{\rm split}
\end{equation}

\subsection{Langlands correspondence and Levi subgroups}
\label{rmk-Levi}
Consider an $F$-stable Levi subgroup $G \supseteq L \supseteq T$. We denote by  $\Phi_L \subseteq  X(T)$ and $\Phi^{\vee}_L  \subseteq Y(T)$ its corresponding roots and coroots systems. 

The subroot systems $\tau(\Phi_{L})\subseteq Y(T^{\flat})$ and $\tau(\Phi^{\vee}_L) \subseteq X(T^{\flat})$ determine a unique Levi subroup $L^{\flat}$ such that $\Phi_{L^{\flat}}=\tau(\Phi^{\vee}_L)$ and $\Phi^{\vee}_{L^{\flat}}=\tau(\Phi_{L})$. If $L$ is $F$-stable, since $\tau$ commutes with $F$, the subgroup $L^{\flat}$ is $F$-stable too.

Consider an element $s \in T^F$ and the corresponding character $\Psi(s) \in \widehat{T^{\flat F}}$. The argument of \cite[Proposition 11.4.12]{DM} shows the following.
\begin{prop}
\label{proplevitocenter}
We have $s \in Z_L^F$ if and only if $\Psi(s)$ is the restriction of a character $L^{\flat F} \to \overline{\Q}^\times_{\ell}$ which is trivial on $[L^{\flat},L^{\flat}]^F$ .
\end{prop}

\begin{esempio}
\label{example-gln-levilanglands}
Consider $G=L=\Gl_n$, $T$ the torus of diagonal matrices, $s=(s_1,\dots,s_n) \in T^F=(\F_q^\times)^n$ and $\Psi(s)=(\alpha_1,\dots,\alpha_n) \in \Hom(\F_q^\times,\overline{\Q}_{\ell}^\times)^n$. 

Notice that $s \in Z_{\Gl_n}^F$ if and only if $s_i=s_j$ for each $i,j$. Conversely,  $\Psi(s)$ is the restriction of a character $\gamma \circ \det:\Gl_n(\F_q) \to \overline{\Q}_{\ell}^\times$ with $\gamma \in \Hom(\F_q^\times,\overline{\Q}_{\ell}^\times)$ if and only $\alpha_i=\gamma$ for each $i$, i.e. if and only if $\alpha_i=\alpha_j$ for each $i,j$. 

Since $\alpha_i=\psi_q(s_i)$ and $\psi_q$ is an isomorphism, this explains Proposition \ref{proplevitocenter} in this case.
\end{esempio}

\bigskip

 The same type of argument shows the following.

\begin{prop}
\label{propcentertolevi}
We have that $s \in [L,L]^F$  if and only if $\Psi(s)|_{Z_{L^{\flat F}}}$ is trivial.
\end{prop}

\begin{esempio}
Consider the same situation of Example \ref{example-gln-levilanglands} above. Notice that $s \in \Sl_n(\F_q)$ if and only if $s_1 \cdots s_n=1$. Conversely, $\alpha$ is trivial on $Z|_{\Gl_n^F}$ if and only if $\alpha(\lambda I_n)=1$ for every $\lambda \in \F_q^*$ i.e. if and only if $\alpha_1 \cdots \alpha_n=1$. 

Since $\alpha_1 \cdots \alpha_n=\psi_q(s_1\cdots s_n)$ and $\psi_q$ is an isomorphism, this explains Proposition \ref{propcentertolevi} in this case.
\end{esempio}

\subsection{The case $\mathrm{GL}_n$}\label{GL}

The situation of \S \ref{LLD} simplifies a lot in the case of $\Gl_n$ because the stabilisers of the elements of $\Gl_n$ are all connected. Therefore, the only irreducible $\Gl_n$-equivariant local system on conjugacy classes is the constant sheaf. For any semisimple element $s\in T$, the stabilizer $W_s$ of $s$ in $W$ is the Weyl group (with respect to $T$) of the Levi subgroup $L_s:=C_{\Gl_n}(s)$ of $\Gl_n$.
\bigskip

An important property of $\Gl_n$ is that the irreducible characters of $\Gl_n(\F_q)$ are exactly (up to an explicit sign) the characteristic functions of the $F$-stable character-sheaves on $\Gl_n(\overline{\F}_q)$.
\bigskip

Let $C$ be an $F$-stable conjugacy class of $\Gl_n$ and let $su$ be the Jordan decomposition of an element of $C^F$. 
\bigskip

We will need in the case of $\Gl_n$ to deal with non-split character sheaves (or characters), i.e. unlike \S \ref{LLD}, here we do not assume that $s$ leaves in a split maximal torus (which for $\Gl_n$ could be the torus of diagonal matrices).

We thus explain the full correspondence $\mathfrak{c}_{\Gl_n}$ between the $F$-stable conjugacy classes of $\Gl_n(\overline{\F}_q)$ (which is also the set of conjugacy classes of $\Gl_n(\F_q)$) and  the $F$-stable character-sheaves on $\Gl_n(\overline{\F}_q)$ (which is in bijection with the set of the irreducible characters of $\Gl_n(\F_q)$).
\bigskip

The finite group $L_s^F$ is of the form

$$
L_s^F\simeq\prod_{i=1}^r\Gl_{n_i}(\F_{q^{d_i}}).
$$
By Proposition \ref{proplevitocenter}, to $s$ corresponds a unique character $\theta_s \in \Hom(L_s^F,\overline{\Q}_{\ell}^\times)$.

The $L_s^F$ conjugacy class of $u\in L_s^F$ corresponds (by Springer correspondence \cite{Springer}) to an $F$-stable irreducible character of the Weyl group of $L_s$ (with respect to an maximally split $F$-stable  maximal torus of $L_s$). Namely, the element $u$ is $L_s^F$-conjugated to an element $$(J(1,\lambda^1),\dots,J(1,\lambda^r))$$ with $\lambda^i \in \mathcal{P}_{n_i}$.

For a partition $\lambda \in \mathcal{P}_m$, we denote by $\chi_{\lambda}$ the corresponding character of $S_m$. The $L_s^F$-conjugacy class of $u$ corresponds to $$\chi=(\chi_{(\lambda^1)'}, \chi_{(\lambda^2)'},\dots,\chi_{(\lambda^r)'}) .$$ We define a unipotent character $\mathcal{U}_{\chi}$ of $L_s^F$ as 

$$
\mathcal{U}_{\chi}=\mathcal{U}_{\chi_{(\lambda^1)'}}\boxtimes\cdots\boxtimes \mathcal{U}_{\chi_{(\lambda^r)'}}.
$$
Let $R_{L_s}^{\Gl_n}$ denote the Lusztig induction from virtual characters of $L_s^F$ to virtual characters of $G^F$, for more details see \cite[Chapter 9]{DM}. Then by \cite{LS}, the apriori virtual character 

$$
R_C^{\Gl_n}:=\epsilon_sR_{L_s}^{\Gl_n}(\theta_s\otimes\mathcal{U}_\chi)
$$
(where $\epsilon_s=(-1)^{\F_q-\text{rank}(L_s)}$) is a true irreducible character of $\Gl_n(\F_q)$. All irreducible characters of $\Gl_n(\F_q)$ are obtained in this way.
\bigskip

Then $\epsilon_s(-1)^nR_C^{\Gl_n}$ is the characteristic function of the character sheaf $$
\mathcal{X}_C^{\Gl_n}=\mathfrak{c}_{\Gl_n}(C,\overline{\Q}_\ell)$$
on $\Gl_n(\overline{\F}_q)$. 
\bigskip

\begin{esempio}
Under this correspondence, the trivial conjugacy class  (resp. the regular unipotent conjugacy class) of $\Gl_n$ corresponds to the trivial character (resp. the Steinberg character) of $\Gl_n(\F_q)$.
\end{esempio}
\section{The dual pair $(\mathrm{SL}_n,\mathrm{PGL}_n)$}

In this section we assume that $n\mid q-1$.

\subsection{Orbital complexes on $\mathrm{PGL_n}$}\label{orbital}

Assume that $G$ is a connected linear algebraic group over $\overline{\F}_q$ equipped with a geometric Frobenius $F:G\rightarrow G$. Let $C$ be an $F$-stable conjugacy class of $G$ together with an $F$-stable $G$-equivariant irreducible local system $\mathcal{E}$ on $C$. We also fix an $F$-equivariant structure $\phi:F^*(\mathcal{E})\simeq\mathcal{E}$ and we denote again by $\phi$ the induced $F$-equivariant structure on $IC^\bullet_{\overline{C},\mathcal{E}}$. 

\begin{prop} \cite[Proposition 4.4.13]{Let} The set $\{{\bf X}_{IC^\bullet_{\overline{C},\mathcal{E}},\phi}\}$, where $(C,\mathcal{E})$ runs over the pairs as above,  forms a basis of the space $\mathcal{C}(G^F)$ of class functions $G^F/G^F\rightarrow\overline{\Q}_\ell$.
\end{prop}

The above basis is a geometric counterpart of the basis of characteristic functions of conjugacy classes of $G^F$.
\bigskip

We now assume that $G=\PGl_n$. We fix an $F$-stable conjugacy class $\mathcal{C}$ of $\PGl_n$ and $\overline{x}\in\mathcal{C}^F$. We choose $x\in p_n^{-1}(\overline{x})^F$ where $p_n:\Gl_n\rightarrow\PGl_n$ is the quotient, and we denote by $C$ the conjugacy class of $x$.
\bigskip

Recall that the restriction of $p_n$ to $C$ is an $A(\mathcal{C})$-Galois covering and that 

\begin{equation}
(p_n)_*(\IC^\bullet_{\overline{C}})=\bigoplus_{\chi\in \widehat{A(\mathcal{C})}}V_\chi\otimes \IC^\bullet_{\overline{\mathcal{C}},\mathcal{L}^{\mathcal{C}}_\chi}.
\label{pIC}\end{equation}
\bigskip

The canonical $F$-equivariant structure $\varphi$ on $\IC^\bullet_{\overline{C}}$ is compatible with the trivial $A(\mathcal{C})$-equivariant structure and so by Formulas (\ref{theta}) and (\ref{equa-wF}) we have

\begin{equation}
 {\bf X}_{\IC^\bullet_{\overline{\mathcal{C}},\mathcal{L}^{\mathcal{C}}_\chi},\varphi_\chi}=\frac{1}{|A(\mathcal{C})|}\sum_{y\in A(\mathcal{C})}\chi(y)\, (p_n^{y F})_*\left({\bf X}_{\IC^\bullet_{\overline{C}},\varphi_y}\right)
\end{equation}
\bigskip

For $y\in A(\mathcal{C})$, let $\alpha\in\overline{\F}_q^\times$ be such that

$$
F(\alpha)=y\alpha,
$$
and let $O_y$ be the $\Gl_n$-conjugacy class of $\alpha x$; it is $F$-stable and the Frobenius $F$ on $C$ corresponds to the Frobenius $y F$ on $C$, i.e. the following diagram commutes

$$
\xymatrix{C\ar[rr]^{z\mapsto\alpha z}\ar[d]_{y F}&&O_y\ar[d]^F\\
C\ar[rr]^{z\mapsto\alpha z}&& O_y}
$$
The $y F$-equivariant complex $(\IC^\bullet_{\overline{C}},\varphi_y)$ on $\overline{C}$ corresponds (under the isomorphism $z\mapsto \alpha z$) to the $F$-equivariant complex $\IC^\bullet_{\overline{O_y}}$ equipped with its natural $F$-equivariant structure which we also denote by $\varphi$.
\bigskip

We deduce the following proposition.

\begin{prop}For any $\chi\in\widehat{A(\mathcal{C})}$ we have
$${\bf X}_{\IC^\bullet_{\overline{\mathcal{C}},\mathcal{L}^{\mathcal{C}}_\chi},\varphi_\chi}=\frac{1}{|A(\mathcal{C})|}\sum_{y\in A(\mathcal{C})}\chi(y)\, (p_n)^F_*\left({\bf X}_{\IC^\bullet_{\overline{O_y}},\varphi}\right).
$$
\label{proj-finite}\end{prop}

\subsection{Character-sheaves on $\mathrm{SL}_n$}\label{CSSL}

In this section, let $T$, $\overline{T}$ and $T'$ be the maximal tori of diagonal matrices respectively of $\Gl_n$, $\PGl_n$ and $\Sl_n$. 

Fix an $F$-stable conjugacy class $C$ of $\Gl_n$ with eigenvalues in $\F_q^*$ and let $su$ be the Jordan decomposition of an element of $C^F$ with $s\in T(\F_q)$.
\bigskip

Let $\mathcal{C}$ be the image of $C$ in $\PGl_n$. Then $\overline{s}=p_n(s)$ is the semisimple part of an element of $\mathcal{C}^F$. Recall, see Formula (\ref{pIC}), that

\begin{equation}
(p_n)_*(\IC^\bullet_{\overline{C}})=\bigoplus_{\chi\in \widehat{A(\mathcal{C})}}V_\chi\otimes \IC^\bullet_{\overline{\mathcal{C}},\mathcal{L}^{\mathcal{C}}_\chi}
\label{proj}\end{equation}
Let $\mathcal{X}_C^{\Gl_n}$ be the character sheaf on $\Gl_n$ corresponding to $(C,\overline{\Q}_\ell)$ under the correspondence $\mathfrak{c}_{\Gl_n}$ and by $\mathcal{X}_{\mathcal{C},\chi}^{\Sl_n}$ the character sheaf on $\Sl_n$ corresponding to $(\mathcal{C},\mathcal{L}^{\mathcal{C}}_\chi)$ under the correspondence $\mathfrak{c}_{\PGl_n}$, see (\ref{langlandsfinite}).

We have the following result which is the dual version of (\ref{proj}).

\begin{teorema} The restriction $\mathcal{X}_C^{\Sl_n}$ of $\mathcal{X}_C^{\Gl_n}$ to $\Sl_n$ decomposes as 

$$
\mathcal{X}_C^{\Sl_n}=\bigoplus_{\chi\in \widehat{A(\mathcal{C}})}V_\chi\otimes \mathcal{X}_{\mathcal{C},\chi}^{\Sl_n}[1].
$$
\label{teoremadual}\end{teorema}

\begin{proof} We first analyse Formula (\ref{pIC}).

Since $\Gl_n$ and $\PGl_n$ are of type $A$ with connected center, by Theorem \ref{I} the complexes $\IC^\bullet_{\overline{C}}$ and $\IC^\bullet_{\overline{\mathcal{C}},\mathcal{L}^{\mathcal{C}}_\chi}$ corresponds to irreducible local systems on $[T/W]$ and $[\overline{T}/W]$ respectively.
\bigskip

More precisely, we have a cartesian diagram

$$
\xymatrix{[\overline{T}/W_L]\ar[rr]^{\overline{\pi}_L}&&[\overline{T}/W]\\[T/W_L]\ar[rr]^{\pi_L}\ar[u]^{p_n}&&[T/W]\ar[u]_{p_n}}
$$
where $L:=C_{\Gl_n}(s)$.
\bigskip

Then the local system on $[T/W]$ corresponding to $\IC^\bullet_{\overline{C}}$ is the local system $(\pi_L)_*((\overline{\Q})_{s,\psi})$ where $(\overline{\Q})_{s,\psi}$ is the direct factor corresponding to $\psi\in \widehat{W_L}$ of the pushforward of the skyscraper sheaf $(\overline{\Q}_\ell)_s$ along the map $T\rightarrow[T/W_L]$. The character $\psi$ corresponds under the Springer correspondence (of Borho-MacPherson) to the conjugacy class of $u$ in $L$.
\bigskip

The local system on $[\overline{T}/W]$ corresponding to $(p_n)_*(\IC^\bullet_{\overline{C}})\in D_c^b([\PGl_n/\PGl_n])$ is 

$$
(\overline{\pi}_L\circ p_n)_*((\overline{\Q}_\ell)_{s,\psi}).
$$
Moreover, $(p_n)_*((\overline{\Q})_{s,\psi})$ is irreducible as no element of the kernel of $T\mapsto\overline{T}$ fixes $s$. In fact

$$
(p_n)_*((\overline{\Q})_{s,\psi})=(\overline{\Q}_\ell)_{\overline{s},\psi}.
$$
We have 

$$
A(\mathcal{C})={\rm Stab}_{W/W_L}((\overline{\Q}_\ell)_{\overline{s},\psi})={\rm Stab}_{W_{\overline{s}}/W_L}(\psi).
$$
As $W=S_n$ and $W_L$ is of the form

$$
W_L=(S_{n_1})^{d_1}\times\cdots\times(S_{n_r})^{d_r}
$$
the group $W_{\overline{s}}/W_L$ is a subgroup of  $\prod_{i=1}^rS_{d_i}$ where each $S_{d_i}$ acts by permutation of the coordinates in $(S_{n_i})^{d_i}$.
\bigskip

We thus have a decomposition

$$
(\overline{\pi}_L)_*((\overline{\Q})_{\overline{s},\psi})=\bigoplus_{\chi\in\widehat{A(\mathcal{C})}}V_\chi\otimes (\overline{\Q})_{\overline{s},\psi;\chi},
$$
for some irreducible local systems $(\overline{\Q})_{\overline{s},\psi;\chi}$ on $[\overline{T}/W]$ which corresponds to $\IC^\bullet_{\overline{\mathcal{C}},\mathcal{L}^{\mathcal{C}}_\chi}$ under the equivalence $$
\Perv([\overline{T}/W])\simeq \Perv([\PGl_n/\PGl_n]).$$

According to \S \ref{LLD}, the skyscraper sheaf $(\overline{\Q}_\ell)_{\overline{s}}$ corresponds to a Kummer local system $\mathcal{E}_{\overline{s}}$ on $T'$ which is $W_L$-equivariant. 
\bigskip

The local system $\mathcal{E}_{\overline{s}}$ is also the restriction to $T'$ of  $\mathcal{E}_s$ on $T$.
\bigskip

As the two local systems $\mathcal{E}_s$ and $\mathcal{E}_{\overline{s}}$ are $W_L$-equivariant, their pushforwards along the maps $T\mapsto [T/W_L]$ and $T'\mapsto [T'/W_L]$ decomposes as direct sum of irreducible local systems $\mathcal{E}_{s,\psi}$ and $\mathcal{E}_{\overline{s},\psi}$ with multiplicity $V_\psi$  where $\psi$ runs over $\widehat{W_L}$, and the restriction of $\mathcal{E}_{s,\psi}$ to $T'$ is $\mathcal{E}_{\overline{s},\psi}$.

As induction commutes with restriction to $\Sl_n$, we thus have

$$
\Irm\left((\pi_L)_*\mathcal{E}_{s,\psi}[{\rm dim}\, T]\right)|_{\Sl_n}=\Irm'\left((\pi'_L)_*\mathcal{E}_{\overline{s},\psi}[{\rm dim}\, T']\right)[1]
$$
where 

$$
\Irm:\Perv([T/W])\rightarrow\Perv([\Gl_n/\Gl_n]),\hspace{1cm}\Irm':\Perv([T'/W])\rightarrow\Perv([\Sl_n/\Sl_n]).$$
Now the stabiliser of $\mathcal{E}_{s,\psi}$ in $W$ is $W_L$ and so the perverse sheaf 

$$
\Irm\left((\pi_L)_*\mathcal{E}_{s,\psi}[{\rm dim}\, T]\right)
$$
is irreducible. The stabiliser of $\mathcal{E}_{\overline{s},\psi}$ in $W/W_L$ is precisely ${\rm Stab}_{W_{\overline{s}}/W_L}(\psi)=A(\mathcal{C})$, and so

$$
(\pi'_L)_*(\mathcal{E}_{\overline{s},\psi})=\bigoplus_{\chi\in\widehat{A(\mathcal{C}})}V_\chi\otimes \mathcal{E}_{\overline{s},\psi;\chi}.
$$
We thus deduce that

$$
\Irm\left((\pi_L)_*\mathcal{E}_{s,\psi}[{\rm dim}\, T]\right)|_{\Sl_n}=\bigoplus_{\chi\in\widehat{A(\mathcal{C})}}V_\chi\otimes\Irm'(\mathcal{E}_{\overline{s},\psi;\chi}[{\rm dim}\, T'])[1]
$$
We conclude by noticing that

$$
\Irm\left((\pi_L)_*\mathcal{E}_{s,\psi}[{\rm dim}\, T]\right)=\mathcal{X}_C^{\Gl_n},\hspace{1cm}\Irm'(\mathcal{E}_{\overline{s},\psi;\chi}[{\rm dim}\, T'])=\mathcal{X}_{\mathcal{C},\chi}^{\Sl_n}.
$$
\end{proof}

For $y\in A(\mathcal{C})$, let $O_y$ be the $F$-stable conjugacy class of $\Gl_n(\overline{\F}_q)$ as defined in \S \ref{orbital}. Denote by $\mathcal{X}_{O_y}^{\Gl_n}$ the $F$-equivariant character sheaf on $\Gl_n$ corresponding to $\IC^\bullet_{\overline{O}_y}$ via the Langlands correspondence over finite fields for $\Gl_n$. 
\bigskip

The following result is the dual version of Proposition \ref{proj-finite}.

\begin{prop}

$$
{\bf X}_{\mathcal{X}^{\Sl_n}_{\mathcal{C},\chi}[1]}=\frac{1}{|A(\mathcal{C})|}\sum_{y\in A(\mathcal{C})}\chi(y)\,R_{O_y}^{\Sl_n}
$$
where the characteristic function is taken with the natural $F$-equivariant structure.
\label{decomp}\end{prop}

\begin{proof} From Theorem \ref{teoremadual}, we have an action of the group $A(\mathcal{C})$ on $\mathcal{X}_C^{\Sl_n}=\mathcal{X}^{\Gl_n}_C|_{\Sl_n}$ and we denote by $\varphi_y$ the $F$-equivariant structure on $\mathcal{X}^{\Sl_n}_C$ obtained from the natural $F$-equivariant structure twisted by $y$, i.e. composed with $\tilde{\theta}(y)$ where

$$
\tilde{\theta}:A(\mathcal{C})\rightarrow{\rm Aut}\left(\mathcal{X}^{\Sl_n}_C\right).
$$
Then from the orthogonality relation in the character ring of $A(\mathcal{C})$ we have

$$
{\bf X}_{\mathcal{X}^{\Sl_n}_{\mathcal{C},\chi}[1]}=\frac{1}{|A(\mathcal{C})|}\sum_{y\in A(\mathcal{C})}\chi(y)\,{\bf X}_{\mathcal{X}^{\Sl_n}_C,\varphi_y}.
$$
Analogously to conjugacy classes in $\PGl_n$ (see \S \ref{orbital}), we conclude from the formula

$$
{\bf X}_{\mathcal{X}^{\Sl_n}_C,\varphi_y}={\bf X}_{\mathcal{X}^{\Sl_n}_{O_y}}=R_{O_y}^{\Sl_n}.
$$
\end{proof}

\section{The main result}

In this section, as in the previous ones, $F$ denotes the standard Frobenius on $\Gl_n(\overline{\F}_q)$, $\Sl_n(\overline{\F}_q)$ and $\PGl_n(\overline{\F}_q)$ that raises matrix coefficients to their $q$-th power.

\subsection{Tensor products of irreducible characters of $\mathrm{GL}_n(\mathbb{F}_q)$}

For this section we refer to \cite[\S6.8, \S 6.9]{letellier2}. Recall (see \S \ref{GL}) that an irreducible character of $\Gl_n(\F_q)$ is given by a triple $(L,\theta,\chi)$ where $L$ is an $F$-stable maximal torus of $\Gl_n$, $\theta$ is a linear character of $L^F$ and $\chi$ is an $F$-stable irreducible character of the Weyl group of $L$ (with respect to some maximally split $F$-stable maximal torus of $L$).
\bigskip

Let $(R_1,\dots,R_k)$ be a $k$-tuple of irreducible characters of $\Gl_n(\F_q)$ and, for each $i$, let $(L_i,\theta_i,\chi_i)$ be a triple defining $R_i$.
\bigskip

\begin{definizione}
We say that $(R_1,\dots,R_k)$ is \emph{generic} if the two following conditions are satisfied :
\bigskip

(1) \begin{equation}\prod_{i=1}^k\theta_i|_{(Z_G)^F}=1\label{gen1}\end{equation}where for a group $G$, we denote by $Z_G$ the center of $G$.

(2) 
$$
\prod_{i=1}^k({^{g_i}}\theta_i)|_{(Z_M)^F} \neq 1
$$
for any proper $F$-stable Levi subgroup $M \subsetneq \Gl_n$ and for any $g_i\in\Gl_n(\F_q)$ such that $Z_M\subset g_iL_ig_i^{-1}$.
\end{definizione}
\bigskip

Notice that the genericity condition is only on the linear characters $\theta_i$ and not on the characters $\chi_i$.
\bigskip

Let $(C_1,\dots,C_k)$ be the $k$-tuple of $F$-stable conjugacy classes of $\Gl_n$ that corresponds to $(R_1,\dots,R_k)$ under the correspondence $\mathfrak{c}_{\Gl_n}$ explained in \S \ref{GL}, i.e. $R_i=R^{\Gl_n}_{C_i}$.

\bigskip

\begin{prop}
\label{generic-from-ktuples}If $(C_1,\dots, C_k)$ is generic then so is $(R_1,\dots,R_k)$.
\end{prop}

\begin{proof}
We can assume that $C_1,\dots,C_k$ are semisimple. Fix elements $s_1 \in C_1^F,\dots,s_k \in C_k^F$ corresponding respectively to the linear characters $\theta_1,\dots,\theta_k$. Then for all $i$ we have $L_i=C_{\Gl_n}(s_i)$.  
\bigskip

Let $L \subseteq \Gl_n$ be a proper $F$-stable Levi subgroup of $\Gl_n$ such that $g_iL_ig_i^{-1} \supseteq Z_L$ for some $g_1,\dots,g_k\in\Gl_n(\F_q)$. For each $i$, fix an $F$-stable maximal torus $T_i \subseteq L_i$ such that $Z_L \subseteq g_iT_ig_i^{-1} \subseteq g_iL_ig_i^{-1}$. We have $s_i \in T_i$ for each $i$, since $s_i \in Z_{L_i}.$

Moreover, since $L$ is the centralizer of $Z_L$ inside $\Gl_n$, we have that $g^{-1}_iT_ig_i \subseteq L$ for each $i$. We deduce that $g_i^{-1}s_ig_i \in L$ for each $i$. 

From the genericity condition, we deduce that $$\prod_{i=1}^k g_i^{-1}s_ig_i \not \in [L^F,L^F] $$
and so, from Proposition \ref{propcentertolevi}, we have
$$\prod_{i=1}^k({^{g_i}}\theta_i)|_{(Z_L)^F} \neq 1.$$
\end{proof}
\bigskip

We have the following result.

\begin{teorema}
\label{non-vanishing-theorem}The condition (\ref{gen1})  is a necessary condition to have

$$
\left\langle R_1\otimes\cdots\otimes R_k,1\right\rangle_{\Gl_n}\neq 0.$$
\label{nec-cond}\end{teorema}

\begin{proof}
From the proof of \cite[Theorem 4.3.1]{HA} and, in particular, \cite[Formula (4.3.1)]{HA}, we see that, if $\left\langle R_1\otimes\cdots\otimes R_k,1\right\rangle_{\Gl_n}\neq 0$, there must exist an $F$-stable $L \subseteq \Gl_n$ and $g_1,\dots,g_k \in \Gl_n(\F_q)$ such that $g_iZ_Lg_i^{-1} \subseteq L_i$ for each $i=1,\dots,k$ and $$\sum_{z \in ((Z_L)_{reg})^F}\prod_{i=1}^k({^{g_i}}\theta_i)(z) \neq 0 .$$ 

 The proof of \cite[Proposition 4.2.1]{HA} shows that \begin{equation}\label{formula-non-vanishing}\sum_{z \in ((Z_L)_{reg})^F}\Gamma(z)=\sum_{\substack{H \supseteq L}}\mu_H \sum_{z \in (Z_H)^F}\theta_{(R_1,\dots,R_k)}(z) ,\end{equation} where the sum is over the $F$-stable Levi subgroup $H \supseteq L$ and $\mu_H$ are certain integers.

In particular, Formula (\ref{formula-non-vanishing}) shows that if $\displaystyle \sum_{z \in ((Z_L)_{reg})^F}\theta_{(R_1,\dots,R_k)}(z) \neq 0$, there must exist $H \supseteq L$ such that  \begin{equation}\label{formula-non-vanishing2}
\displaystyle \sum_{z \in (Z_H)^F}\theta_{(R_1,\dots,R_k)}(z) \neq 0\end{equation}

Since $(Z_H)^F$ is a finite abelian group, Formula (\ref{formula-non-vanishing2}) is equivalent to ask for $(\theta_{(R_1,\dots,R_k)})|_{(Z_H)^F}=1$. Notice that $Z_{\Gl_n}^F \subseteq (Z_H)^F$ for any $F$-stable Levi subgroup $H$ and thus $\Gamma|_{(Z_{\Gl_n})^F}=1$. Theorem \ref{non-vanishing-theorem} is thus a consequence of the fact that, for any $z \in (Z_{\Gl_n})^F$, we have $$(\theta_{(R_1,\dots,R_k)})(z)=\prod_{i=1}^k({^{g_i}}\theta_i)(z)=\prod_{i=1}^k \theta(z) .$$
\end{proof}

\bigskip

For each $i$, let $\omega_i$ be the type of the conjugacy class $C_i$ and put $\bm{\omega}:=(\omega_1,\dots,\omega_k)$.
\bigskip

\begin{teorema}\cite[Theorem 6.10.1]{letellier2}\label{generic-multiplicities-gln} If $(R_1,\dots,R_k)$ is generic then

$$
\left\langle R_1\otimes\cdots\otimes R_k,1\right\rangle_{\Gl_n}=\mathbb{H}_{{\bm \omega}}(0,\sqrt{q}).
$$
\label{letmain}\end{teorema}

\subsection{Tensor products of character-sheaves on $\mathrm{SL_n}$}
\label{section-sln}

We assume that $n\mid q-1$.
\bigskip

Choose a $k$-tuple $\left((\mathcal{C}_1,\chi_1),\dots,(\mathcal{C}_k,\chi_k)\right)$ of pairs consisting of a split  $F$-stable conjugacy class $\mathcal{C}_i$ of $\PGl_n$ and an irreducible character $\chi_i$ of $A(\mathcal{C}_i)$.

We wish to study the inner product

$$
P:=\left\langle {\bf X}_{\mathcal{X}_{\mathcal{C}_1,\chi_1}^{\Sl_n}}\cdots{\bf X}_{\mathcal{X}_{\mathcal{C}_k,\chi_k}^{\Sl_n}},1\right\rangle_{\Sl_n}.
$$
For each $i=1,\dots,k$ let $C_i$ be an $F$-stable conjugacy class of $\Gl_n$ with eigenvalues in $\F_q^\times$ above $\mathcal{C}_i$ and for $y\in A(\mathcal{C}_i)$, let $O_{i,y}$ be an $F$-stable conjugacy class of $\Gl_n(\overline{\F}_q)$ defined from $(C_i,y)$ as $O_y$ were defined from $(C,y)$ in \S \ref{important}. If $y=1$ then we can choose  $O_{i,y}:=C_i$. 

From Proposition \ref{decomp}

$$
P=\frac{1}{|A(\bm{\mathcal{C}})|}\sum_{(y_1,\dots,y_k)\in A(\bm{\mathcal{C}})}\chi_1(y_1)\cdots\chi_k(y_k)\left\langle R^{\Sl_n}_{O_{1,y_1}}\otimes\cdots\otimes R^{\Sl_n}_{O_{k,y_k}},1\right\rangle_{\Sl_n}.
$$
By Frobenius reciprocity we have

\begin{equation}
\label{frobenius-reciprocity}
P=\frac{1}{|A(\bm{\mathcal{C}})|}\sum_{(y_1,\dots,y_k)\in A(\bm{\mathcal{C}})}\chi_1(y_1)\cdots\chi_k(y_k)\left\langle R_{O_{1,y_1}}^{\Gl_n}\otimes\cdots\otimes R_{O_{k,y_k}}^{\Gl_n} ,{\rm Ind}_{\Sl_n(\F_q)}^{\Gl_n(\F_q)}(1)\right\rangle_{\Gl_n}.
\end{equation}
Since 
$$
{\rm Ind}_{\Sl_n(\F_q)}^{\Gl_n(\F_q)}(1)=\sum_{\alpha\in\widehat{\F_q^\times}}\alpha\circ\det$$
we are reduced to study the multiplicities

$$
\left\langle R_{O_{1,y_1}}^{\Gl_n}\otimes\cdots\otimes R_{O_{k,y_k}}^{\Gl_n}\otimes(\alpha^{-1}\circ\det),1\right\rangle_{\Gl_n}$$
for linear characters $\alpha$ of $\F_q^\times$ and $(y_1,\dots,y_k)\in A(\bm{\mathcal{C}})$.
\bigskip

By Theorem \ref{nec-cond} the above multiplicitiy vanishes unless

\begin{equation}
\theta_{(R_{O_{1,y_1}}^{\Gl_n},\dots,R_{O_{k,y_k}}^{\Gl_n})}=\alpha^n.\label{root}\end{equation}

\bigskip

\begin{remark}For a linear character $\delta\in\widehat{\F_ q^\times}$ the equation 
$$
\delta=\alpha^n
$$
has a solution $\alpha\in\widehat{\F_q^\times}$ if and only if $\delta(\xi)=1$ where $\xi$ is a primitive $n$-th root of unity in $\F_q^\times$. Indeed, the morphism 

$$
\widehat{\F_q^\times}\longrightarrow\bm\mu_n\subset \overline{\Q}_\ell^\times,\,\,\chi\mapsto\chi(\xi)
$$
is surjective with kernel $\{\alpha^n\,|\, \alpha\in\widehat{\F_q^\times}\}$.
\end{remark}

We assume from now on that 

$$
\theta_{(R_{C_1}^{\Gl_n},\dots,R_{C_k}^{\Gl_n})}=(\lambda_R)^n
$$
for some linear character $\lambda_R$ of $\F_q^\times$.

\begin{oss}
\label{remark-nthpower}
Notice that, thanks to Proposition \ref{propcentertolevi} this is equivalent to ask that \begin{equation}\det(C_1) \cdots \det(C_k)=(\lambda_C)^n \end{equation} for some $\lambda_C \in \F_q^*$, which was the hypothesis under which we worked in \cref{paragraph-cohomology}.
\end{oss}

\bigskip

\begin{lemma}The equation (\ref{root}) has a solution $\alpha=\lambda_R^y\in\widehat{\F_q^\times}$ if and only if $y\in H(\bm{\mathcal{C}})$.
\end{lemma}
\bigskip

\begin{proof}
Thanks to Remark \ref{remark-nthpower} above, it is enough to show that there exists $\lambda^y_C \in \F_q^*$ such that $$\det(O_{1,y_1}) \cdots \det(O_{k,y_k})=(\lambda^y_{\bm C})^n $$ if and only if $y \in H(\bm{\mathcal{C}})$. For each $i$, let $\alpha_i$ such that $F(\alpha_i)=y_i \alpha_i$, as in \cref{important}. 

We have thus $$\det(O_{1,y_1}) \cdots \det(O_{k,y_k})=(\alpha_1\cdots \alpha_k)^n \det(C_1) \cdots \det(C_k).$$ Since $\det(C_1) \cdots \det(C_k)$ is an $n$-th power in $\F_q^\times$ and $\bm\mu_n \subseteq \F_q^\times$, we deduce that $$\det(O_{1,y_1}) \cdots \det(O_{k,y_k})$$ is an $n$-th power in $\F_q^\times$ if and only if $\alpha_1 \cdots \alpha_k \in \F_q^\times$.

Since $$F(\alpha_1 \cdots \alpha_k)=(y_1 \cdots y_k)\alpha_1 \cdots \alpha_k,$$ we have that $\alpha_1 \cdots \alpha_k \in \F_q^\times$ if and only if $y_1 \cdots y_k=1$, i.e. if and only if $y \in H(\bm{\mathcal{C}})$.
\end{proof}
\bigskip

\begin{prop}
\label{p2}Assume that the $k$-tuple $\bm{\mathcal{C}}:=(\mathcal{C}_1,\dots,\mathcal{C}_k)$ is generic (see Definition \ref{definitiongenericity}). Then for any $y\in H(\bm{\mathcal{C}})$ and any $\alpha\in\widehat{\F_q^\times}$ such that $\alpha^n=1$, the $(k+1)$-tuple

$$
\left(R_{O_{1,y_1}}^{\Gl_n},\dots,R_{O_{k,y_k}}^{\Gl_n},((\lambda_R^y)^{-1}\alpha)\circ\det\right)
$$
of irreducible characters of $\Gl_n(\F_q)$ is generic of type $\bm{\omega}_{o(y)}$.
\end{prop}

\begin{proof}
By Proposition \ref{generic-from-ktuples}, it is enough to show that the $(k+1)$-tuple of conjugacy classes $(O_{1,y_1},\dots,O_{k,y_k},(\lambda_{\bm C}^y)^{-1}\zeta I_n)$ is generic for every $y \in H(\bm{\mathcal{C}})$ and $\zeta \in \bm\mu_n$, i.e. that the $k$-tuple $$(C_1,\dots,C_k,(\lambda_{\bm C}^y)^{-1}\alpha_1\cdots \alpha_k\zeta I_n)$$ is generic for every $\zeta \in \bm\mu_n$.

Notice that $$(\lambda_{\bm C}^y)^n=\lambda_{\bm C}^n (\alpha_1 \cdots \alpha_k)^n$$ and thus we have that $(\lambda_{\bm C}^y)^{-1}\alpha_1\cdots \alpha_k$ is an $n$-th root of unity. We deduce that $$(O_{1,y_1},\dots,O_{k,y_k},(\lambda_{\bm C}^y)^{-1}\zeta I_n)$$ is generic from every $\zeta \in \bm\mu_n$ from Definition \ref{definitiongenericity}.
\end{proof}

\bigskip

Notice that the type of the above $(k+1)$-tuple does not depend on $\alpha$ such that $\alpha^n=1$. Hence if $\bm{\mathcal{C}}$ is generic, by Theorem \ref{letmain}, the multiplicities

$$
\left\langle R_{O_{1,y_1}}^{\Gl_n}\otimes\cdots\otimes R_{O_{k,y_k}}^{\Gl_n}\otimes((\lambda_R^y)^{-1}\alpha)\circ\det),1\right\rangle_{\Gl_n}$$
are independent from the character $\alpha$ such that $\alpha^n=1$.
\bigskip

We thus have the following formula.

\begin{prop}We have
\label{p3}
$$
P=\frac{\iota(\bm{\mathcal{C}})}{|H(\bm{\mathcal{C}})|}\sum_{(y_1,\dots,y_k)\in H(\bm{\mathcal{C}})}\chi_1(y_1)\cdots\chi_k(y_k)\left\langle R_{O_{1,y_1}}^{\Gl_n}\otimes\cdots\otimes R_{O_{k,y_k}}^{\Gl_n}\otimes ((\lambda_R^y)^{-1}\circ\det),1\right\rangle_{\Gl_n}.
$$

\end{prop}

Notice that $|A(\bm{\mathcal{C}})|=|H(\bm{\mathcal{C}})||H'(\bm{\mathcal{C}})|$ and $n=|H'(\bm{\mathcal{C}})| \iota(\bm{\mathcal{C}})$. 

\begin{teorema}
    \label{theoremmultiplicitiessln}
For any generic $k$-tuple $\bm{\mathcal{C}}$ of conjugacy classes of $\PGl_n$ and any $\chi \in \widehat{A(\bm{\mathcal{C}})}$, we have
\begin{equation}\label{tensor}
\left\langle {\bf X}_{\mathcal{X}_{\mathcal{C}_1,\chi_1}^{\Sl_n}}\cdots{\bf X}_{\mathcal{X}_{\mathcal{C}_k,\chi_k}^{\Sl_n}},1\right\rangle_{\Sl_n}=\dfrac{\iota(\bm{\mathcal{C}})}{|A(\bm{\mathcal{C}})|}\sum_{r \in R_{d_1,\dots,d_k}}\Delta_{r}^{s_{\chi}}\,\mathbb{H}_{\bm \omega_r}\left(0,\sqrt{q}\right).     
\end{equation}
\end{teorema}

\begin{proof}From Proposition \ref{p3}, Proposition \ref{p2} and Theorem \ref{generic-multiplicities-gln}, we have 

$$
P=\frac{\iota(\bm{\mathcal{C}})}{|H(\bm{\mathcal{C}})|}\sum_{y=(y_1,\dots,y_k)\in H(\bm{\mathcal{C}})}\chi_1(y_1)\cdots\chi_k(y_k)\,\mathbb{H}_{\bm{\omega}_{o(y)}}(0,\sqrt{q}).
$$
We compute this formula to get Formula (\ref{tensor}) as in the proof of Theorem \ref{E-polynomial-general}.

\end{proof}

\subsection{Convolution of orbital complexes on $\mathrm{GL}_n$}

In this section, $K$ is either $\C$ or $\overline{\F}_q$. Fix a multitype $\bm \omega \in (\mathbb{T}_n^{\circ})^k$ and let $\bm C$ be a generic $k$-tuple of conjugacy classes of multitype $\bm \omega$. We denote by $\mathcal{M}_{\overline{\bm C}}(K)$ the corresponding character stack.

\bigskip

The results of \cite[Theorem 4.10]{L} imply that we have an isomorphism \begin{equation}
    \label{isom-conv-gln}
    i_{\overline{\bm C}}^*\left(\IC^{\bullet}_{\overline{C}_1}\boxtimes \cdots \boxtimes \IC^{\bullet}_{\overline{C}_k}\right) \cong \IC^{\bullet}_{X_{\overline{\bm C}}}.
\end{equation}

If $K=\overline{\F}_q$, then the eigenvalues of the conjugacy classes $C_1,\dots,C_k$ are all in $\mathbb{F}_q$ and

\begin{align*}(q-1)\left\langle {\bf X}_{\IC^{\bullet}_{\overline{C}_1}} \ast \cdots \ast {\bf X}_{\IC^{\bullet}_{\overline{C}_k}}, 1_{1} \right\rangle_{\Gl_n(\F_q)}&=\dfrac{1}{|\PGl_n(\F_q)|}\sum_{\substack{(x_1,\dots,x_k) \in C_1^F \times \cdots \times C^F_k\\x_1 \cdots x_k=1}}{\bf X}_{\IC^{\bullet}_{\overline{C}_1}}(x_1)\cdots {\bf X}_{\IC^{\bullet}_{\overline{C}_k}}(x_k)\\
&=\sum_{x\in X_{\overline{\bm C}}(\F_q)}{\bf X}_{\IC^{\bullet}_{X_{\overline{\bm C}}}}(x),\end{align*} where the last equality is a consequence of Formula (\ref{isom-conv-gln}) above. 

From Theorem \ref{countingfq} and  Theorem \ref{E-polynomial-GLn},  we deduce the following result.

\begin{teorema}\cite[Theorem 4.14]{L} If $K=\mathbb{C}$, by abuse of notation, we still denote by $(C_1,\dots,C_k)$ a generic $k$-tuple of conjugacy classes of $\Gl_n(\overline{\F}_q)$ of type $\bm{\omega}$. 

For $K=\mathbb{C}$ or $K=\overline{\F}_q$ we have
\begin{align*}
(q-1)\left\langle {\bf X}_{\IC^{\bullet}_{\overline{C}_1}} \ast \cdots \ast {\bf X}_{\IC^{\bullet}_{\overline{C}_k}}, 1_{1} \right\rangle_{\Gl_n(\F_q)}&=q^{d_{\bm \omega}}\mathbb{H}_{\bm \omega} \left(\sqrt{q},\dfrac{1}{\sqrt{q}}\right)\\&=IE(\mathcal{M}_{\overline{\bm C}}(K);q). 
\end{align*}

\label{theorem-conv-gln}
\end{teorema}

\bigskip
\subsection{Picture for $\mathrm{GL}_n$}
\label{picture-GLn}
Let $\bm{C}=(C_1,\dots,C_k)$ be a generic $k$-tuple of conjugacy classes of $\Gl_n$ of type $\bm{\omega}\in(\mathbb{T}^o_n)^k$. 
\bigskip

We can put together the results of Theorem \ref{theorem-conv-gln}, \ref{generic-multiplicities-gln} and Conjecture \ref{mixed-polynomial-Gln} in the following diagram

\begin{center}
\begin{tikzcd}
\label{diag2}
& &  IH_c(\mathcal{M}_{\overline{\bm C}};q,t)\ar[dd,"t\mapsto -1"'] &(qt^2)^{d_{\bm \omega}}\mathbb{H}_{\bm \omega}\left(-t\sqrt{q},\dfrac{1}{\sqrt{q}}\right)\arrow[dd, "\text{"Pure part"}"]\arrow[ddl, "t\mapsto-1"']\arrow[l, equal, "\text{Conjecture \ref{mixed-polynomial-Gln}}"']\\
&&&\\
& &  (q-1)\left\langle {\bf X}_{\IC^{\bullet}_{\overline{C}_1}} \ast \cdots \ast {\bf X}_{\IC^{\bullet}_{\overline{C}_k}},1_1 \right\rangle_{\Gl_n(\F_q)} &  q^{d_{\bm \omega}} \left\langle R^{\Gl_n}_{C_1} \otimes \cdots \otimes R^{\Gl_n}_{C_k},1 \right\rangle_{\Gl_n(\F_q)}
\end{tikzcd}
\end{center}

 \bigskip

\subsection{Convolution of orbital complexes on $\mathrm{PGL}_n$}

We assume that $n\mid q-1$.
\bigskip

Choose first a generic $k$-tuple $\bm{\mathcal{C}}=(\mathcal{C}_1,\dots,\mathcal{C}_k)$ of $F$-stable split conjugacy classes of $\PGl_n(\overline{\F}_q)$  and a character $\chi \in \widehat{A(\bm{\mathcal{C}})}$. 

We want to understand the quantity $$Q\coloneqq \left\langle {\bf X}_{{\IC^{\bullet}_{\overline{\mathcal{C}}_1,\mathcal{L}_{\chi_1}}}} \ast \cdots \ast {\bf X}_{{\IC^{\bullet}_{\overline{\mathcal{C}}_k,\mathcal{L}_{\chi_k}}}},1_{1} \right\rangle_{\PGl_n(\F_q)}$$
with $\mathcal{L}_\chi:=\mathcal{L}^\mathcal{C}_\chi$ defined by Formula (\ref{pIC}).

Notice that we have \begin{align*}
Q&=\dfrac{1}{|\PGl_n(\F_q)|}\sum_{\substack{(x_1,\dots,x_k) \in \mathcal{C}_1^F \times \cdots \times\mathcal{C}_k^F \\ x_1 \cdots x_k=1}}{\bf X}_{{\IC^{\bullet}_{\overline{\mathcal{C}}_1,\mathcal{L}_{\chi_1}}}}(x_1)\cdots {\bf X}_{{\IC^{\bullet}_{\overline{\mathcal{C}}_k,\mathcal{L}_{\chi_k}}}}(x_k)\\&=\dfrac{1}{|\PGl_n(\F_q)|}\sum_{x=(x_1,\dots,x_k) \in X_{\overline{\mathcal{C}}}(\F_q)}{\bf X}_{\IC^{\bullet}_{X_{\overline{\mathcal{C}},\mathcal{E}^{\bm{\mathcal{C}}}_{\chi}}}}(x)   
\end{align*}
where the last equality is a consequence of Lemma \ref{restriction}. 

From Deligne-Grothendieck's trace formula and the isomorphism (\ref{decomposition-cohomology}), we have thus
\begin{align*}
Q&=\sum_i(-1)^i \tr\left(F\,|\,IH^i_c(\mathcal{M}_{\overline{\bm{\mathcal{C}}}},\mathcal{E}^{\bm{\mathcal{C}}}_{\chi})\right)\\
&=\sum_i \dfrac{(-1)^i}{|H(\bm{\mathcal{C}})|}\sum_{y \in H(\bm{\mathcal{C}})}\sum_{\zeta \in I(\bm{\mathcal{C}})}\tr\left(yF\mid IH^i_c(\mathcal{M}_{\overline{\bm C(\zeta)}})\right).    
\end{align*}

The proof of Theorem \ref{theoremtwistedEFq} shows that, for any $y \in H(\bm{\mathcal{C}})$ and any $\zeta \in I(\bm{\mathcal{C}})$, we have

\begin{equation}
\sum_i (-1)^i \tr\left(yF|IH^i_c(\mathcal{M}_{\overline{\bm C(\zeta)}})\right)=IE^y\left(\mathcal{M}_{\overline{\bm C(\zeta)}};q\right).
\end{equation}

We deduce thus that $$Q=\dfrac{1}{|H(\bm{\mathcal{C}})|}\sum_{\zeta \in I(\bm{\mathcal{C}})}\sum_{y \in H(\bm{\mathcal{C}})}IE^y(\mathcal{M}_{\overline{\bm C(\zeta)}};q)=IE(\mathcal{M}_{\overline{\bm{\mathcal{C}}}},\mathcal{E}^{\bm{\mathcal{C}}}_{\chi};q) .$$
\bigskip

From Theorem \ref{E-polynomial-general}, we deduce the following Theorem.

\begin{teorema}
\label{theorem-diagram}
For any generic  $k$-tuple $\bm{\mathcal{C}}$ and $\chi \in \widehat{A(\bm{\mathcal{C}})}$, we have
\begin{equation}
\left\langle {\bf X}_{{\IC^{\bullet}_{\overline{\mathcal{C}}_1,\mathcal{L}_{\chi_1}}}} \ast \cdots \ast {\bf X}_{{\IC^{\bullet}_{\overline{\mathcal{C}}_k,\mathcal{L}_{\chi_k}}}},1_{1} \right\rangle_{\PGl_n(\F_q)}=\dfrac{q^{d_{\bm \omega}}\iota(\bm{\mathcal{C}})}{|A(\bm{\mathcal{C}})|}\sum_{r \in R_{d_1,\dots,d_k}}\Delta_{r}^{s_{\chi}}\,\mathbb{H}_{\bm \omega_r}\left(\sqrt{q},\dfrac{1}{\sqrt{q}}\right).  
\label{equation-convolution}
\end{equation}
\end{teorema}

\bigskip

\begin{oss}
\label{remark-convolution}
It is possible to give another proof of Theorem \ref{theorem-diagram} above. Using Formula (\ref{proj-finite}), in a dual way to what we did for the computation of tensor product of character sheaves in Formula (\ref{frobenius-reciprocity}), we can  express the quantity $$\left\langle {\bf X}_{{\IC^{\bullet}_{\overline{\mathcal{C}}_1,\mathcal{L}_{\chi_1}}}} \ast \cdots \ast {\bf X}_{{\IC^{\bullet}_{\overline{\mathcal{C}}_k,\mathcal{L}_{\chi_k}}}},1_{1} \right\rangle_{\PGl_n(\F_q)}$$ in terms of the quantities $$\left\langle {\bf X}_{{\IC^{\bullet}_{\overline{O}_{1,y_1}}}} \ast \cdots \ast {\bf X}_{{\IC^{\bullet}_{\overline{O}_{k,y_k}}}},1_{1} \right\rangle_{\Gl_n(\F_q)}$$ for $(y_1,\dots,y_k) \in H(\bm{\mathcal{C}})$.

Formula (\ref{equation-convolution}) can then be obtained as a consequence of Theorem \ref{theorem-conv-gln} above.

\end{oss}

\subsection{Main result}

We can summarize the results of Theorem \ref{theorem-diagram}, \ref{theoremmultiplicitiessln} and Conjecture \ref{IH-conjecture-thm} in the following conjectural diagram, relating the cohomology of the complex character stack $\mathcal{M}_{\overline{\bm{\mathcal{C}}}}$ to the structure coefficients of the two rings $(\mathcal{C}(\PGl_n(\F_q)),\ast)$ and $(\mathcal{C}(\Sl_n(\F_q)),\otimes)$.

\begin{center}
\begin{tikzcd}
\label{diag1}
& & IH_c\left(\mathcal{M}_{\overline{\bm{\mathcal{C}}}},\mathcal{E}^{\bm{\mathcal{C}}}_{\chi};q,t\right)\ar[dd,"t\mapsto-1"']&\text{RHS Formula (\ref{IHconjecture})}\arrow[dd, "\text{ "pure part"}"]\arrow[ddl,"t \mapsto -1"']\arrow[l, equal, "\text{Conjecture \ref{IH-conjecture-thm}}"']\\
&&&\\
& & \left\langle {\bf X}_{{\IC^{\bullet}_{\overline{\mathcal{C}}_1,\mathcal{L}_{\chi_1}}}} \ast \cdots \ast {\bf X}_{{\IC^{\bullet}_{\overline{\mathcal{C}}_k,\mathcal{L}_{\chi_k}}}},1_{1} \right\rangle_{\PGl_n(\F_q)} &  q^{d_{\bm \omega}}\left\langle {\bf X}_{\mathcal{X}_{\mathcal{C}_1,\chi_1}^{\Sl_n}}\cdots{\bf X}_{\mathcal{X}_{\mathcal{C}_k,\chi_k}^{\Sl_n}},1\right\rangle_{\Sl_n(\F_q)}
\end{tikzcd}
\end{center}
\bigskip

Recall that Conjecture \ref{IHconjecture} reduces to the conjectural formulas for the mixed Poincar\'e polynomials for $\Gl_n$-character varieties (see Formula (\ref{inversion-formula-cohomology}). In the $\Gl_n$ case we have many evidences for the conjectural formula for mixed Hodge polynomials.

\section{The case of $n=2$}

In this section we give a concrete description of our results in the case of $n=2$, i.e. for the dual pair $(\PGl_2,\Sl_2)$. The interesting cases is when conjugacy classes have a non-connected stabilizer. For $\PGl_2$ this happens only for semisimple regular conjugacy classes. We will thus treat the case of semisimple monodromies only. 


\subsection{Mixed Poincaré polynomials of generic $\mathrm{GL}_2$-character varieties}

Fix $k \in \N$. For any $0 \leq r \leq k$ put $$
\mathbb{A}_r(z,w):=\begin{cases}\frac{(w^2+1)^{k-r}(1-w^2)^r}{(z^2-w^2)(1-w^4)}+\frac{(1-z^2)^r(z^2+1)^{k-r}}{(z^4-1)(z^2-w^2)}&\text{ if }0<r\leq k,\\
&\\
\frac{(w^2+1)^k}{(z^2-w^2)(1-w^4)}+\frac{(z^2+1)^k}{(z^4-1)(z^2-w^2)}-\frac{2^{k-1}}{(z^2-1)(1-w^2)}&\text{ if }r=0.
\end{cases}
$$

Say that a type $\omega \in \mathbb{T}_2$ is semisimple if $\omega=(1,(1))(1,(1))$ or $\omega=(2,(1))$. Notice that for any semisimple type $\omega$, we have that $\omega'=\omega.$ Say that a multitype $\bm \omega=(\omega_1,\dots,\omega_k)\in \mathbb{T}_2^k$ is semisimple if $\omega_i$ is semisimple for each $i$. A direct computation shows the following.

\begin{lemma}
\label{lemma-combinatorial-2}
For any semisimple multi-type $\bm \omega \in \mathbb{T}_2^k$, we have \begin{equation}
    \mathbb{H}_{\bm \omega}(z,w)=(-1)^r\mathbb{A}_r(z,w),
\end{equation}    
where $r=\#\{i \in \{1,\dots,k\} \ | \ \omega_i=(2,(1))\}$.
\end{lemma}

\bigskip

As a consequence of Lemma \ref{lemma-combinatorial-2}, for $n=2$, Conjecture \ref{mixed-polynomial-Gln} for $\Gl_2$-character varieties has the following expression. 
\begin{conjecture}
\label{conjn=2}
If $\bm C$ is a generic $k$-tuple of regular semisimple conjugacy classes of $\Gl_2$, we have $$
H_c(\mathcal{M}_{\bm C};q,t)=(qt^2)^{k-3}\mathbb{A}_0\left(-t\sqrt{q},\dfrac{1}{\sqrt{q}}\right).
$$

\end{conjecture}

\subsection{Local systems on $\mathrm{PGL}_2$-conjugacy classes}

For $x \in K^*\setminus{1}$, let $g_x$ be the matrix 
$$
g_x=\begin{pmatrix}1 &0 \\ 0 &x \end{pmatrix}.
$$ We denote by $C_x$ the conjugacy class of $g_x$ in $\Gl_2$ and by $\mathcal{C}_x$ the conjugacy class of $p_2(g_x)$
in $\PGl_2$. We have the following:

\bigskip

$\bullet$ If $x \neq -1$,  then $A(\mathcal{C}_{x})=\{1\}$, i.e. $p_2$ restricts to an isomorphism 
$$
p_2:C_x \to \mathcal{C}_x.
$$

$\bullet$ If $x=-1$, then $A(\mathcal{C}_{-1})=\bm\mu_2$. We denote by $\mathcal{L}_{\epsilon}$ the non-trivial $\PGl_2$-equivariant local system on $\mathcal{C}_{-1}$ associated with the non-trivial character $\epsilon$ of $\bm\mu_2$. We thus have, for the $2$-covering 
$$
p_2:C_{-1} \to \mathcal{C}_{-1},
$$
the decomposition
$$
(p_2)_*(\kappa) \cong \kappa \oplus \mathcal{L}_\epsilon.
$$

\subsection{Cohomology of $\mathrm{PGL}_2$-character stacks}

Fix a $k$-tuple $\bm{\mathcal{C}}$ of regular semisimple conjugacy classes of $\PGl_2$, a $k$-tuple $\bm C$ and $\lambda_{\bm C}$ as in \cref{section-geometry}. Notice that $\overline{\mathcal{C}_i}=\mathcal{C}_i$ and $\overline{C_i}=C_i$ for each $i=1,\dots,k$, since the classes are semisimple. Moreover, let $m$ be the number of degenerate conjugacy classes among $\mathcal{C}_1,\dots,\mathcal{C}_k$.  Put $$\bm{C}^+\coloneqq \bm{C}(1)= (C_1,\dots,C_k,\lambda_{\bm C}^{-1}I_2) $$ $$\bm{C}^- \coloneqq \bm{C}(-1)=(C_1,\dots,C_k,-\lambda_{\bm C}^{-1}I_2) .$$

Recall that $\bm{\mathcal{C}}$ is generic if and only if $\bm C^+,\bm C^-$ are both generic. Notice that, if $m \geq 1$, then $\bm C^+$ is generic if and only if $\bm C^-$ is generic. 

\subsection{Non-degenerate case}

Assume that $m=0$, i.e. that $\mathcal{C}_1,\dots,\mathcal{C}_k$ are all non-degenerate. With the notations of \cref{description-section}, the groups $A(\bm{\mathcal{C}}),H(\bm{\mathcal{C}}) $ and $ H'(\bm{\mathcal{C}})$ are all trivial. Proposition \ref{geometry-character-stacks} and Proposition \ref{proposition-DM} implies thus the following

\begin{prop}
The map $p:\mathcal{M}_{\bm C^+} \bigsqcup \mathcal{M}_{\bm C^-} \to \mathcal{M}_{\bm{\mathcal{C}}}$ is an isomorphism. In particular, if $\mathcal{C}$ is generic, the character stack $\mathcal{M}_{\mathcal{C}}$ is a smooth algebraic variety of dimension $2k-6$ with $2$ connected components.    
\end{prop}

In particular, in this case, Conjecture \ref{conj-nondegenerate} and Theorem \ref{theorem-nondegenerate} reads as follows.

\begin{teorema}
If $\bm{\mathcal{C}}$ is generic and non-degenerate, we have
$$E(\mathcal{M}_{\bm{\mathcal{C}}};q)=2q^{k-3}\mathbb{A}_0\left(\sqrt{q},\dfrac{1}{\sqrt{q}}\right). $$
\end{teorema}

\begin{conjecture}
If $\bm{\mathcal{C}}$ is generic and non-degenerate, we have
$$H_c(\mathcal{M}_{\bm{\mathcal{C}}};q,t)=2(qt^2)^{k-3}\mathbb{A}_0\left(-t\sqrt{q},\dfrac{1}{\sqrt{q}}\right). $$    
\end{conjecture}

\subsection{Degenerate case}

Assume now that $m \geq 1$. In what follows, put $H_m=\{(y_1,\dots,y_m) \in (\bm\mu_2)^m \ | \ y_1\cdots y_m=1\}$. With the notations of \cref{description-section}, we have that $$A(\bm{\mathcal{C}})=(\bm\mu_2)^m,  \hspace{12 pt} H(\bm{\mathcal{C}})=H_m, \hspace{12 pt} H'(\bm{\mathcal{C}})=\bm\mu_2 .$$

We thus have $\iota(\bm{\mathcal{C}})=1$. From Proposition \ref{proposition-DM} and Proposition \ref{geometry-character-stacks}, we have the following result.

\begin{prop}
The map $\overline{p}:\mathcal{M}_{\bm C^+} \to \mathcal{M}_{\bm{\mathcal{C}}}$ is an $H_m$-covering, i.e. it factorizes through an isomorphism $$\mathcal{M}_{\bm{\mathcal{C}}} \cong [\mathcal{M}_{\bm C^+}/H_m] .$$

In particular, if $\bm{\mathcal{C}}$ is generic, the stack $\mathcal{M}_{\bm{\mathcal{C}}}$ is a smooth and connected Deligne-Mumford stack of dimension $2k-6$. 
\end{prop}

\bigskip

For each subset $A \subseteq  \{1,\dots,m\}$, define $\chi_A \in \widehat{({\bm\mu}_2)^m}$ as 

$$\chi_A:({\bm\mu}_2)^m \to \C^\times$$ $$(y_1,\dots,y_m) \mapsto \prod_{j \in A}y_j .$$

Notice that $\widehat{({\bm\mu}_2)^{m}}=\{\chi_A\}_{A \subseteq \{1,\dots,m\}}$. For each $A \subseteq \{1,\dots,m\}$, denote by $\mathcal{E}_{A}$ the local system $\mathcal{E}^{\bm{\mathcal{C}}}_{\chi_{A}}$ on $\mathcal{M}_{\bm{\mathcal{C}}}$ (see \S \ref{localsystemsPGLncohomology}).

\bigskip

For any $m_1,m_2,l \in \N$, denote by $C_{m_1,m_2,l}$ the coefficient of $y^lx^{m_1+m_2-l}$ in the product $(x-y)^{m_1}(x+y)^{m_2}$. 
\bigskip

Theorem \ref{E-polynomial-general} and Conjecture \ref{IH-conjecture-thm} read as follows.

\begin{teorema}
\label{Epolynomialgenericdegenerate}
For any $A \subseteq \{1,\dots,m\}$, we have  \begin{equation}
\label{Eseries2}
E(\mathcal{M}_{\bm{\mathcal{C}}},\mathcal{E}_A;q)=\frac{q^{k-3}}{2^{m-1}}\sum_{\substack{l=0\\ \text{ even}}}^{m}C_{|A|,m-|A|,l}\,\,\mathbb{A}_l\left(\frac{1}{\sqrt{q}},\sqrt{q}\right).
\end{equation}

\end{teorema}

\bigskip

\begin{conjecture}
\label{conjmhslocalsystems}
For any $A \subseteq \{1,\dots,m\}$, we have   
\begin{equation}
\label{HCseries2}
H_c(\mathcal{M}_{\bm{\mathcal{C}}},\mathcal{E}_A;q,t)=\frac{(qt^2)^{k-3}}{2^{m-1}}\sum_{\substack{l=0\\l \text{ even}}}^{m}C_{|A|,m-|A|,r}\,\,\mathbb{A}_l\left(\frac{1}{\sqrt{q}},-t\sqrt{q}\right).
\end{equation}
In particular, the Poincar\'e polynomial of the pure part is given by

\begin{equation}
    PH_c(\mathcal{M}_{\bm{\mathcal{C}}},\mathcal{E}_A;q)=\frac{(qt^2)^{k-3}}{2^{m-1}}\sum_{\substack{l=0\\l \text{ even}}}^{m}C_{|A|,m-|A|,l}\,\,\mathbb{A}_l\left(0,\sqrt{q}\right).
\label{pure-conj}\end{equation}
\end{conjecture}

Let us explain how to get these formulas from Formulas (\ref{edIE-theorem}) and (\ref{IHconjecture}). 

Put

$$
L_r:=\{i \in \{1,\dots,k\}\ | \ r_i=2\},\hspace{1cm}l_r:=|L_r|.
$$
For any $r \in R_{d_1,\dots,d_k}$, we get from Lemma \ref{lemma-combinatorial-2} the following formula 
    $$\mathbb{H}_{\bm \omega_r}(z,w)=(-1)^{l_r}\mathbb{A}_{l_r}(z,w).$$

Moreover, we can rewrite $\Delta^s_r$ for any $r \in R_{d_1,\dots,d_k}$ and $s=(s_1,\dots,s_k) \in \N_{>0}^k$ as follows. Since each $r_i$ is either $1$ or $2$, for each $i$ and $j$, we have $$\phi(r_i)=\phi\left(\dfrac{r_i}{\gcd(r_i,s_i+j)}\right)=1 .$$ In particular, \begin{equation}
\label{eq1-n2}
    \Delta^s_r=\prod_{i=1}^k C_{\frac{r_i}{\gcd(r_i,s_i)}}+ \prod_{i=1}^k C_{\frac{r_i}{\gcd(r_i,s_i+1)}}.
\end{equation}

Notice that $C_2=-1$ and $C_1=1$, see Formula (\ref{cyclo}). Therefore, if $r_i=1$, we have $C_{\frac{r_i}{\gcd(r_i,s_i)}}=C_{\frac{r_i}{\gcd(r_i,s_i+1)}}=1$, and, if $r_i=2$, we have $C_{\frac{r_i}{\gcd(r_i,s_i)}}=\pm 1$ and $C_{\frac{r_i}{\gcd(r_i,s_i)}}=-C_{\frac{r_i}{\gcd(r_i,s_i+1)}}$. From Formula (\ref{eq1-n2}) we thus have 
\begin{equation}
    \Delta^s_r=(1+(-1)^{l_r})\prod_{i \in L_r}C_{\frac{r_i}{\gcd(r_i,s_i)}}.
\end{equation}

In particular, $\Delta^s_r=0$ if $l_r$ is odd. If $l_r$ is even and $s=s_{\chi_A}$ for some $A \subseteq \{1,\dots,m\}$, we have \begin{equation}
  \Delta^{s_{\chi_A}}_r=2 (-1)^{|A \cap L_r|}.  
\end{equation} 

We have thus \begin{equation}
    \sum_{r \in R_{d_1,\dots,d_k}}\Delta^{s_{\chi_A}}_r\mathbb{H}_{\bm \omega_r}(z,w)=2\sum_{\substack{r \in R_{d_1,\dots,d_k}\\ l_r \text{ even }}}\mathbb{A}_{l_r}(z,w)(-1)^{|A \cap L_r|}=
\end{equation}

\begin{equation}
=2\sum_{\substack{l=0 \\ \text{even}}}^m \mathbb{A}_l(z,w)\sum_{\substack{r \in R_{d_1,\dots,d_k}\\l_r=l}}(-1)^{|A \cap L_r|}= 2\sum_{\substack{l=0 \\ \text{even}}}^m\mathbb{A}_l(z,w)C_{|A|,m-|A|,l}.    
\end{equation}

\subsection{Langlands duality and multiplicities for $\mathrm{SL}_2(\mathbb{F}_q)$}

In this section $K=\overline{\F}_q$ and $2$ does not divide $q$. We fix an embedding $\overline{\F}_q^\times \subseteq \overline{\Q}_{\ell}^\times$ and a generator $\zeta_{q-1}$ of $\F_q^\times$.  Recall that $T \subseteq \Gl_2$ is the torus of diagonal matrices, $T'= T \cap \Sl_2 \subseteq \Sl_2$ and $\overline{T}=p_2(T)$. The Weyl group $W$ of the three maximal tori is $W=\bm \mu_2=\{1,-1\}$. Denote by $\pi':T' \to [T'/W]$ and $\overline{\pi}:\overline{T} \to [\overline{T}/W]$ the projection maps.
\bigskip

We now describe the explain the correspondence (\ref{langlandsfinite}) in the semisimple case.

\bigskip

Notice that we have  isomorphisms $\gamma:T'\to \mathbb{G}_m$ and  $\overline{\gamma}:\overline{T} \to \mathbb{G}_m$ defined as

\begin{equation}
\label{isomorphismT'T}
\gamma\begin{pmatrix}
    x &0\\
    0 &x^{-1}
\end{pmatrix}=x,      \hspace{14 pt} \text{ and }\hspace{0.5cm}\overline{\gamma}\left(p_2\begin{pmatrix}
    1 &0\\
    0 &x
\end{pmatrix}\right)=x .\end{equation} 

From which we identify $\widehat{(T')^F}=\Hom(\F_q^\times,\overline{\Q}^\times_{\ell})$ and $\overline{T}{^F}=\F_q^\times$. The isomorphism (\ref{bijection-LLD}) reads

\begin{equation}
\Psi^{-1}: \Hom(\F_q^\times,\overline{\Q}^\times_{\ell}) \to \F^\times_q,\hspace{1cm}\alpha \rightarrow \alpha(\zeta_{q-1}).
\label{iso2}\end{equation}

Notice that, through the isomorphisms $\gamma$ and $\overline{\gamma}$, the action of $W=\bm \mu_2$ on $T'$ and on $\overline{T}$ is identified with the action of $\bm \mu_2$ on $\mathbb{G}_m$ given by $$ (-1) \cdot x=x^{-1} .$$

\bigskip

 The elements of $({\rm LS}_o(\PGl_2)^F)_{\text{split}}$ are the pairs $(\mathcal{C},\zeta)$ where $\mathcal{C}$ is a conjugacy class of $\PGl_2$ and $\zeta$ an irreducible $\PGl_2$-equivariant local system on $\mathcal{C}$. The only conjugacy class of $\PGl_2$ which supports a non-trivial local system is the degenerate class $\mathcal{C}_{-1}$. 
 \bigskip

Consider the inductions

$$
{\rm I}':\Perv([T'/W])\rightarrow\Perv([\Sl_2/\Sl_2]),\hspace{1cm}\overline{{\rm I}}:\Perv([\overline{T}/W])\rightarrow\Perv([\PGl_2/\PGl_2]), 
$$
and $\overline{\rm Ind}:\Perv(\overline{T})\rightarrow\Perv([\PGl_2/\PGl_2])$ defined in \S \ref{geometric-induction}.
\bigskip

Fix a regular element $x\in \F_q^\times\backslash\{1\}\subset\overline{T}{^F}$.
 \bigskip
 
 Notice that $\overline{{\rm I}}(\overline{\pi}_*((\overline{\Q}_{\ell})_{x}))=\overline{\Ind}((\overline{\Q}_{\ell})_{x})$ and that the support of $\overline{\Ind}((\overline{\Q}_{\ell})_{x})$ (viewed as a $\PGl_2$-equivariant complex on $\PGl_2$) is the conjugacy class $\mathcal{C}_{x}$ (as it is  semisimple regular).

 Consider the cartesian diagram:
 
 \begin{equation}
\xymatrix{[\overline{B}/\overline{B}]\ar[rr]&&[\PGl_2/\PGl_2]\\
B(\overline{T})=B(C_{\PGl_2}(x)^o)\ar[u]\ar[rr]&&B(C_{\PGl_2}(x))\ar[u]}.
\end{equation}
The bottom arrow is in fact the quotient of $p_2:C_x\rightarrow\mathcal{C}_x$ by $\PGl_2$.

 The complex $\overline{{\rm I}}(\overline{\pi}_*((\overline{\Q}_{\ell})_x))$ viewed as a $\PGl_2$-equivariant complex on $\PGl_2$ is the local system $p_2(\overline{\Q}_\ell[{\rm dim}\, C_x])$ on $\mathcal{C}_x$ extended by $0$ on $\PGl_2$. Therefore
$$
\overline{{\rm I}}(\overline{\pi}_*((\overline{\Q}_{\ell})_x))=\begin{cases}\overline{\Q}_\ell[{\rm dim}\, \mathcal{C}_{-1}]\oplus\mathcal{L}_\epsilon[{\rm dim}\,\mathcal{C}_{-1}]&\text{ if }x=-1,\\
\overline{\Q}_\ell[{\rm dim}\, \mathcal{C}_x]&\text{ if }x\neq -1.\end{cases}
$$
The element $x$ corresponds, under $\Psi^{-1}$, to linear character $\alpha_x$ of $\F_q^\times\simeq T'{^F}$ and so corresponds to an $F$-stable Kummer local system $\mathcal{A}_x$ on $T'$. The local system $\mathcal{A}_{-1}$ is the non-trivial square of the trivial local system on $\mathbb{G}_m$.
 \bigskip

 As $(-1)^*(\mathcal{A}_x)\neq\mathcal{A}_x$ if $x\neq -1$, the complex ${\rm I}'(\mathcal{A}_x[{\rm dim}\, T'])$ is an irreducible perverse sheaf on $\Sl_2$.
\bigskip

As $\mathcal{A}_{-1}$ is $W$-equivariant, the perverse sheaf ${\rm I}'(\mathcal{A}_{-1}[{\rm dim}\, T'])$ decomposes as a sum of two irreducible simple perverse sheaves on $\Sl_2$
$$
{\rm I}'(\mathcal{A}_{-1}[{\rm dim}\, T'])=\mathcal{K}_1\oplus\mathcal{K}_\epsilon
$$
parametrized by the irreducible characters of $W$.
\bigskip

Therefore the correspondence (\ref{langlandsfinite}) maps the pair $(\mathcal{C}_x,\overline{\Q}_\ell)$ to ${\rm I}'(\mathcal{A}_x[{\rm dim}\, T'])$ if $x\neq -1$ and maps $(\mathcal{C}_{-1},\mathcal{L})$ to $\mathcal{K}_1$ if $\mathcal{L}=\overline{\Q}_\ell$ and to $\mathcal{K}_\epsilon$ if $\mathcal{L}=\mathcal{L}_\epsilon$.
\bigskip

In other words, 

$$
\mathcal{X}_{\mathcal{C}_{-1},1}^{\Sl_2}=\mathcal{K}_1, \hspace{1cm}\mathcal{X}_{\mathcal{C}_{-1},\epsilon}^{\Sl_2}=\mathcal{K}_\epsilon
$$
with the definition of $\mathcal{X}_{\mathcal{C},\chi}^{\Sl_2}$ given before Theorem \ref{teoremadual}.
\bigskip

Although they are not considered in this section (as we consider only semisimple regular conjugacy classes), the trivial pair $(\{1\},\overline{\Q}_\ell)$ is mapped to the constant perverse sheaf $\overline{\Q}_\ell[{\rm dim}\, \Sl_2]$ on $\Sl_2$ and maps the trivial local system on the regular unipotent conjugacy class of $\PGl_2$ to the Steinberg character-sheaf on $\Sl_2$.
\bigskip

\begin{remark}The characteristic functions of the character-sheaves $\mathcal{X}^{\Sl_2}_{\mathcal{C},\chi}$ are all (up to an explicit sign) irreducible characters of $\Sl_2(\F_q)$ except for the two pairs $(\mathcal{C}_{-1},\overline{\Q}_\ell)$ and $(\mathcal{C}_{-1},\mathcal{L}_\epsilon)$ for which we give the values below.
\end{remark}

\bigskip

Let $X_{\rm Id}={\bf X}_{\mathcal{K}_1}$ and $X_{\epsilon}={\bf X}_{\mathcal{K}_{\epsilon}}$ be the characteristic functions of $\mathcal{K}_1$ and $\mathcal{K}_\epsilon$. We use Proposition \ref{decomp} to compute their values on the conjugacy classes of $\Sl_2(\F_q)$. For notations and details concerning conjugacy classes of $\Sl_2(\F_q)$, we refer to \cite[Chapter 12.5]{DM}.

Let $\bm \mu_{q+1}=\{x \in \F_{q^2}^\times \ | \ x^{q+1}=1\}$.  Let $\alpha_{-1}:\F_q^\times\rightarrow\overline{\Q}_\ell^\times$ (resp. $\omega_{-1}:\bm \mu_{q+1}\rightarrow \overline{\Q}_\ell^\times$) be the characteristic function of $\mathcal{A}_{-1}$ with respect to the canonical $F$-equivariant structure (resp. with respect to the $F$-equivariant structure twisted by the non-trivial element of $W$). It takes the value $1$ at squares and the value $-1$ at non-squares elements of $\F_q^\times$  (resp. of $\bm \mu_{q+1}$).

\begin{scriptsize}
\begin{equation}
\label{tableSL}
\begin{array}{|c|c|c|c|c|c|c|}
\hline
&&&&&&\\
&\left(\begin{array}{cc}1&0\\0&1\end{array}\right)&\left(\begin{array}{cc}-1&0\\0&-1\end{array}\right)&\left(\begin{array}{cc}a&0\\0&a^{-1}\end{array}\right)&  \left(\begin{array}{cc}x&0\\0& x^q\end{array}\right)&\left(\begin{array}{cc}1&1\\0&1\end{array}\right)&\left(\begin{array}{cc}-1&1\\0&-1\end{array}\right)\\
&&&a\neq a^{-1}\in\F_q^\times&x^{q+1}=1, x\neq x^q&&\\
\hline
&&&&&&\\
X_{{\rm Id}}&1&q\alpha_{-1}(-1)&\alpha_{-1}(a)&0&1&0 \\
&&&&&&\\
\hline
&&&&&&\\
X_\epsilon&q&\alpha_{-1}(-1)&0&\omega_{-1}(x)&1&\alpha_{-1}(-1)\\
&&&&&&\\
\hline
\end{array}
\end{equation}
\end{scriptsize}

\bigskip

From Table (\ref{tableSL}) above and the description of the map $\mathfrak{c}_{\PGl_2}$ given above, we can check by direct computation the following result. 
\begin{teorema}
    \label{theoremmultiplicitiessl2}
Let $\bm{\mathcal{C}}=(\mathcal{C}_1,\dots,\mathcal{C}_k)$ be a generic $k$-tuple of regular semisimple conjugacy classes of $\PGl_2$. Let  $A \subseteq \{1,\dots,m\}$ and denote by  $\chi_A=((\chi_A)_1,\dots,(\chi_A)_k) \in \widehat{A(\bm{\mathcal{C}})}$ the corresponding irreducible character. We have
\begin{equation}
\left\langle {\bf X}_{\mathcal{X}_{\mathcal{C}_1,(\chi_A)_1}^{\Sl_2}}\cdots{\bf X}_{\mathcal{X}_{\mathcal{C}_k,(\chi_A)_k}^{\Sl_2}},1\right\rangle_{\Sl_2}=\frac{1}{2^{m-1}}\sum_{\substack{l=0\\l \text{ even}}}^{m}C_{|A|,m-|A|,l}\,\,\mathbb{A}_l\left(0,\sqrt{q}\right).     
\end{equation}
\end{teorema}

\end{document}